\documentclass[12pt]{amsart}

\usepackage{amssymb}
\usepackage[scr=boondox]{mathalpha}
\usepackage{xypic}
\usepackage{enumerate}
\usepackage{color}
\usepackage{graphicx}
\usepackage[text={6in,9in},centering]{geometry}
\usepackage{hyperref}
\usepackage{relsize}
\usepackage[english]{babel}
\usepackage{booktabs}
\usepackage{longtable}
\usepackage{makecell}
\usepackage{multirow}
\usepackage{caption}

\hypersetup{
    colorlinks = true,
    linkcolor = blue,
    anchorcolor = blue,
    citecolor = blue
}

\numberwithin{equation}{section}

\theoremstyle{plain}
  \newtheorem{thm}{Theorem}[section]
  \newtheorem*{thm*}{Theorem}
  \newtheorem{cor}[thm]{Corollary}
  \newtheorem{lem}[thm]{Lemma}
  \newtheorem{prop}[thm]{Proposition}
  
  \newtheorem*{ques}{Question}
\theoremstyle{definition}  
  \newtheorem{defn}{Definition} [section]
  \newtheorem{exmp}{Example} [section]
  
\theoremstyle{remark}
  \newtheorem{rmk}{Remark}  [section]

\newcommand{\id}[0]{\mathrm{id}}
\newcommand{\e}[0]{\mathrm{e}}

\newcommand{\dist}[0]{\mathrm{dist}}
\newcommand{\diam}[0]{\mathrm{diam}}
\newcommand{\dif}[0]{\mathrm{d}}

\newcommand{\intr}[0]{\operatorname{int}}

\newcommand{\dimh}[0]{\dim_{\mathrm{H}}}
\newcommand{\dimp}[0]{\dim_{\mathrm{P}}}
\newcommand{\dimt}[0]{\dim_{\mathrm{T}}}
\newcommand{\vect}[1]{\boldsymbol{#1}}
\newcommand{\mat}[1]{\boldsymbol{#1}}
\newcommand{\numcov}{\#^{\mathrm{cov}}}

\newcommand{\entlow}{\underline{h}}
\newcommand{\entup}{\overline{h}}
\newcommand{\enttop}{h_{\mathrm{top}}}

\newcommand{\enttopbow}{\enttop^{\mathrm{B}}}
\newcommand{\enttoppac}{\enttop^{\mathrm{P}}}
\newcommand{\enttoplow}{\entlow_{\mathrm{top}}}
\newcommand{\enttopup}{\entup_{\mathrm{top}}}

\newcommand{\prelow}{\underline{P}}
\newcommand{\preup}{\overline{P}}
\newcommand{\qrelow}{\underline{Q}}
\newcommand{\qreup}{\overline{Q}}
\newcommand{\preL}{P^{\mathrm{L}}}
\newcommand{\preU}{P^{\mathrm{U}}}
\newcommand{\prebpp}{Q^{\mathrm{B}}}
\newcommand{\prebow}{P^{\mathrm{B}}}

\newcommand{\prepac}{P^{\mathrm{P}}}

\newcommand{\carpeslow}{\underline{\mathscr{Q}}}

\newcommand{\carpepup}{\overline{\mathscr{K}}}
\newcommand{\msrbpp}{\mathscr{M}}
\newcommand{\msrbpplow}{\underline{\msrbpp}}
\newcommand{\msrbppup}{\overline{\msrbpp}}
\newcommand{\msrppp}{\mathscr{N}}
\newcommand{\msrbow}{\mathscr{R}}
\newcommand{\msrpac}{\mathscr{P}}

\newcommand{\eqlbrbow}{M^{\mathrm{B}}}
\newcommand{\eqlbrpac}{M^{\mathrm{P}}}

\newcommand{\bu}{{\bf u}} 
\newcommand{\bv}{{\bf v}} 
\newcommand{\bw}{{\bf w}} 

\newcommand{\R}{\mathbb{R}}

\newcommand{\Z}{\mathbb{Z}}
\newcommand{\N}{\mathbb{N}}

\newcommand{\E}{\mathbb{E}}
\newcommand{\var}{\mathbb{D}}

\newcommand{\red}[1]{{\color{red}#1}}
\newcommand{\blue}[1]{{\color{blue}#1}}

\newcommand{\len}[1]{\lvert #1 \rvert}

\begin{document}

\title[Nonautonomous Dynamical Systems]{Nonautonomous Dynamical Systems \uppercase\expandafter{\romannumeral3}: \\
Symbolic and Expansive Systems}

\author{Zhuo Chen}
\address{School of Mathematical Sciences, East China Normal University, No. 500, Dongchuan Road, Shanghai 200241, P. R. China}
\email{10211510056@stu.ecnu.edu.cn, charlesc-hen@outlook.com}

\author{Jun Jie Miao}
\address{School of Mathematical Sciences,  Key Laboratory of MEA(Ministry of Education) \& Shanghai Key Laboratory of PMMP,  East China Normal University, Shanghai 200241, China}

\email{jjmiao@math.ecnu.edu.cn}

\subjclass[2020]{37D35, 37B55, 37B10}


\maketitle
\begin{abstract}
  A nonautonomous dynamical system $(\boldsymbol{X},\boldsymbol{T})=\{(X_{k},T_{k})\}_{k=0}^{\infty}$ is a sequence of continuous mappings $T_{k}:X_{k} \to X_{k+1}$ along with a sequence of compact metric spaces $X_{k}$.
  In this paper, we study the nonautonomous symbolic  dynamical systems   and nonautonomous expansive dynamical systems.

We first study the  homogeneous properties of pressures in nonautonomous symbolic systems $(\boldsymbol{\Sigma}(\boldsymbol{m}),\boldsymbol{\sigma})$, and  we  simplify the formulae of Bowen, packing, lower and  upper  topological     pressures   for potentials $\boldsymbol{f}=\{f_{k} \in C(\Sigma_{k}^{\infty}(\boldsymbol{m}),\mathbb{R})\}_{k=0}^{\infty}$ with strongly bounded variation.   Then we   apply a law of large numbers  to obtain the formulae for the lower and upper measure-theoretic pressures with respect to  nonautonomous   Bernoulli measures and obtain  Bowen equilibrium states and packing   equilibrium states for potentials in   nonautonomous symbolic systems.
  Finally, we study the generators in  nonautonomous  expansive   systems $(\boldsymbol{X},\boldsymbol{T})$, and we obtain that $(\boldsymbol{X},\boldsymbol{T})$ is expansive if and only if it has a generator. Moreover,  strongly uniformly expansive $(\boldsymbol{X},\boldsymbol{T})$ is equisemiconjugate to a subsystem of the nonautonomous symbolic dynamical system.

\end{abstract}


\section{Introduction}
\subsection{Nonautonomous dynamics and pressures}
Let $(\vect{X},\vect{d})=\{(X_{k},d_{k})\}_{k=0}^{\infty}$ be a sequence of compact metric spaces $(X_{k},d_{k})$ and $\vect{T}=\{T_{k}\}_{k=0}^{\infty}$ a sequence of continuous mappings $T_{k}:X_{k} \to X_{k+1}$.
We call the pair $(\vect{X},\vect{T})$ a \emph{nonautonomous dynamical system (NDS)}. We sometimes write the triplet $(\vect{X},\vect{d},\vect{T})$ to emphasize the dependence on the metrics $\vect{d}$.
If $(\vect{X},\vect{d})$ is a constant sequence, i.e., $(X_{k},d_{k})=(X,d)$ for all $k \in \N$ where $(X,d)$ is a compact metric space, then we say that $(\vect{X},\vect{T})$ \emph{is a nonautonomous dynamical system with an identical space} and denote it by $(X,\vect{T})$. See \cite{Kloeden&Rasmussen2011} for an introduction to the theory of NDSs $(X,\vect{T})$ with an identical space,
Moreover, if $\vect{T}$ is also a constant sequence, i.e., $T_{k}=T$ for all $k \in \N$ where $T:X \to X$ is a continuous mapping, then the NDS $(\vect{X},\vect{T})$ degenerates into the (autonomous) \emph{topological dynamical system (TDS)} $(X,T)$.
Therefore, NDSs may be considered as a generalization of TDSs.

The classic theories of TDSs $(X,T)$ concern aspects of topological dynamics, ergodic theory, and thermodynamic formalism; see \cite{Bowen2008,Pesin1997,Walters1982}. Moreover,  these theories are also highly related to fractal geometry, especially dimension theory, see \cite{Falconer2014} for details.  Given a TDS $(X,T)$. We write $M(X,T)$ for the set of all $T$-invariant measures on $X$.
In the late 1950s, Kolmogorov \cite{Kolmogorov1958} and Sinai \cite{Sinai1959} introduced the measure-theoretic entropy $h_{\mu}(T)$ of $T$ with respect to $\mu \in M(X,T)$. In 1975,  Ruelle \cite{Ruelle1972} and Walters \cite{Walters1975} introduced the topological pressure $P(T,f)$ of a TDS $(X,T)$ for a potential $f \in C(X,\R)$ which extends  the earlier topological  entropies in \cite{AKM1965,Bowen1971}. From then on, topological pressures became the center  of thermodynamic formalism.
One outstanding result is the   classic variational principle by Walters  \cite{Walters1975} which  states that
given a TDS $(X,T)$ and $f \in C(X,\R)$,
$$
P(T,f)=\sup\Bigl\{h_{\mu}(T) + \int_{X}f\dif{\mu}: \mu \in M(X,T)\Bigr\}.
$$
See \cite{Goodman1971,Misiurewicz1976,Ollagnier&Pinchon1982,Ruelle1972} for  various relevant studies on variational principles in TDSs.

The  expansive dynamics was first introduced for homeomorphisms in \cite{Utz1950} and generalized to positively expansiveness in \cite{Eisenberg1966}.
 It  plays  important roles in many fields of Mathematics.
For example, it is well known that the shifts of sequences on finite symbols are expansive, and so are hyperbolic systems restricted to their hyperbolic sets.
In particular, Anosov diffeomorphisms are expansive.
For expansive TDSs, there are generators which simplify the formulations and calculations for pressures and entropies; see \cite{Bowen1972,Burguet2020, Keynes&Robertson1969,Mane1979,Morales2012,Reddy1970,Walters1975} for various related works and generalizations on expansive dynamics.

In the classic theory of TDSs, people  mainly focus on invariant subsets which may be regarded as `dynamically regular' sets, but  `dynamically irregular' sets are one of the research objects in NDSs. Especially,  topological pressures and entropies are the main tools to study the dimensions of such sets, see \cite{Gu&Miao2022,GHM2025,GM,Wen00}. Therefore, it is natural to develop the theory of topological pressures and entropies on NDSs.

Kolyada and Snoha \cite{Kolyada&Snoha1996} and Huang et al \cite{Huang&Wen&Zeng2008} introduced topological entropy and pressure in NDSs with identical spaces.
In \cite{Kawan2015},  Kawan obtained  partial result on the variational principles in NDSs $(\vect{X},\vect{T})$   for the topological pressure, and he also discussed various generalizations of  expansiveness in NDSs $(\vect{X},\vect{T})$.
In particular, he obtained the existence of generators for the upper topological entropy (see Subsection \ref{ssect:PLPU} for its definition) in strongly uniformly expansive NDSs (see Definition~\ref{def:sue}).
\begin{thm*}[{\cite[Prop.7.12(\romannumeral4)]{Kawan2015}}]
  Assume that $(\vect{X},\vect{T})$ is strongly uniformly expansive with expansive constant $\delta>0$.
  Then there exists a sequence $\vect{\mathscr{U}}=\{\mathscr{U}_{k}\}_{k=0}^{\infty}$ of open covers $\mathscr{U}_{k}$ of $X_{k}$ with a Lebesgue number 
such that 
  $$
  \enttopup(\vect{T},X_{0}) = \varlimsup_{n \to \infty}\frac{1}{n}\log{\numcov\Big(\bigvee_{j=1}^{n-1}\vect{T}^{-j}\mathscr{U}_{j}\Big)},
  $$
  where $\numcov(\mathscr{A})$ denotes the minimal cardinality of subcovers of $\mathscr{A}$.
\end{thm*}

Over the years, techniques dealing with geometric irregularities in fractal geometry  have been introduced to handle dynamical irregularities.
In \cite{Bowen1973},  Bowen defined a new type of topological entropy on subsets $Z \subseteq X$ similar to the Hausdorff dimension  of $Z$, and  Pesin and Pitskel' \cite{Pesin&Pitskel1984} generalized the idea to define what we call the \emph{Bowen-Pesin-Pitskel' pressure} $\prebow(T,f,Z)$ of $T$ for $f$ on $Z$.
Pesin also introduced the  lower and upper capacity topological pressures    which correspond to the lower and upper box dimensions of fractal sets using the Carath\'{e}odory dimensional structure \cite{Pesin1997}.
  Feng and Huang \cite{Feng&Huang2012} formulated a type of entropies similar to the packing dimensions of fractal sets,
and Zhong and Chen \cite{Zhong&Chen2023} extended it to the \emph{packing pressures $\prepac(T,f,Z)$}.
These pressures coincide on compact invariant subsets (essentially autonomous subsystems).
However,  they may be distinct if the sets are not compact or  invariant under $T$.
Moreover, Pesin and Pitskel' \cite{Pesin&Pitskel1984} showed 
$$
\prebow(T,f,Z) \leq \sup\Bigl\{h_{\mu}(T) + \int_{Z}f\dif{\mu}: \mu \in M(X,T)\Bigr\},
$$
 where the inequality may  strictly hold. 

We are interested in  the connection between   Nonautonomous dynamical systems and fractal geometry, in particular, the relation of various topological pressures and the dimension theory of Nonautonomous fractals. In \cite{CM1}, we   systematically studied the properties of Bowen, packing, lower and upper topological pressures and compared them with Hausdorff, packing, lower and upper box dimensions, and in \cite{CM2}, we obtained Billingsley type theorems and variational principles for Bowen-Pesin-Pitskel' and packing pressures in Nonautonomous dynamical systems. The following conclusion is proved in \cite{CM2}.
\begin{thm}\label{thm_vpBP}
  Given  $(\vect{X},\vect{T})$ and a compact $K \subseteq X_{0}$, for all equicontinuous $\vect{f}=\{f_{k}:X_{k} \to \R\}_{k=0}^{\infty}$,
  $$
  \prebow(\vect{T},\vect{f},K) = \sup\{\prelow_{\mu}(\vect{T},\vect{f}): \mu \in M(X_{0}), \mu(K)=1\},
  $$
  and for all equicontinuous $\vect{f}$ satisfying $\|\vect{f}\|<+\infty$ and $\prepac(\vect{T},\vect{f},K)>\|\vect{f}\|$,
  $$
  \prepac(\vect{T},\vect{f},K) = \sup\{\preup_{\mu}(\vect{T},\vect{f}): \mu \in M(X_{0}), \mu(K)=1\}.
  $$
where $\prelow_\mu$ and $\preup_\mu$ denote the lower and upper  measure-theoretic pressures, respectively.
\end{thm}

In this paper, we study the properties of topological pressures and entropies on nonautonomous expansive dynamical systems and  nonautonomous symbolic dynamical systems. By applying Theorem \ref{thm_vpBP} to nonautonomous symbolic dynamical systems, we explore the  Bowen equilibrium state and Packing   equilibrium state for potentials, and we also investigate the generators for nonautonomous expansive dynamical systems.

\subsection{Nonautonomous symbolic dynamical systems}\label{ssect:symdyn}
The autonomous symbolic dynamical systems is the study of dynamical systems defined on discrete spaces, focusing on sequences of symbols and their evolution through shifts. It provides a powerful framework for analyzing and modeling the behavior of dynamical systems through this discrete lens, making it a vital area in mathematics and its applications; see \cite{Lind&Marcus2021,McGoff2012,McGoff2019, Sinai1968}. Particularly, symbolic dynamics are closely related to iterated function systems in fractal geometry; see \cite{Barreira1996,Cao&Feng&Huang2008,DasSim17,Falconer1988a,Falconer1988b,Feng23,FH08, Feng&Kaenmaki2011} for various related studies.

Nonautonomous  symbolic dynamical systems  are  strongly connected to nonautonomous iterated function systems and nonautonomous fractals. Pressures  and entropies which are essentially defined on the corresponding symbolic systems play an important role in studying the dimensions of such fractals,   \cite{Gu&Miao2022,GHM2025,GM, Hua1994,R-G&Urbanski2012,Wen00}.


 Given a sequence $\vect{m}=\{m_{k}\}_{k=0}^{\infty}$ of positive integers $m_{k} \geq 2$ for every $k \in \mathbb{N}$, let
  \begin{equation}\label{eq:symspacelvk}
    \Sigma_{k}^{\infty}(\vect{m})=\{\omega=\omega_{k}\omega_{k+1}\ldots: \omega_{j} \in \{1,\ldots,m_j\}, j \geq k\} 
  \end{equation}
  be the \emph{sequence space of level $k$}, and for each $l\geq k$, we write
  \begin{equation}
    \Sigma_{k}^{l}(\vect{m})=\{\mathbf{u}=u_{k} \ldots u_{l}: u_{j} \in \{1,\ldots,m_j\}, k \leq j \leq l\},
  \end{equation}
and
  $$
  \Sigma_{k}^{*}(\vect{m})=\bigcup_{l=k}^{\infty}\Sigma_{k}^{l}(\vect{m}).
  $$
For $\mathbf{u} \in \Sigma_{k}^{l}(\vect{m})$, we write $\len{\mathbf{u}}=l-k+1$ for the \emph{length} of $\mathbf{u}$.   Given an integer $n \geq 1$, for  $\omega \in \Sigma_{k}^{\infty}(\vect{m})$, we write $\omega\vert{n}=\omega_{k}\ldots\omega_{k+n-1}$ for the \textit{$n$-th curtailment of $\omega$},
  and it is clear that $\omega\vert{n} \in \Sigma_{k}^{k+n-1}(\vect{m})$;
    we also write $\mathbf{u}\vert{n}=u_{k} \ldots u_{k+n-1} \in \Sigma_{k}^{k+n-1}(\vect{m})$ for $n \leq \len{\mathbf{u}}$.
    Given $\mathbf{u} \in \Sigma_{k}^{l}(\vect{m})$ and $\mathbf{v}=v_{l+1}v_{l+2} \ldots v_{N} \in \Sigma_{l+1}^{N}(\vect{m})$, we write $\mathbf{u}\mathbf{v}=u_{k}u_{k+1} \ldots u_{l}v_{l+1}v_{l+2} \ldots v_{N} \in \Sigma_{k}^{N}(\vect{m})$. 
Given $\mathbf{u} \in \Sigma_{k}^{*}(\vect{m})$ and $\omega \in \Sigma_{k}^{\infty}(\vect{m})$, we write $\bu \preceq\omega$ if $\bu$ is a curtailment of $\omega$, and we call the set $[\bu]_k =\{\omega \in \Sigma_{k}^{\infty}(\vect{m}): \bu \preceq \omega \}$ the \textit{cylinder} of $\mathbf{u}$, where $\bu\in \Sigma_k^*(\vect{m})$. If $\mathbf{u}=\emptyset$, its cylinder is $[\bu]_k= \Sigma_{k}^{\infty}(\vect{m})$.   The \emph{rank} of the cylinder $[\mathbf{u}]_{k}$ refers to $\len{\mathbf{u}}$.
The cylinders  $[\bu]_k =\{\omega \in \Sigma_{k}^{\infty}(\vect{m}): \bu \preceq \omega \}$  for  $\Sigma_k^\infty(\vect{m})$ form a base
of open and closed neighbourhoods for $ \Sigma_k^\infty(\vect{m})$.
We call a set of finite words $A \subset \Sigma_k^* (\vect{m})$ a \textit{covering set} for $ \Sigma_k^\infty(\vect{m})$ if $\Sigma_k^\infty \subset \bigcup_{\mathbf{u}\in A}[\bu]_k$.

For $\omega,\vartheta \in \Sigma_k^\infty(\vect{m})$, let $\omega\wedge \vartheta\in \Sigma_k^*(\vect{m})$ denote the maximal common initial finite word of both $\omega$ and $\vartheta$. We topologise $\Sigma_k^\infty(\vect{m})$  using the metric
$d_k(\omega,\vartheta ) = \e^{-|\omega\wedge\vartheta  |}$ for distinct $\omega,\vartheta  \in\Sigma_k^\infty(\vect{m})$ to make $\Sigma_k^\infty(\vect{m})$  a compact metric space. The open and closed balls with center $\omega \in \Sigma_{k}^{\infty}(\vect{m})$ and  radius $\varepsilon$ are 
  $$
  B_{d_{k}}(\omega,\varepsilon)=[\omega\vert{\lfloor-\log{\varepsilon+1}\rfloor}]_{k},\qquad \textit{and }\qquad
  \overline{B}_{d_{k}}(\omega,\varepsilon) =[\omega\vert{\lceil-\log{\varepsilon}\rceil}]_{k}.
  $$
  Note that the sequence spaces are ultrametric spaces, i.e., $d(x,z) \leq \max\{d(x,y), d(y,z)\}$. 
  As a result, the cylinder sets have the  \emph{net property}\label{prop:net}:
  For $\bu,\bv \in \Sigma_k^\infty$,  either $[\bu]_k \cap [\bv]_k = \varnothing$, or one of them is contained in the other.

Let  $\vect{\sigma}$  be a   sequence of shift mappings $\sigma_{k}:\Sigma_{k}^{\infty}(\vect{m}) \to \Sigma_{k+1}^{\infty}(\vect{m})$ where
  \begin{equation}\label{eq:shiftlvk}       
               \sigma_{k}:\omega_{k}\omega_{k+1}\ldots \mapsto \omega_{k+1}\omega_{k+2}\ldots.
  \end{equation}
It is obvious that the left shift $\sigma_{k}$ is continuous for every $k \in \mathbb{N}$. Then $ (\vect{\Sigma}(\vect{m}),\vect{\sigma})$ forms a nonautonomous symbolic dynamical system.  When all $m_{k}$'s are equal, namely $m_{k}=m$ for all $k \in \mathbb{N}$, the symbolic system $(\vect{\Sigma}(\vect{m}),\vect{\sigma})$   becomes the well-known autonomous symbolic system $(\Sigma(m),\sigma)$. 
  Nonautonomous subshifts (commonly termed \emph{nonstationary subshifts} in literature) may be defined as  nonautonomous shifts on compact subsets,
  and they have been considered in \cite{Fisher2009,Kawan&Latushkin2015}.
  These studies originated in the two-sided NDSs introduced in \cite{Arnoux&Fisher2005}, where Arnoux and Fisher \cite{Arnoux&Fisher2005} generalized the classic Anosov diffeomorphisms into Anosov families and studied the two-sided symbolic dynamics of a particular class of Anosov families;
  Fisher \cite{Fisher2009} continued investigations into mixing and other dynamical properties of the nonautonomous subshifts and adic transformations;
  Kawan and Latushkin \cite{Kawan&Latushkin2015} gave entropy formulae for nonautonomous subshifts and studied particular cases of variational principles;
  Wu and Zhou \cite{Wu&Zhou2023} provided the two-sided symbolic dynamics of a general class of Anosov families.

For each $k \in \mathbb{N}$, let  $\mathbf{p}_{k}=(p_{k,1},\cdots,p_{k,m_{k}})$ be a positive probability vector, that is,  $\sum_{i=1}^{m_{k}}p_{k,i}=1$ and $p_{k,i}>0$ for all $i=1,\ldots,m_{k}$, and  write $ \mat{P}_{k}=(\mathbf{p}_{j})_{j=k}^{\infty}.$
 For each $k \in \mathbb{N}$, we define $\mu_{k}$ on the semi-algebra of cylinders on $\Sigma_{k}^{\infty}(\vect{m})$
  by
  \begin{equation}\label{eq:NAPV}
    \mu_{k}([\mathbf{u}]_{k})=\prod_{j=k}^{k+\len{\mathbf{u}}-1}p_{j,u_{j}},
  \end{equation}
  for all $\mathbf{u}$ of finite length. By the standard argument,  we extend it to  a Borel probability measure,  and  we call it the \emph{$\mat{P}_{k}$-Bernoulli measure} on $\Sigma_{k}^{\infty}(\vect{m})$ and still denote it by $\mu_{k}$.
In this paper, we are only concerned with $\Sigma_{0}^{\infty}(\vect{m})$ and $\mu_{0}$.

The rest of the paper is organized as follows. In Section \ref{sect:mc}, we  present our main results of pressures and entropies on nonautonomous symbolic systems and nonautonomous expansive dynamical systems.  In Section \ref{sect:pre}, we give  some basic notations and recall the definition and properties of various pressures and entropies. In Section \ref{sect:toppresymsys}, we prove the properties of  topological  pressures and entropies on nonautonomous symbolic systems.
In Section \ref{sect:msrpre}, we apply a law of large numbers to explore the measure-theoretic pressures and entropies on nonautonomous Bernoulli measures,  and we prove the existence of Bowen  equilibrium states and packing equilibrium states.  Finally, we investigate   expansiveness and generators on nonautonomous expansive dynamical systems in Section \ref{sect:exps}.

\section{Main Conclusions}\label{sect:mc}

From now on, we always write $(\vect{X},\vect{T})$ for a  nonautonomous dynamical system, that is,
$ \vect{X}=\{ X_{k}  \}_{k=0}^{\infty} $ is  a sequence of compact metric spaces $X_{k}$, and   $\vect{T}=\{T_{k}\}_{k=0}^{\infty}$ is a sequence of
continuous mappings $T_{k}:X_{k} \to X_{k+1}$. We write  $\prebow,\prepac, \preL$ and $ \preU$ for  the  Bowen, packing, lower and upper topological pressures on  $(\vect{X},\vect{T})$, respectively, and write $\vect{C}(\vect{X},\mathbb{R})$  for the collection of all sequences of continuous functions $f_k: X_k \to \mathbb{R}$. They are all concisely defined in  Section \ref{sect:pre}.

  \subsection{Topological pressures of nonautonomous shifts}

  The first theorem shows that subsets with non-empty interior share the same pressures as the entire sequence space $\Sigma_{0}^{\infty}(\vect{m})$.
  \begin{thm}\label{thm:POeqPS}
Given $P\in  \{\prebow,\prepac, \preU\}$,  if $\Omega \subseteq \Sigma_{0}^{\infty}$ has non-empty interior, then for all equicontinuous $\vect{f} \in \vect{C}(\vect{\Sigma}(\vect{m}),\R)$,
$$
    P(\vect{\sigma},\vect{f},\Omega) = P(\vect{\sigma},\vect{f},\Sigma_{0}^{\infty}).
    $$
   \end{thm}
   
  In particular, non-empty open sets are ``pressurely homogeneous''.
  \begin{cor}\label{cor:PBPupopen}
 Given $P\in  \{\prebow,\prepac, \preU\}$, if $\Omega \subseteq \Sigma_{0}^{\infty}$ is non-empty and open, then for all equicontinuous $\vect{f} \in \vect{C}(\vect{\Sigma}(\vect{m}),\R)$
    and all open $V \subseteq \Sigma_{0}^{\infty}$ with $\Omega \cap V \neq \varnothing$,
    $$
    P(\vect{\sigma},\vect{f},\Omega \cap V) = P(\vect{\sigma},\vect{f},\Omega) = P(\vect{\sigma},\vect{f},\Sigma_{0}^{\infty}).
    $$

  \end{cor}
  The particular ``homogeneous'' property for upper capacity pressures leads to the coincidence of the packing and upper capacity pressures
  on non-empty open and compact sets.
  \begin{cor}\label{cor:PupeqPP}
    If $\Omega \subseteq \Sigma_{0}^{\infty}$ is non-empty open and compact, then for all equicontinuous $\vect{f} \in \vect{C}(\vect{\Sigma}(\vect{m}),\R)$,
    $$
    \preU(\vect{\sigma},\vect{f},\Omega) = \prepac(\vect{\sigma},\vect{f},\Omega) = \preU(\vect{\sigma},\vect{f},\Sigma_{0}^{\infty}) = \prepac(\vect{\sigma},\vect{f},\Sigma_{0}^{\infty}).
    $$
  \end{cor}

Given  $\vect{f} \in \vect{C}(\vect{\Sigma}(\vect{m}),\R)$, if $f_{k} \in C(\Sigma_{k}^{\infty},\mathbb{R})$ depends only on the 1st coordinate $\omega_{k}$ of $\omega \in \Sigma_{k}^{\infty}$ for every $k>0$, namely,
  \begin{equation}\label{eq:funct1stcoord}
    f_{k}(\omega)=a_{k,\omega_{k}} \quad (\omega \in \Sigma_{k}^{\infty}),
  \end{equation}
  the topological pressures of $\vect{f}$ are significantly simplified into the following forms.
  \begin{thm}\label{thm:Psymsysf1stcoord}
    Suppose that $\vect{f} \in \vect{C}(\vect{\Sigma}(\vect{m}),\R)$ satisfies \eqref{eq:funct1stcoord}.
    Then
\begin{eqnarray}
   &&\label{eq:PloweqPBsys}
    \preL(\vect{\sigma},\vect{f},\Sigma_{0}^{\infty})=\prebow(\vect{\sigma},\vect{f},\Sigma_{0}^{\infty}) =  \varliminf_{n \to \infty}\frac{1}{n}\sum_{j=0}^{n-1}\log{\Big(\sum_{i=1}^{m_{j}}\e^{a_{j,i}}\Big)},  \\
   &&  \label{eq:PupeqPsymsys}
    \preU(\vect{\sigma},\vect{f},\Sigma_{0}^{\infty})=\prepac(\vect{\sigma},\vect{f},\Sigma_{0}^{\infty}) =  \varlimsup_{n \to \infty}\frac{1}{n}\sum_{j=0}^{n-1}\log{\Big(\sum_{i=1}^{m_{j}}\e^{a_{j,i}}\Big)}.
    \end{eqnarray}
  \end{thm}
Next, we  extend these formulae for general potentials.   Given $\vect{f} \in \vect{C}(\vect{\Sigma}(\vect{m}),\R)$. For each $k \in \mathbb{N}$ and for all $\omega \in \Sigma_{k}^{\infty}$, we write
\begin{equation} \label{def_f*f_*}
  f_{k,*}(\omega) = \inf_{\vartheta \in [\omega_{k}]_{k}}f_{k}(\vartheta)
  \quad \text{and} \quad
  f_{k}^{*}(\omega) = \sup_{\vartheta \in [\omega_{k}]_{k}}f_{k}(\vartheta).
\end{equation}
It is clear that $\vect{f}_{*}=\{f_{k,*}\}_{k=1}^{\infty}$, $\vect{f}^{*}=\{f_{k}^{*}\}_{k=1}^{\infty} \in \vect{C}(\vect{\Sigma}(\vect{m}),\R)$,  and  they are ``dependent only on the 1st coordinate''.   We say $\vect{f}$ is of \emph{strongly bounded variation}
  if there exists a number $b>0$ such that for all $n>0$ and all $\mathbf{u} \in \Sigma_{0}^{n-1}$,
  \begin{equation}\label{eq:strbndvar}
    \lvert S_{n}^{\vect{\sigma}}\vect{f}^{*}(\omega)-S_{n}^{\vect{\sigma}}\vect{f}_{*}(\vartheta) \rvert \leq b,
  \end{equation}
  whenever $\omega,\vartheta \in [\mathbf{u}]$.
Note that 
  $$
  \varlimsup_{n \to \infty}S_{n}^{\vect{\sigma}}(\vect{f}^{*}-\vect{f}_{*})(\omega)<\infty \quad \textit{and }\quad \sum_{j=0}^{\infty}\max_{1 \leq i \leq m_{j}}\sup_{\omega_{k}=\vartheta_{k}=i}\lvert f_{j}(\omega)-f_{j}(\vartheta) \rvert < \infty
  $$  
are two  sufficient conditions for $\vect{f}$ to be of strongly bounded variation. 

  \begin{thm}\label{thm:Psymsys}
    Suppose that $\vect{f} \in \vect{C}(\vect{\Sigma}(\vect{m}),\R)$ is of strongly bounded variation. Then        
\begin{eqnarray}
   &&\label{eq:PlowPBformula}
      \preL(\vect{\sigma},\vect{f},\Sigma_{0}^{\infty}) = \prebow(\vect{\sigma},\vect{f},\Sigma_{0}^{\infty}) = \varliminf_{n \to \infty}\frac{1}{n}\sum_{j=0}^{n-1}\log{\Big(\sum_{i=1}^{m_{j}}\e^{a_{j,i}}\Big)}  \\
    && \label{eq:PupPPformula}
      \preU(\vect{\sigma},\vect{f},\Sigma_{0}^{\infty}) = \prepac(\vect{\sigma},\vect{f},\Sigma_{0}^{\infty}) = \varlimsup_{n \to \infty}\frac{1}{n}\sum_{j=0}^{n-1}\log{\Big(\sum_{i=1}^{m_{j}}\e^{a_{j,i}}\Big)},
    \end{eqnarray}          
    where each $a_{j,i}$ is an arbitrary number in $\big[ \inf_{\vartheta \in [i]_{j}}f_{j}(\vartheta),\sup_{\vartheta \in [i]_{j}}f_{j}(\vartheta) \big]$.
  \end{thm}

 \subsection{Measure-theoretic pressures}
Given  $(\vect{\Sigma}(\vect{m}),\vect{\sigma})$, the \emph{lower and upper  measure-theoretic pressures} of $\vect{\sigma}$ for $\vect{f} \in \vect{C}(\vect{\Sigma}(\vect{m}),\R)$
    with respect to $\mu \in M(\Sigma_{0}^{\infty}(\vect{m}))$ are given respectively by
    $$
      \prelow_{\mu}(\vect{\sigma},\vect{f}) = \int_{\Sigma_{0}^{\infty}}\prelow_{\mu}(\vect{\sigma},\vect{f},\omega)\dif{\mu(\omega)}
      \quad \text{and} \quad
      \preup_{\mu}(\vect{\sigma},\vect{f}) = \int_{\Sigma_{0}^{\infty}}\preup_{\mu}(\vect{\sigma},\vect{f},\omega)\dif{\mu(\omega)},
    $$
    where for each $\omega \in \Sigma_{0}^{\infty}(\vect{m})$,
    \begin{equation}\label{eq:defpremsr}
      \begin{aligned}
        \prelow_{\mu}(\vect{\sigma},\vect{f},\omega) &= \lim_{\varepsilon \to 0}\varliminf_{n \to \infty}\frac{1}{n} \Big(-\log{\mu\big([\omega\vert{(n+\lfloor -\log{\varepsilon} \rfloor)}]\big)}+S_{n}^{\vect{\sigma}}\vect{f}(\omega)\Big) ,\\
        \preup_{\mu}(\vect{\sigma},\vect{f},\omega) &= \lim_{\varepsilon \to 0}\varlimsup_{n \to \infty}\frac{1}{n}\Big(-\log{\mu\big([\omega\vert{(n+\lfloor -\log{\varepsilon} \rfloor)}]\big)}+S_{n}^{\vect{\sigma}}\vect{f}(\omega)\Big).
      \end{aligned}
    \end{equation}
Note that  if $\vect{f}$ is equicontinuous, then for all $\omega \in \Sigma_{0}^{\infty}$,
    \begin{equation}\label{eq:defpremsrequiv}
      \begin{aligned}
        \prelow_{\mu}(\vect{\sigma},\vect{f},\omega) &= \varliminf_{n \to \infty}\frac{1}{n}\big(-\log{\mu([\omega\vert{n}])}+\sup_{\vartheta \in [\omega\vert{n}]}S_{n}^{\vect{\sigma}}\vect{f}(\vartheta)\big), \\
        \preup_{\mu}(\vect{\sigma},\vect{f},\omega) &= \varlimsup_{n \to \infty}\frac{1}{n}\big(-\log{\mu([\omega\vert{n}])}+\sup_{\vartheta \in [\omega\vert{n}]}S_{n}^{\vect{\sigma}}\vect{f}(\vartheta)\big).
      \end{aligned}
    \end{equation}
In particular, the \emph{lower and upper  measure-theoretic local entropies of $\vect{\sigma}$ with respect to $\mu$ at $\omega$} are defined by $\entlow_{\mu}(\vect{\sigma},\omega)= \prelow_{\mu}(\vect{\sigma},\vect{0},\omega)$ and $ \entup_{\mu}(\vect{\sigma},\omega)= \preup_{\mu}(\vect{\sigma},\vect{0},\omega)$. We call 
$$ 
\entup_{\mu}(\vect{\sigma})=\prelow_{\mu}(\vect{\sigma},\vect{0}) \qquad \textit{ and } \qquad \entup_{\mu}(\vect{\sigma},\omega)=   \preup_{\mu}(\vect{\sigma},\vect{0})
$$ 
the \textit{lower and upper measure-theoretic entropies of $\vect{\sigma}$ with respect to $\mu$} respectively.   See \cite{CM2} for details.

We first obtain the formulae for the measure-theoretic entropies. Let  $H(\mathbf{p}_{j})$ denote  the entropy of the probability vector $\mathbf{p}_{j}$, i.e., $H(\mathbf{p}_{j})=-\sum_{i=1}^{m_j}p_{j,i}\log p_{j,i}$.
  \begin{thm}\label{thm:msrent}
 Given  $(\vect{\Sigma}(\vect{m}),\vect{\sigma})$ with $\sup_{n\geq 1}{m_{n}}<+\infty$, let $\mu$ be the nonautonomous Bernoulli measure given by \eqref{eq:NAPV}.  Then
\begin{equation}\label{fml_hlhu}
    \entlow_{\mu}(\vect{\sigma}) = \varliminf_{n \to \infty}\frac{1}{n} \sum_{j=0}^{n-1} H(\mathbf{p}_{j}),
    \qquad
    \entup_{\mu}(\vect{\sigma}) = \varlimsup_{n \to \infty}\frac{1}{n} \sum_{j=0}^{n-1} H(\mathbf{p}_{j}).
\end{equation}
  \end{thm}

Similarly,   measure-theoretic pressures have the following simplified formulae. 
 \begin{thm}\label{thm:msrpre}
 Given  $(\vect{\Sigma}(\vect{m}),\vect{\sigma})$ with $\sup_{n\geq 1}{m_{n}}<+\infty$, let $\mu$ be the nonautonomous Bernoulli measure  given by \eqref{eq:NAPV}
    and   $\vect{f} \in \vect{C}(\vect{\Sigma}(\vect{m}),\R)$ be of strongly bounded variation with  $\|\vect{f}\|<+\infty$.   
    \begin{enumerate}[(a)]
      \item If $p_{*}=\inf_{j \in \N, 1\leq i\leq m_j}\{p_{j,i}\}>0$, then  
$$
 \prelow_{\mu}(\vect{\sigma},\vect{f}) = \varliminf_{n \to \infty}\frac{1}{n}\sum_{j=0}^{n-1}\big(H(\mathbf{p}_{j}) + \E(f_j)\big), 
\qquad  \preup_{\mu}(\vect{\sigma},\vect{f}) = \varlimsup_{n \to \infty}\frac{1}{n}\sum_{j=0}^{n-1}\big(H(\mathbf{p}_{j}) + \E(f_j)\big).
$$
      \item If for $\mu$-a.e. $\omega \in \Sigma_{0}^{\infty}$, either $\entlow_{\mu}(\vect{\sigma},\omega) = \entup_{\mu}(\vect{\sigma},\omega)$ or 
$\{f_{k} \circ \vect{\sigma}^{k}(\omega)\}_{k=1}^\infty$ converges pointwise to some integrable $f(\omega)$,  then  
$$
 \prelow_{\mu}(\vect{\sigma},\vect{f}) = \entlow_{\mu}(\vect{\sigma}) +\varliminf_{n \to \infty}\frac{1}{n}\sum_{j=0}^{n-1}\E(f_j), 
\qquad  \preup_{\mu}(\vect{\sigma},\vect{f}) =\entup_{\mu}(\vect{\sigma})+ \varlimsup_{n \to \infty}\frac{1}{n}\sum_{j=0}^{n-1} \E(f_j).
$$

    \end{enumerate}
  \end{thm}
 
Note that $\sup_{n=1,2,\ldots}{m_{n}}<+\infty$ and $\|\vect{f}\|<+\infty$ may be relaxed to some weak assumptions; see Section \ref{sect:msrpre} for details.

In the next example, we show that  the measure-theoretic entropies are equal to the topological entropies for certain measures, which are frequently used in dimension theory. 
\begin{exmp}\label{exmp:bowethausonseqsp}

For $ (\vect{\Sigma}(\vect{m}),\vect{\sigma})$,   let $\nu$ be  the nonautonomous Bernoulli measure on $\Sigma_{0}^{\infty}(\vect{m})$ generated by  the  probability vectors $   \mathbf{p}_{k}=\left(\frac{1}{m_{k}},\cdots,\frac{1}{m_{k}}\right).$
Then  we have 
  \begin{equation}\label{eq:mmeunidistr}
      \begin{aligned}
        \entlow_{\nu}(\vect{\sigma})=\varliminf_{n \to \infty}\frac{1}{n}\sum_{j=0}^{n-1}\log{m_{j}}=\enttopbow(\vect{\sigma},\Sigma_{0}^{\infty}), \\
        \entup_{\nu}(\vect{\sigma})=\varlimsup_{n \to \infty}\frac{1}{n}\sum_{j=0}^{n-1}\log{m_{j}}=\enttoppac(\vect{\sigma},\Sigma_{0}^{\infty}).
      \end{aligned}
    \end{equation}

By simple calculation, we have that 
\begin{eqnarray*}
&&    \enttopbow(\vect{\sigma},\Sigma_{0}^{\infty})=\dimh(\Sigma_{0}^{\infty})=\varliminf_{k \to \infty}\frac{1}{n}\sum_{j=0}^{n-1}\log{m_{j}},  \\
&&    \enttoppac(\vect{\sigma},\Sigma_{0}^{\infty})=\dimp(\Sigma_{0}^{\infty})=\varlimsup_{k \to \infty}\frac{1}{n}\sum_{j=0}^{n-1}\log{m_{j}},
\end{eqnarray*}
 see \cite{Mattila1995} for the definition of Hausdorff dimension.  
By Theorem  \ref{thm:msrent}, 
    $$
    \entlow_{\nu}(\vect{\sigma},\omega) =\varliminf_{n \to \infty}\frac{1}{n}\sum_{j=0}^{n-1}\log{m_{j}},
    \qquad
    \entup_{\nu}(\vect{\sigma},\omega) =\varlimsup_{n \to \infty}\frac{1}{n}\sum_{j=0}^{n-1}\log{m_{j}}.
    $$
Then the equalities in \eqref{eq:mmeunidistr} hold. 

 \end{exmp}
 
Note that the coincidence  in \eqref{eq:mmeunidistr} shows  that the uniform mass distribution $\nu$ is a \emph{measure with maximal entropy},
    that is, the  lower and upper measure-theoretic entropies of the measure attain  the supremum of the system's Bowen and packing topological entropy, respectively. Such measures are often referred to the equilibrium states for $0$-potentials, and we study such measures in details in the next subsection.

\subsection{Equilibrium states and Gibbs states}
First, we reformulate Theorem \ref{thm_vpBP} in the context of symbolic dynamics as follows.
  If $\Omega \subseteq \Sigma_{0}^{\infty}$ is non-empty and compact, then for all equicontinuous $\vect{f} \in \vect{C}(\vect{\Sigma}(\vect{m}),\R)$,
  \begin{equation}\label{eq:varprinPB}
    \prebow(\vect{\sigma},\vect{f},\Omega) = \sup\{\prelow_\mu(\vect{\sigma},\vect{f}): \mu \in M(\Sigma_{0}^{\infty})\ \text{and}\ \mu(\Omega)=1\},
  \end{equation}
  and for equicontinuous $\vect{f}$ with $\|\vect{f}\|<+\infty$ and $\prepac(\vect{\sigma},\vect{f},\Omega)>\|\vect{f}\|$, we have
  \begin{equation}\label{eq:varprinPP}
    \prepac(\vect{\sigma},\vect{f},\Omega) = \sup\{\preup_\mu(\vect{\sigma},\vect{f}): \mu \in M(\Sigma_{0}^{\infty})\ \text{and}\ \mu(\Omega)=1\}.
  \end{equation}
   Note that the variational principle  for the packing pressure  requires  additional assumptions that   $\vect{f}$ is uniformly bounded and that $\prepac(\vect{\sigma},\vect{f},\Omega)>\|\vect{f}\|$. However, these assumptions may be removed in  symbolic dynamical systems for certain potentials; see Corollary \ref{cor_VP} and 
 Corollary \ref{cor_VPS}.

  In the remainder of this subsection, we   discuss the equilibrium and Gibbs states for the various pressures in nonautonomous symbolic dynamical systems,
  which are natural follow-up objects of the variational principles.   We first define the equilibrium states for the Bowen and packing pressures.
  \begin{defn}
    Given a compact set $\Omega \subseteq \Sigma_{0}^{\infty}(\vect{m})$ and equicontinuous $\vect{f} \in \vect{C}(\vect{\Sigma}(\vect{m}),\R)$,
    a Borel probability measure $\mu \in M(\Sigma_{0}^{\infty}(\vect{m}))$ is called a \emph{Bowen equilibrium state for $\vect{f}$ on $\Omega$}
    if $\mu(\Omega)=1$ and $\prebow(\vect{\sigma},\vect{f},\Omega)=\prelow_{\mu}(\vect{\sigma},\vect{f})$;
    and $\mu$ is called a \emph{packing equilibrium state for $\vect{f}$ on $\Omega$}
    if $\mu(\Omega)=1$ and $\prepac(\vect{\sigma},\vect{f},\Omega)=\preup_{\mu}(\vect{\sigma},\vect{f})$.
  \end{defn}
 Let $\eqlbrbow_{\vect{f}}(\Omega)$ denote the collection of all Bowen equilibrium states for $\vect{f}$ on $\Omega$
    and $\eqlbrpac_{\vect{f}}(\Omega)$ the collection of all packing equilibrium states for $\vect{f}$ on $\Omega$.

  A particular case of the equilibrium states is the measures of maximal entropy, as introduced in Example~\ref{exmp:bowethausonseqsp}.
  By \eqref{eq:mmeunidistr}, the sets $\eqlbrbow_{\vect{0}}(\Sigma_{0}^{\infty}(\vect{m}))$ and $\eqlbrpac_{\vect{0}}(\Sigma_{0}^{\infty}(\vect{m}))$ are non-empty by choosing nonautonomous Bernoulli measures with  $p_*>0$.
  The following result generalizes \eqref{eq:mmeunidistr}, which is a direct consequence of Theorems~\ref{thm:Psymsys} and \ref{thm:msrpre}.
  \begin{prop}\label{prop:eqlbrNABmsr}
    Suppose that $\vect{f}$ satisfies \eqref{eq:strbndvar}. Then $\eqlbrbow_{\vect{f}}(\Sigma_{0}^{\infty}(\vect{m})) \neq \varnothing$ and $\eqlbrpac_{\vect{f}}(\Sigma_{0}^{\infty}(\vect{m})) \neq \varnothing$.
 
    In particular, let $a_{j,i}\in\big[ \inf_{\vartheta \in [i]_{j}}f_{j}(\vartheta),\sup_{\vartheta \in [i]_{j}}f_{j}(\vartheta) \big]$. Then the nonautonomous Bernoulli measure $\mu$ generated by
\begin{equation} \label{pvama}
    p_{j,i}=\frac{\e^{a_{j,i}}}{\sum_{i=1}^{m_{j}}\e^{a_{j,i}}}
\end{equation}
    is both a Bowen equilibrium state and a packing equilibrium state for $\vect{f}$ on $\Sigma_{0}^{\infty}$.
  \end{prop}

  The following results on equilibrium states are inspired by a recent work of Wang and Zhang \cite{WZ}, where they studied the properties of measures of maximal Bowen and packing entropies on analytic subsets in TDSs.
  \begin{prop}\label{prop:eqlbrae}
   Given $\Omega \subseteq \Sigma_{0}^{\infty}$ and  $\vect{f} \in \vect{C}(\vect{\Sigma}(\vect{m}),\R)$.
    \begin{enumerate}[(1)]
      \item If $\mu \in \eqlbrbow_{\vect{f}}(\Omega)$, then $\prelow_{\mu}(\vect{\sigma},\vect{f},\omega)=\prebow(\vect{\sigma},\vect{f},\Omega)$ for $\mu$-a.e. $\omega \in \Sigma_{0}^{\infty}(\vect{m})$.
      \item If $\mu \in \eqlbrpac_{\vect{f}}(\Omega)$, then $\preup_{\mu}(\vect{\sigma},\vect{f},\omega)=\prepac(\vect{\sigma},\vect{f},\Omega)$ for $\mu$-a.e. $\omega \in \Sigma_{0}^{\infty}(\vect{m})$.
    \end{enumerate}
  \end{prop}

The following conclusion is a direct consequence of Proposition~\ref{prop:eqlbrae}.
  \begin{prop}\label{prop:eqlbrrestr}
    Given $\Theta \subseteq \Omega \subseteq \Sigma_{0}^{\infty}(\vect{m})$ and $\vect{f} \in \vect{C}(\vect{\Sigma}(\vect{m}),\R)$.
 For each $\mu \in M(\Sigma_{0}^{\infty})$ with $\mu(\Theta)>0$, we write $\nu=\frac{\mu\vert_{\Theta}}{\mu(\Theta)}$. 
    \begin{enumerate}[(1)]
      \item If $\mu \in \eqlbrbow_{\vect{f}}(\Omega)$, then $\nu \in \eqlbrbow_{\vect{f}}(\Omega)$ and $\nu \in \eqlbrbow_{\vect{f}}(\Theta)$;
            in particular,
            $$
            \prelow_{\nu}(\vect{\sigma},\vect{f})=\prebow(\vect{\sigma},\vect{f},\Omega)=\prebow(\vect{\sigma},\vect{f},\Theta).
            $$
      \item If $\mu \in \eqlbrpac_{\vect{f}}(\Omega)$, then $\nu \in \eqlbrpac_{\vect{f}}(\Omega)$ and $\nu \in \eqlbrpac_{\vect{f}}(\Theta)$;
            in particular,
            $$
            \preup_{\nu}(\vect{\sigma},\vect{f})=\prepac(\vect{\sigma},\vect{f},\Omega)=\prepac(\vect{\sigma},\vect{f},\Theta).
            $$
    \end{enumerate}
  \end{prop}

Since entropies are special pressures for $\vect{f}=\vect{0}$, the following conclusion is a  particular case of  Proposition \ref{prop:eqlbrrestr} which was also proved in \cite[Prop.3.2]{WZ}.
  \begin{cor}
    Let $\Omega \subseteq \Sigma_{0}^{\infty}$ be a compact subset.
    \begin{enumerate}
      \item If $\enttopbow(\vect{\sigma},\Omega)>0$, then every $\mu \in \eqlbrbow_{\vect{0}}(\Omega)$ is non-atomic.
      \item If $\enttoppac(\vect{\sigma},\Omega)>0$, then every $\mu \in \eqlbrpac_{\vect{0}}(\Omega)$ is non-atomic.
    \end{enumerate}
  \end{cor}

  The following result on the existence of equilibrium states   is an immediate consequence of Propositions~\ref{prop:eqlbrNABmsr} and \ref{prop:eqlbrrestr}.
  \begin{thm}
    Suppose that $\vect{f}$ satisfies \eqref{eq:strbndvar}. Let $\mu$ be the nonautonomous Bernoulli measure given by \eqref{pvama}.
    Then $\eqlbrbow_{\vect{f}}(\Omega) \neq \varnothing$ and $\eqlbrpac_{\vect{f}}(\Omega) \neq \varnothing$ for all non-empty compact $\Omega \subseteq \Sigma_{0}^{\infty}(\vect{m})$ with $\mu(\Omega)>0$.
  \end{thm}

  \begin{cor}\label{cor_VP}
    Suppose that $\vect{f}$ satisfies \eqref{eq:strbndvar}. Let $\mu$ be the nonautonomous Bernoulli measure generated by \eqref{pvama}.
    Then for all non-empty compact $\Omega \subseteq \Sigma_{0}^{\infty}(\vect{m})$ with $\mu(\Omega)>0$,
  \begin{equation} 
    \prepac(\vect{\sigma},\vect{f},\Omega) = \sup\{\preup_\mu(\vect{\sigma},\vect{f}): \mu \in M(\Sigma_{0}^{\infty}(\vect{m}))\ \text{and}\ \mu(\Omega)=1\}.
  \end{equation}

  \end{cor}

  A class of commonly considered candidates for equilibrium states is the so called Gibbs states.
  They need not be Bernoulli but satisfy a similar property that is intimately related.
  Let $P \in \{\prebow,\prepac,\preL,\preU\}$. Given $\mu \in M(\Omega)$, if there exists a constant $K>1$ such that
  for all $n \geq 1$ and $\omega \in \Omega$,
  $$
  K^{-1} \leq \frac{\mu([\omega\vert{n}] \cap \Omega)}{\exp{(-nP(\vect{\sigma},\vect{f},\Omega)+S_{n}^{\vect{\sigma}}\vect{f}(\omega))}} \leq K,
  $$
  then we say $\mu$ is a \emph{$P$-Gibbs state for $\vect{f}$ on $\Omega$}.

  \begin{ques}
   Under what conditions do equilibrium states or even $P$-Gibbs states exists for broader classes of potentials $\vect{f}$? 
  \end{ques}

\subsection{Expansive nonautonomous dynamical  systems}

 First, we provide the definition of expansiveness in nonautonomous dynamical systems.
      \begin{defn}
We say that  $(\vect{X},\vect{T})$ is  \emph{expansive} if there exists $\delta>0$ such that  for all distinct $x, y \in X_{0}$, we have $d_{X_{j}}(\vect{T}^{j}x,\vect{T}^{j}y)>\delta$ for  some  $j \in \N$.
        We call $\delta$ an \emph{expansive constant} for $\vect{T}$.
      \end{defn}
Note that  if $\delta_{0}>0$ is an expansive constant for $\vect{T}$, then so is every $\delta$ with $0<\delta \leq \delta_{0}$.
Our definition of expansiveness in NDSs is a generalization of positively expansiveness in TDSs.
        A positively expansive TDS $(X,T)$ is expansive as an NDS in our terms.
        See \cite{Wu&Zhou2023} for the definition of expansiveness in two-sided NDSs and the example of a certain class of Anosov families.  Expansiveness is invariant under equiconjugacies of NDSs \cite[Prop.7.7 (\romannumeral1)]{Kawan2015}.
        However, it is not preserved under the operation of taking factors even in the autonomous case \cite[\S5.6 Rmk.(3)]{Walters1982}.

   \begin{rmk}
        (1) Let $(\vect{X},\vect{T})$ be expansive, and let $K \subseteq X_{0}$. Write $\vect{X}\vert_{K}=\{\vect{T}^{k}K\}_{k=0}^{\infty}$ and $\vect{T}\vert_{K}=\{T_{k}\vert_{\vect{T}^{k}K}\}_{k=0}^{\infty}$.
        Then $(\vect{X}\vert_{K},\vect{T}\vert_{K})$ is an expansive NDS.

        (2) Given   $(\vect{X},\vect{T})$. For $m \geq 1$, write $\vect{X}^{[m]}=\{X_{km}\}_{k=0}^{\infty}$ and $\vect{T}^{[m]}=\{\vect{T}_{km}^{m}\}_{k=0}^{\infty}$.
        Then $(\vect{X},\vect{T})$ is expansive iff $(\vect{X}^{[m]},\vect{T}^{[m]})$ expansive.
      \end{rmk}

To study  pressures for expansive NDSs, we define generators and weak generators which was  originally introduced   for entropies in TDSs by Keynes and Robertson \cite{Keynes&Robertson1969}.
 Given a sequence $\vect{\mathscr{U}}$ of open covers $\mathscr{U}_{k}$ of $X_{k}$.  If there exists $\delta>0$ such that  $\delta$ is a Lebesgue number for $\mathscr{U}_{k}$ for all  $k \geq 0$, then we say $\delta$ is a \emph{Lebesgue number} for $\vect{\mathscr{U}}$.

      \begin{defn} \label{def_gnrt}
        Given  $(\vect{X},\vect{T})$. Let $\vect{\mathscr{U}}=\{\mathscr{U}_{K}\}_{k=0}^{\infty}$ be a sequence of finite open covers $\mathscr{U}_{k}$ of $X_{k}$ with a Lebesgue number.
        We call $\vect{\mathscr{U}}$ a \emph{generator for $\vect{T}$} if for every sequence $\{U_{k}\}_{k=0}^{\infty}$ of sets $U_{k} \in \mathscr{U}_{k}$,
        the set $\bigcap_{j=0}^{\infty}\vect{T}^{-j}\overline{U_{j}}$ contains at most one point of $X_{0}$;
        and we call $\vect{\mathscr{U}}$ a \emph{weak generator for $\vect{T}$} if $\bigcap_{j=0}^{\infty}\vect{T}^{-j}U_{j}$ contains at most one point of $X_{0}$ for all sequences $\{U_{k} \in \mathscr{U}_{k}\}_{k=0}^{\infty}$.
      \end{defn}

      \begin{thm}\label{prop:expseqgnrtr}
        Given $(\vect{X},\vect{T})$, the following are equivalent:
        \begin{enumerate}[(1)]
          \item $\vect{T}$ is expansive;
          \item $\vect{T}$ has a generator;
          \item $\vect{T}$ has a weak generator.
        \end{enumerate}
        Moreover, if $\delta>0$ is a Lebesgue number for a (weak) generator, then it is also an expansive constant for $\vect{T}$;
        conversely, if $\delta_{0}>0$ is an expansive constant for $\vect{T}$, then there exists $\varepsilon$ with $0<\varepsilon<\frac{\delta_{0}}{4}$
        such that every $\delta$ with $0<\delta \leq \varepsilon$ is a Lebesgue number for a (weak) generator.
      \end{thm}

 In general, the expansiveness of $(\vect{X},\vect{T})$  does not imply  the expansiveness of the shifted $(\vect{X}_{k},\vect{T}_{k})$ for any $k \geq 1$, see \cite[Exmp.7.3]{Kawan2015}.   A sufficient condition is provided by Kawan in \cite[Prop.7.7 (\romannumeral2)]{Kawan2015}. Next, we consider the system is uniformly expansive.

      \begin{defn}\label{def:uniexps}
        We say $(\vect{X},\vect{T})$ is \emph{uniformly expansive} if
        \begin{enumerate}[(a)]
          \item $(\vect{X}_{k},\vect{T}_{k})$ is expansive for every $k \geq 0$ and
          \item there exists a uniform expansive constant $\delta>0$ for all $\vect{T}_{k}$.
        \end{enumerate}

        We call $\vect{\mathscr{U}}$ a \emph{uniform (weak) generator} for $\vect{T}$
        if for every $k \geq 0$, $\vect{\mathscr{U}}_{k}=\{\mathscr{U}_{k+j}\}_{j=0}^{\infty}$ is a (weak) generator for $\vect{T}_{k}$.
      \end{defn}
There exist examples such that  $(\vect{X},\vect{T})$ satisfies (a) but not (b); see \cite[Exmp..4]{Kawan2015}.
Note that Subsystems of uniformly expansive NDSs are uniformly expansive, and  positively expansive TDSs are uniformly expansive.

      \begin{prop}\label{prop_eque}
        Given   $(\vect{X},\vect{T})$, the following are equivalent:
        \begin{enumerate}[(1)]
          \item $\vect{T}$ is uniformly expansive;
          \item $\vect{T}$ has a uniform generator;
          \item $\vect{T}$ has a uniform weak generator.
        \end{enumerate}
      \end{prop}

    The following definition was introduced by Kawan in \cite{Kawan2015}.
      \begin{defn}\label{def:sue}
        We say that $(\vect{X},\vect{T})$ is  \emph{ strongly uniformly expansive (sue)} if there exists $\delta>0$ such that
        for every $\varepsilon>0$, there is an integer $N \geq 1$ satisfying that
        for all $k \in \N$ and $x,y \in X_{k}$,
        $$
        d_{k}(x,y)<\varepsilon \quad \text{whenever}\ d_{k,N}^{\vect{T}}(x,y)<\delta.
        $$
        We call $\delta$ a \emph{strongly uniformly expansive  constant} of $\vect{T}$.
      \end{defn}
      
 Note that strong uniform expansiveness, uniform expansiveness and expansiveness are all equivalent to positive expansiveness in TDSs.
        It is clear by \eqref{eq:bowenballsymsys} that nonautonomous symbolic dynamical systems   $(\vect{\Sigma}(\vect{m}),\vect{\sigma})$ are strongly uniformly expansive with the expansive constant $\e^{-1}$. See \cite[\S7.2]{Kawan2015} for  other properties and examples of sue NDSs,.

    \begin{rmk}
    Strong uniform expansiveness  is invariant under equiconjugacies of NDSs but not under equisemiconjugacies (see \cite[Prop.7.12(\romannumeral2)]{Kawan2015} and \cite[\S5.6 Rmk.(3)]{Walters1982}).     
      \end{rmk}

The next conclusion  provides not only the simplified formulation for the topological pressures of  sue NDSs  but also a sufficient condition for the existence of generators.  
  \begin{thm}\label{thm:suepre}
   Let $(\vect{X},\vect{T})$ be strongly uniformly expansive. Given $Z \subseteq X_{0}$.
    \begin{enumerate}[(1)]
      \item Let $\vect{\mathscr{U}}=\{\mathscr{U}_{k}\}_{k=0}^{\infty}$ be a uniform (weak) generator for $\vect{T}$ satisfying that for every $\varepsilon>0$, there exists an integer $N>0$ such that $\diam(\vee_{k,N}^{\vect{T}}\vect{\mathscr{U}}) \leq \varepsilon$ for all $k \in \N$. Then for all equicontinuous $\vect{f} \in \vect{C}_{b}(\vect{X},\R)$,
        $$
        P(\vect{T},\vect{f},Z) = Q(\vect{T},\vect{f},Z,\vect{\mathscr{U}}),
        $$
        where $P \in \{\prebow,\prelow,\preup\}$ and $Q$ is the corresponding one in $\{\prebpp,\qrelow,\qreup\}$.
      \item If $\delta>0$ is an expansive constant for $\vect{T}$, then for every $\varepsilon$ with $0<\varepsilon<\frac{\delta}{4}$ and  all equicontinuous $\vect{f} \in \vect{C}_{b}(\vect{X},\R)$,
        $$
        P(\vect{T},\vect{f},Z) = P(\vect{T},\vect{f},Z,\varepsilon),
        $$
        where $P \in \{\prebow,\prepac,\prelow,\preup\}$.
    \end{enumerate}
  \end{thm}
 Consequently, we have the following result.
      \begin{cor}
  Let $(\vect{X},\vect{T})$ be strongly uniformly expansive. Given $Z \subseteq X_{0}$.
        \begin{enumerate}[(1)]
          \item If $\vect{\mathscr{U}}=\{\mathscr{U}_{k}\}_{k=0}^{\infty}$ is a uniform (weak) generator for $\vect{T}$ satisfying Lemma~\ref{lem: A.5}(2), then  for each $h \in \{\enttopbow,\enttoplow,\enttopup\}$, 
                $$
                h(\vect{T},Z) = h(\vect{T},Z,\vect{\mathscr{U}}).
                $$
              
          \item If $\delta>0$ is an expansive constant for $\vect{T}$, then  for each  $h \in \{\enttopbow,\enttoppac,\enttoplow,\enttopup\}$,
                $$
                h(\vect{T},Z) = h(\vect{T},Z,\varepsilon),
                $$
               for all $0<\varepsilon<\frac{\delta}{4}$.
        \end{enumerate}
      \end{cor}

  The next conclusion states that  every sue NDS is equisemiconjugate to a subsystem of the nonautonomous symbolic dynamical system.
  In other words, sue NDSs are \emph{equi-factors} of nonautonomous subshifts.       
      \begin{thm}\label{thm:suesymext}
  Let $(\vect{X},\vect{T})$ be strongly uniformly expansive.  Then
          there exists a nonautonomous shift $(\vect{\Sigma}(\vect{m}),\vect{\sigma})$, a closed $\Omega \subseteq \Sigma_{0}^\infty(\vect{m})$
                and an equicontinuous sequence $\vect{\pi}$ of surjections  $\pi_{k}: \vect{\sigma}^{k}\Omega \to X_{k}$ such that
                $\pi_{k+1} \circ \sigma_{k} = T_{k+1} \circ \pi_{k}$ for all $k \in \N$.
      \end{thm}

      The final example shows that cylinders form a uniform generator of $(\vect{\Sigma}(\vect{m}),\vect{\sigma})$.
      \begin{exmp}\label{exmp:symsuepre}
        Given a nonautonomous shift $(\vect{\Sigma}(\vect{m}),\vect{\sigma})$, let $\vect{\mathscr{C}}=\{[\bu]_k: \bu\in \Sigma_{k}^{k}\}_{k=0}^{\infty}$.
        Then $\vect{\mathscr{C}}$ is a uniform generator for $\vect{\sigma}$. Moreover, by \eqref{eq:bowenballsymsys} and Theorem~\ref{thm:suepre}, for all equicontinuous $\vect{f} \in \vect{C}_{b}(\vect{X},\R)$, we have 
        $$
        \prelow(\vect{\sigma},\vect{f},\Omega)=\qrelow(\vect{\sigma},\vect{f},\Omega,\vect{\mathscr{C}}), \qquad   \qquad 
        \preup(\vect{\sigma},\vect{f},\Omega)=\qreup(\vect{\sigma},\vect{f},\Omega,\vect{\mathscr{C}}),
        $$
and 
        $$
        \prebow(\vect{\sigma},\vect{f},\Omega)=\prebpp(\vect{\sigma},\vect{f},\Omega,\vect{\mathscr{C}}) .
        $$

      \end{exmp}

\section{Topological Pressures and Entropies}\label{sect:pre}

In this section, we cite the definitions and properties of topological pressures and entropies for readers' convenience, and we refer the readers to  \cite{CM1} for details.

\subsection{Bowen metrics and Bowen balls}\label{ssect:BB}

Given   $(\vect{X},\vect{d},\vect{T})$, for each integer $k\geq 0$, we write
  $$
  \vect{T}_{k}^{j}=T_{k+(j-1)} \circ \cdots \circ T_{k}: X_{k} \to X_{k+j}
  $$
for $ j = 1,2,3,\cdots$, and we adopt the convention that  $ \vect{T}_{k}^{0}=\id_{X_{k}}$
where $\id_{X_{k}}:X_{k} \to X_{k}$ is the identity mapping.
Since the mappings $T_{k}$ are not necessarily bijective, we write $\vect{T}_{k}^{-j} = (\vect{T}_{k}^{j})^{-1}$ for the preimage of subsets of $X_{k+j}$ under $\vect{T}_{k}^{j}$.
For simplicity, we often write $\vect{T}^{j}=\vect{T}_{0}^{j}$ for $j \in \mathbb{Z}$.

Given $k \in \mathbb{N}$, for $n=1,2,3,\ldots,$ we define the {\emph{$n$-th Bowen metric at level $k$}} or the {\emph{$n$-th Bowen metric} on $X_{k}$} by
$$
d_{k,n}^{\vect{T}}(x,y)=\max_{0 \leq j \leq n-1}d_{k+j}(\vect{T}_{k}^{j}x,\vect{T}_{k}^{j}y)
$$
for all $x,y \in X_{k}$.  Given $\varepsilon>0$ and $x \in X_{k}$, the
\emph{open  and closed Bowen balls  with center  $x$ and  radius $\varepsilon$ at level $k$} are respectively given by
  \begin{equation}\label{eq:bowenballopen}
    B_{k,n}^{\vect{T}}(x,\varepsilon)=\bigcap_{j=0}^{n-1}\vect{T}_{k}^{-j}B_{X_{k+j}}(\vect{T}_{k}^{j}x,\varepsilon),
 \qquad
    \overline{B}_{k,n}^{\vect{T}}(x,\varepsilon)=\bigcap_{j=0}^{n-1}\vect{T}_{k}^{-j}\overline{B}_{X_{k+j}}(\vect{T}_{k}^{j}x,\varepsilon).
  \end{equation}
For simplicity, we write $d_{n}^{\vect{T}}=d_{0,n}^{\vect{T}}$ and 
$$
B_{n}^{\vect{T}}(x,\varepsilon)=B_{0,n}^{\vect{T}}(x,\varepsilon), \qquad
 \overline{B}_{n}^{\vect{T}}(x,\varepsilon)=\overline{B}_{0,n}^{\vect{T}}(x,\varepsilon).
$$

We denote the collection of all sequences of continuous functions $f_k: X_k \to \mathbb{R}$ by
$$
\vect{C}(\vect{X},\mathbb{R})=\prod_{k=0}^{\infty}C(X_{k},\mathbb{R}).
$$
 We often write $\vect{0}$ and $\vect{1} \in \vect{C}(\vect{X},\mathbb{R})$ for the sequence of constant $0$ functions and constant $1$ functions, respectively, and $a\vect{1} \in \vect{C}(\vect{X},\mathbb{R})$ for the sequence of constant $a$ functions where $a \in \mathbb{R}$.
  Given $\vect{f}=\{f_{k}\}_{k=0}^{\infty}$ and $\vect{g}=\{g_{k}\}_{k=0}^{\infty} \in \vect{C}(\vect{X},\mathbb{R})$, we write $\vect{f} \preceq \vect{g}$
  if $f_{k} \leq g_{k}$ for all $k \in \mathbb{N}$.
 Given $\vect{f}\in \vect{C}(\vect{X},\mathbb{R})$, we write
\begin{equation}\label{eq_fnorm}
\|\vect{f}\|=\sup_{k \in \mathbb{N}}\big\{\|f_{k}\|_{\infty}=\max_{x \in X_{k}}\lvert f_{k}(x) \rvert\big\}.
\end{equation}
  It is clear that $\|\vect{f}\|<+\infty$ implies that $\vect{f}$ is uniformly bounded. We denote the collection of all uniformly bounded function sequences  in $\vect{C}(\vect{X},\mathbb{R})$ by
  $$\vect{C}_{b}(\vect{X},\mathbb{R})=\{\vect{f}\in \vect{C}(\vect{X},\mathbb{R}): \|\vect{f}\|<+\infty\}.$$
Given $\vect{f} \in \vect{C}(\vect{X},\mathbb{R})$, we say $\vect{f}$ is \emph{equicontinuous} if for every $\varepsilon>0$,
  there exists $\delta>0$ such that for all $k \in \mathbb{N}$ and all $x^{\prime},x^{\prime\prime} \in X_{k}$ satisfying
  $d_{k}(x^{\prime},x^{\prime\prime})<\delta$, we have 
  $$
  \lvert f_{k}(x^{\prime})-f_{k}(x^{\prime\prime}) \rvert<\varepsilon.
  $$

Given $\vect{f} \in \vect{C}(\vect{X},\mathbb{R})$, for  $k, n \in \mathbb{N}$, we write
  \begin{equation}\label{eq:sumknTf}
    S_{k,n}^{\vect{T}}\vect{f}=\sum_{j=0}^{n-1} f_{k+j} \circ \vect{T}_{k}^{j}.
  \end{equation}
 For simplicity, we write
  \begin{equation}\label{eq:sum0nTf}
    S_{n}^{\vect{T}}\vect{f}=S_{0,n}^{\vect{T}}\vect{f}=\sum_{j=0}^{n-1} f_{j} \circ \vect{T}^{j}.
  \end{equation}

In the  nonautonomous symbolic system $(\vect{\Sigma}(\vect{m}),\vect{\sigma})$, 
  by the essence of the shift dynamics and \eqref{eq:bowenballopen}, the Bowen balls  have much simpler expressions. The $n$-th Bowen ball of level $k$ about $\omega \in \Sigma_{k}^{\infty}(\vect{m})$
  of radius $\varepsilon>0$ is
  \begin{equation}\label{eq:bowenballsymsys}
    B_{k,n}^{\vect{\sigma}}(\omega,\varepsilon)=[\omega\vert{(n+\lfloor -\log{\varepsilon}+1 \rfloor-1)}]_{k},
  \end{equation}
  the cylinder of base $\omega\vert{(n+N-1)}$, which is of rank $n+\lfloor -\log{\varepsilon}+1 \rfloor-1$,
  Similarly, the $n$-th closed Bowen ball of level $k$ about $\omega \in \Sigma_{k}^{\infty}(\vect{m})$
  of radius $\varepsilon>0$ is
  \begin{equation}\label{def_BSkn}
    \overline{B}_{k,n}^{\vect{\sigma}}(\omega,\varepsilon)=[\omega\vert{(n+\lceil -\log{\varepsilon} \rceil-1)}]_{k},
  \end{equation}
  It is clear that cylinders are also exactly the open and at the same time closed Bowen balls, that is,  $\{[\bu]_{0}: \bu \in \Sigma_{0}^{*}\}$
  is the collection of all Bowen balls.

  \subsection{Lower and upper topological pressures and entropies}\label{ssect:PLPU}
      In this subsection, we give definitions for the lower and upper topological pressures and entropies of NDSs on subsets.
      
      A standard approach is via the $(n,\varepsilon)$-spanning and $(n,\varepsilon)$-separated sets.      Given a subset $Z \subset X_{0}$, we call $F\subset X_{0}$   a \emph{$(n,\varepsilon)$-spanning set for $Z$ with respect to $\vect{T}$}  if for every $x \in Z$,  there exists $y \in F$ with $d_{0,n}^{\vect{T}}(x,y) \leq \varepsilon$.
    As a dual counterpart, a set $E \subset Z$ is called a \emph{$(n,\varepsilon)$-separated set for $Z$ with respect to $\vect{T}$}
    if  any two distinct points $x,y \in E$  implies $d_{0,n}^{\vect{T}}(x,y)>\varepsilon$.

      For $\vect{f} \in \vect{C}(\vect{X},\R)$, integral $n \geq 1$ and real $\varepsilon>0$, we write
      \begin{eqnarray}
  && \label{eq:defQn}
       Q_{n}(\vect{T},\vect{f},Z,\varepsilon) = \inf\Big\{ \sum_{x \in F}\e^{S_{n}^{\vect{T}}\vect{f}(x)}: F\ \text{is a}\ (n,\varepsilon)\text{-spanning set for}\ Z\  \Big\}; \\
  && \label{eq:defPn}
        P_{n}(\vect{T},\vect{f},Z,\varepsilon) = \sup\Big\{ \sum_{x \in E}\e^{S_{n}^{\vect{T}}\vect{f}(x)}: E\ \text{is a}\ (n,\varepsilon)\text{-separated set for}\ Z\  \Big\}.
      \end{eqnarray}
 We  consider both lower and upper limits and write
      \begin{equation}\label{eq:Qe}
        \begin{aligned}
          \qrelow(\vect{T},\vect{f},Z,\varepsilon) &= \varliminf_{n \to \infty}\frac{1}{n}\log{Q_{n}(\vect{T},\vect{f},Z,\varepsilon)}, \\
          \qreup(\vect{T},\vect{f},Z,\varepsilon) &= \varlimsup_{n \to \infty}\frac{1}{n}\log{Q_{n}(\vect{T},\vect{f},Z,\varepsilon)},
        \end{aligned}
      \end{equation}
      \begin{equation}\label{eq:Pe}
        \begin{aligned}
          \prelow(\vect{T},\vect{f},Z,\varepsilon) &= \varliminf_{n \to \infty}\frac{1}{n}\log{P_{n}(\vect{T},\vect{f},Z,\varepsilon)}, \\
          \preup(\vect{T},\vect{f},Z,\varepsilon) &= \varlimsup_{n \to \infty}\frac{1}{n}\log{P_{n}(\vect{T},\vect{f},Z,\varepsilon)}.
        \end{aligned}
      \end{equation}
     Since $\qrelow(\vect{T},\vect{f},Z,\varepsilon)$, $\qreup(\vect{T},\vect{f},Z,\varepsilon)$, $\prelow(\vect{T},\vect{f},Z,\varepsilon)$, and $\preup(\vect{T},\vect{f},Z,\varepsilon)$ are decreasing in $\varepsilon$, the following limits exist,
      \begin{eqnarray}
&& \label{eq:defqre}        \qrelow(\vect{T},\vect{f},Z) = \lim_{\varepsilon \to 0}\qrelow(\vect{T},\vect{f},Z,\varepsilon),
        \quad
        \qreup(\vect{T},\vect{f},Z) = \lim_{\varepsilon \to 0}\qreup(\vect{T},\vect{f},Z,\varepsilon),       \\
&& \label{eq:defpre}         \prelow(\vect{T},\vect{f},Z) = \lim_{\varepsilon \to 0}\prelow(\vect{T},\vect{f},Z,\varepsilon),
        \quad
        \preup(\vect{T},\vect{f},Z) = \lim_{\varepsilon \to 0}\preup(\vect{T},\vect{f},Z,\varepsilon).
      \end{eqnarray}

      \begin{defn}
        Given  a subset $Z \subset X_{0}$ and $\vect{f} \in \vect{C}(\vect{X},\R)$, we call
        $\qrelow(\vect{T},\vect{f},Z)$ and $\qreup(\vect{T},\vect{f},Z)$ the \emph{lower} and \emph{upper spanning topological pressures of $\vect{T}$ for $\vect{f}$ on $Z$}, respectively;
        and we call $\prelow(\vect{T},\vect{f},Z)$ and $\preup(\vect{T},\vect{f},Z)$ the \emph{lower} and \emph{upper separated topological pressures of $\vect{T}$ for $\vect{f}$ on $Z$}, respectively.

        If  $\qrelow(\vect{T},\vect{f},Z)=\prelow(\vect{T},\vect{f},Z)$, we call it  the \emph{lower topological pressure of $\vect{T}$ for $\vect{f}$ on $Z$} and denote it  by $\preL(\vect{T},\vect{f},Z)$.
        Similarly, if   $\qreup(\vect{T},\vect{f},Z)=\preup(\vect{T},\vect{f},Z)$, we call it the \emph{  upper topological pressure of $\vect{T}$ for $\vect{f}$ on $Z$} and denote it by $\preU(\vect{T},\vect{f},Z)$.

The \emph{lower} and \emph{upper topological entropies of $\vect{T}$ on $Z$} are  respectively defined as 
        $$
        \enttoplow(\vect{T},Z)=\preL(\vect{T},\vect{0},Z) \qquad \textit{and \qquad}
        \enttopup(\vect{T},Z)=\preU(\vect{T},\vect{0},Z).
        $$

      \end{defn}

      The following result was given in \cite[Prop.2.2]{CM1}.
      \begin{prop}\label{prop_ExPUPL}
        Given   $(\vect{X},\vect{T})$ and $Z \subseteq X_{0}$. For all equicontinuous $\vect{f} \in \vect{C}(\vect{X},\R)$, the lower and upper topological pressures  $\preL(\vect{T},\vect{f},Z)$ and $\preU(\vect{T},\vect{f},Z)$ exist, that is, 
\begin{align*}
& \preU(\vect{T},\vect{f},Z)=\preup(\vect{T},\vect{f},Z)=\qreup (\vect{T},\vect{f},Z), \\ 
& \preL(\vect{T},\vect{f},Z)=\prelow(\vect{T},\vect{f},Z)=\qrelow (\vect{T},\vect{f},Z).
\end{align*}

      \end{prop}

      Another method to define the upper and lower pressures for equicontinuous $\vect{f} \in \vect{C}(\vect{X},\R)$ is by open covers. 
    Given a sequence $\vect{\mathscr{U}}=\{\mathscr{U}_{k}\}_{k=0}^{\infty}$ of open covers $\mathscr{U}_{k}$ of $X_{k}$,
    for all integers $k \geq 0$ and $n \geq 1$,
    we write $\vect{\mathscr{U}}_{k}^{n}$ for the set of all strings $\mathbf{U}$ of length $n=\len{\mathbf{U}}$ at level $k$,
    i.e.,
    $$
    \vect{\mathscr{U}}_{k}^{n} = \{\mathbf{U}=U_{k}U_{k+1} \cdots U_{k+n-1}: U_{j} \in \mathscr{U}_{j}, j = k,\ldots,k+n-1\},
    $$
    and for every $\mathbf{U} \in \vect{\mathscr{U}}_{k}^{n}$,
    \begin{equation}\label{def_XkbU}
    X_{k}[\mathbf{U}] = \bigcap_{j=0}^{n-1}\vect{T}_{k}^{-j}U_{k+j}
    = \{x \in X_{k}: \vect{T}_{k}^{j}x \in U_{k+j}, j=0,\dots,n-1\}.
    \end{equation}
   For every $\mathbf{U} \in \vect{\mathscr{U}}_{k}^{n}$, we write
    \begin{equation}\label{eq:SknTfU}
      \underline{S}_{k,n}^{\vect{T}}\vect{f}(\mathbf{U})=\inf_{x \in X_{k}[\mathbf{U}]}S_{k,n}^{\vect{T}}\vect{f}(x),
      \quad\text{and}\quad
      \overline{S}_{k,n}^{\vect{T}}\vect{f}(\mathbf{U})=\sup_{x \in X_{k}[\mathbf{U}]}S_{k,n}^{\vect{T}}\vect{f}(x);
    \end{equation}
   where $\underline{S}_{k,n}^{\vect{T}}\vect{f}(\mathbf{U})=\overline{S}_{k,n}^{\vect{T}}\vect{f}(\mathbf{U})=-\infty$ if $X_{k}[\mathbf{U}] = \emptyset$. 
    Write
    $$
    \vee_{k,n}^{\vect{T}}\vect{\mathscr{U}} = \bigvee_{j=0}^{n-1}\vect{T}_{k}^{-j}\mathscr{U}_{j+k} = \{X_{k}[\mathbf{U}]: \mathbf{U} \in \vect{\mathscr{U}}_{k}^{n}\}.
    $$
    We say that $\Gamma \subset \vect{\mathscr{U}}_{0}^{n}$ \emph{covers} $Z \subset X_{0}$ if $Z \subset \bigcup_{\mathbf{U} \in \Gamma}X_{0}[\mathbf{U}]$.
    For simplicity, we write $\vee_{n}^{\vect{T}}\vect{\mathscr{U}}=\vee_{0,n}^{\vect{T}}\vect{\mathscr{U}}$, $\underline{S}_{n}^{\vect{T}}\vect{f}(\mathbf{U})=\underline{S}_{0,n}^{\vect{T}}\vect{f}(\mathbf{U})$ and $ \overline{S}_{n}^{\vect{T}}\vect{f}(\mathbf{U})= \overline{S}_{0,n}^{\vect{T}}\vect{f}(\mathbf{U})$.    
    
     Similar to \eqref{eq:defQn} and \eqref{eq:defPn}, we define
    \begin{equation}\label{eq:QnPncov}
      \begin{split}
        Q_{n}(\vect{T},\vect{f},Z,\vect{\mathscr{U}}) &= \inf\Big\{\sum_{\mathbf{U} \in \Gamma}\exp{(\underline{S}_{n}^{\vect{T}}\vect{f}(\mathbf{U}))}: \Gamma \subset \vect{\mathscr{U}}_{0}^{n}\ \text{covers}\ Z\Big\}, \\
        P_{n}(\vect{T},\vect{f},Z,\vect{\mathscr{U}}) &= \inf\Big\{\sum_{\mathbf{U} \in \Gamma}\exp{(\overline{S}_{n}^{\vect{T}}\vect{f}(\mathbf{U}))}: \Gamma \subset \vect{\mathscr{U}}_{0}^{n}\ \text{covers}\ Z\Big\}.
      \end{split}
    \end{equation}
We write 
\begin{equation}\label{def_QQPPCV}
    \begin{split}
   &   \qrelow(\vect{T},\vect{f},Z,\vect{\mathscr{U}})=\varliminf_{n \to \infty}\frac{1}{n}\log{Q_{n}(\vect{T},\vect{f},Z,\vect{\mathscr{U}})}, \\
   &   \qreup(\vect{T},\vect{f},Z,\vect{\mathscr{U}})=\varlimsup_{n \to \infty}\frac{1}{n}\log{Q_{n}(\vect{T},\vect{f},Z,\vect{\mathscr{U}})},\\
   & \prelow(\vect{T},\vect{f},Z,\vect{\mathscr{U}})=\varliminf_{n \to \infty}\frac{1}{n}\log{P_{n}(\vect{T},\vect{f},Z,\vect{\mathscr{U}})}, \\
   & \preup(\vect{T},\vect{f},Z,\vect{\mathscr{U}})=\varlimsup_{n \to \infty}\frac{1}{n}\log{P_{n}(\vect{T},\vect{f},Z,\vect{\mathscr{U}})}.
    \end{split}
\end{equation}

    Given a sequence $\vect{\mathscr{U}}$ of open covers $\mathscr{U}_{k}$ of $X_{k}$,
    we write
    $$
    \diam(\vect{\mathscr{U}}) = \sup_{k \in \N}\sup\{\diam(U): U \in \mathscr{U}_{k}\}.
    $$
The following conclusion is  proved in  \cite[Prop.3.3 \& Prop.3.4]{CM1}.
    \begin{prop}\label{prop:qrecov}
 Given  $(\vect{X},\vect{T})$ and $Z \subset X_{0}$. If $\vect{f} \in \vect{C}(\vect{X},\R)$ is equicontinuous,  
 \begin{equation*}
 \begin{split}
   & \preL(\vect{T},\vect{f},Z) = \sup_{\vect{\mathscr{U}}}\qrelow(\vect{T},\vect{f},Z,\vect{\mathscr{U}}) = \lim_{\diam(\vect{\mathscr{U}}) \to 0}\qrelow(\vect{T},\vect{f},Z,\vect{\mathscr{U}})   = \lim_{\diam(\vect{\mathscr{U}}) \to 0}\prelow(\vect{T},\vect{f},Z,\vect{\mathscr{U}})\\ 
    &  \preU(\vect{T},\vect{f},Z) = \sup_{\vect{\mathscr{U}}}\qreup(\vect{T},\vect{f},Z,\vect{\mathscr{U}}) = \lim_{\diam(\vect{\mathscr{U}}) \to 0}\qreup(\vect{T},\vect{f},Z,\vect{\mathscr{U}})= \lim_{\diam(\vect{\mathscr{U}}) \to 0}\preup(\vect{T},\vect{f},Z,\vect{\mathscr{U}}),
  \end{split}
  \end{equation*}
      where $\vect{\mathscr{U}}$ ranges over all sequences of open covers of $X_{k}$ with a Lebesgue number.
    \end{prop}
  
    Given $\Omega \subseteq \Sigma_{k}^{\infty}$, write $\Omega_{k}^{l}=\{\mathbf{u} \in \Sigma_{k}^{l}: [\mathbf{u}]_{k} \cap \Omega \neq \varnothing\}=\{\omega\vert{(l-k+1)}: \omega \in \Omega\}$ for the collection of all $\Omega$-admissible strings.
    By Propositions~\ref{prop:qrecov}, we immediately have the following formulae for lower and upper capacity pressures in symbolic systems.
    \begin{prop} \label{eq:PlowupOmega}
      Given $\Omega \subseteq \Sigma_{0}^{\infty}$, for every integer $n \geq 1$ and $\mathbf{u} \in \Omega_{0}^{n-1}$, let $\omega^{\mathbf{u}} \in [\mathbf{u}]$.
      Then for all equicontinuous $\vect{f} \in \vect{C}(\vect{\Sigma}(\vect{m}),\R)$,
      \begin{equation*}
        \begin{aligned}
          \preL(\vect{\sigma},\vect{f},\Omega)=\varliminf_{n \to \infty}\frac{1}{n}\log{\sum_{\mathbf{u} \in \Omega_{0}^{n-1}}\exp{\left(S_{n}^{\vect{\sigma}}\vect{f}(\omega^{\mathbf{u}})\right)}}, \\
          \preU(\vect{\sigma},\vect{f},\Omega)=\varlimsup_{n \to \infty}\frac{1}{n}\log{\sum_{\mathbf{u} \in \Omega_{0}^{n-1}}\exp{\left(S_{n}^{\vect{\sigma}}\vect{f}(\omega^{\mathbf{u}})\right)}},
        \end{aligned}
      \end{equation*}
      where both lower and upper limits do not depend on the $\omega^{\mathbf{u}} \in [\mathbf{u}]$ chosen.
    \end{prop}

  \subsection{Bowen pressures and entropies}\label{ssect:prebppetentbow}
      In this subsection, we define the Bowen pressures on nonautonomous dynamical systems.

      Given  a subset $Z \subset X_{0}$, real $N>0$ and  real $\varepsilon>0$, we say that a collection $\{B_{n_{i}}^{\vect{T}}(x_{i},\varepsilon)\}_{i\in \mathcal{I}}$ of Bowen balls is a \emph{$(N,\varepsilon)$-cover} of $Z$
      if $\bigcup_{i\in \mathcal{I}}B_{n_{i}}^{\vect{T}}(x_{i},\varepsilon) \supset Z$ where $n_{i} \geq N$ for each $i \in \mathcal{I}$. 
      Given $\vect{f} \in \vect{C}(\vect{X},\mathbb{R})$ and $s \in \mathbb{R}$, for  reals $N >0$ and $\varepsilon>0$, we define
      \begin{equation}\label{eq:defmsrbow}
        \msrbow_{N,\varepsilon}^{s}(\vect{T},\vect{f},Z)=\inf\Big\{\sum_{i=1}^{\infty} \exp{\big(-n_{i}s+S_{n_{i}}^{\vect{T}}\vect{f}(x_{i})\big)}  \Big\},
      \end{equation}
      where the infimum is taken over all countable $(N,\varepsilon)$-covers $\{B_{n_{i}}^{\vect{T}}(x_{i},\varepsilon)\}_{i=1}^\infty$ of $Z$. We write
      $$
      \msrbow_{\varepsilon}^{s}(\vect{T},\vect{f},Z)=\lim_{N \to \infty} \msrbow_{N,\varepsilon}^{s}(\vect{T},\vect{f},Z).
      $$
and
      \begin{equation}\label{def_PBep}
        \prebow(\vect{T},\vect{f},Z,\varepsilon)=\inf\{s: \msrbow_{\varepsilon}^{s}(\vect{T},\vect{f},Z)=0\}
                                                =\sup\{s: \msrbow_{\varepsilon}^{s}(\vect{T},\vect{f},Z)=+\infty\}.
      \end{equation}
      Since $\msrbow_{N,\varepsilon}^{s}$ is monotone in $\varepsilon$, so are $\msrbow_{\varepsilon}^{s}$ and $\prebow(\vect{T},\vect{f},Z,\varepsilon)$.
      \begin{defn}\label{def:prebpp}
        Given  $\vect{f} \in \vect{C}(\vect{X},\mathbb{R})$  and $Z \subset X_{0}$, we call
        $$
        \prebow(\vect{T},\vect{f},Z)=\lim_{\varepsilon \to 0}\prebow(\vect{T},\vect{f},Z,\varepsilon)
        $$
        the \emph{Bowen-Pesin-Pitskel' topological pressure} (\emph{Bowen pressure} for short) \emph{of $\vect{T}$ for $\vect{f}$ on $Z$}. The \emph{Bowen topological entropy} (\emph{Bowen entropy} )  \emph{of $\vect{T}$  on $Z$} is defined as 
        $$
        \enttopbow(\vect{T},Z)=\prebow(\vect{T},\vect{0},Z). 
        $$
        
      \end{defn}

      Bowen pressures for equicontinuous $\vect{f} \in \vect{C}(\vect{X},\R)$ may also be calculated using open covers.    Given a sequence $\vect{\mathscr{U}}=\{\mathscr{U}_{k}\}_{k=0}^{\infty}$ of open covers $\mathscr{U}_{k}$ of $X_{k}$, recall that
    $$
    \vect{\mathscr{U}}_{0}^{n} = \{\mathbf{U}=U_{0}U_{1} \cdots U_{n-1}: U_{j} \in \mathscr{U}_{j}, j = 0,\ldots,n-1\}
    $$
 and $\len{\mathbf{U}}=n $ is the   length of the string  $\mathbf{U}\in \vect{\mathscr{U}}_0^{n}$.     We say that $\Gamma \subset \cup_{n=0}^\infty \vect{\mathscr{U}}_{0}^{n}$ \emph{covers} $Z \subset X_{0}$ if $Z \subset \bigcup_{\mathbf{U} \in \Gamma}X_{0}[\mathbf{U}]$ where $X_{0}[\mathbf{U}]$ is defined by \eqref{def_XkbU}.

Given  $(\vect{X},\vect{T})$, $\vect{f} \in \vect{C}(\vect{X},\R)$ and a sequence $\vect{\mathscr{U}}$ of open covers of $X_{k}$,
    for each $s \in \R$ and $N>0$, we define the measures $\msrbpplow_{N}^{s}(\vect{T},\vect{f},\cdot,\vect{\mathscr{U}})$ and $\msrbppup_{N}^{s}(\vect{T},\vect{f},\cdot,\vect{\mathscr{U}})$.
    For simplicity, we write $\msrbpp$ for one of $\{\msrbpplow,\msrbppup\}$ and correspondingly $S$ for one of $\{\underline{S},\overline{S}\}$ as in \eqref{eq:SknTfU}.
    The measures are given by
    \begin{equation}\label{eq:defmsrbpp}
      \msrbpp_{N}^{s}(\vect{T},\vect{f},Z,\vect{\mathscr{U}}) = \inf_{\Gamma}\Big\{\sum_{\mathbf{U} \in \Gamma}\exp{\Big(-\len{\mathbf{U}}s + S_{\len{\mathbf{U}}}^{\vect{T}}\vect{f}(\mathbf{U})\Big)}\Big\},
    \end{equation}
    where the infimum is taken over all countable covers  $\Gamma$ of  $Z$ satisfying that $\len{\mathbf{U}} \geq N$ for every $\mathbf{U} \in \Gamma$.
    Clearly $\msrbpp_{N}^{s}(\vect{T},\vect{f},Z,\vect{\mathscr{U}})$ is non-decreasing as $N$ tends to $\infty$ for every given $Z$,
    and we write
    $$
    \msrbpp^{s}(\vect{T},\vect{f},Z,\vect{\mathscr{U}}) = \lim_{N \to \infty}\msrbpp_{N}^{s}(\vect{T},\vect{f},Z,\vect{\mathscr{U}}).
    $$
    Similar dimension structures are given by the critical values in $s$, denoted by
    \begin{equation}\label{def_PBPPep}
      \begin{aligned}
        \prebpp(\vect{T},\vect{f},Z,\vect{\mathscr{U}}) &= \sup\{s \in \R: \msrbpplow^{s}(\vect{T},\vect{f},Z,\vect{\mathscr{U}}) = +\infty\} \\
                                                        &= \inf\{s \in \R: \msrbpplow^{s}(\vect{T},\vect{f},Z,\vect{\mathscr{U}}) = 0\}, \\
        \prebow(\vect{T},\vect{f},Z,\vect{\mathscr{U}}) &= \sup\{s \in \R: \msrbppup^{s}(\vect{T},\vect{f},Z,\vect{\mathscr{U}}) = +\infty\} \\
                                                        &= \inf\{s \in \R: \msrbppup^{s}(\vect{T},\vect{f},Z,\vect{\mathscr{U}}) = 0\}.
      \end{aligned}
    \end{equation}

The following result provides an equivalent description of $\prebow$ for equicontinuous potentials.

    \begin{prop}\label{prop:prebppcov}
      Given  $Z \subset X_{0}$, if $\vect{f} \in \vect{C}(\vect{X},\R)$ is equicontinuous, then
      $$
      \prebow(\vect{T},\vect{f},Z) = \lim_{\diam(\vect{\mathscr{U}}) \to 0}\prebpp(\vect{T},\vect{f},Z,\vect{\mathscr{U}}) = \lim_{\diam(\vect{\mathscr{U}}) \to 0}\prebow(\vect{T},\vect{f},Z,\vect{\mathscr{U}}),
      $$
      where $\vect{\mathscr{U}}$ ranges over all sequences of open covers of $X_{k}$ with a Lebesgue number.
    \end{prop}
    \begin{proof}
      The proof is similar to  \cite[Prop.3.5]{CM1}, and we omit it.
    \end{proof}

  \subsection{Packing  pressures and entropies}\label{ssect:prepacetentpac}
      Given  a subset $Z \subset X_{0}$, we say that a collection $\{\overline{B}_{n_{i}}^{\vect{T}}(x_{i},\varepsilon)\}_{i\in \mathcal{I}}$ of closed Bowen balls is a \textit{$(N,\varepsilon)$-packing of $Z$} if $\{\overline{B}_{n_{i}}^{\vect{T}}(x_{i},\varepsilon)\}_{i\in \mathcal{I}}$ is disjoint where $x_{i} \in Z$ and $n_{i} \geq N$ for all $i \in \mathcal{I}$.
      Given $\vect{f} \in \vect{C}(\vect{X},\R)$ and $s \in \mathbb{R}$, for each $N>0$ and $\varepsilon>0$, we define
      \begin{equation}\label{eq:defmsrPack}
        \msrpac_{N,\varepsilon}^{s}(\vect{T},\vect{f},Z)=\sup\Big\{\sum_{i=1}^{\infty}\exp{\left(-n_{i}s+S_{n_{i}}^{\vect{T}}\vect{f}(x_{i})\right)} \Big\},
      \end{equation}
      where the supremum is taken over all countable $(N,\varepsilon)$-packings $\{\overline{B}_{n_{i}}^{\vect{T}}(x_{i},\varepsilon)\}_{i=1}^\infty$ of $Z$. We write
      $$
      \msrpac_{\infty,\varepsilon}^{s}(\vect{T},\vect{f},Z)=\lim_{N \to \infty} \msrpac_{N,\varepsilon}^{s}(\vect{T},\vect{f},Z).
      $$
      and 
      \begin{equation}\label{def_pes}
        \msrpac_{\varepsilon}^{s}(\vect{T},\vect{f},Z)=\inf\Big\{\sum_{i=1}^{\infty}\msrpac_{\infty,\varepsilon}^{s}(\vect{T},\vect{f},Z_{i}): \bigcup_{i=1}^{\infty}Z_{i} \supset Z\Big\}.
      \end{equation}
      Similarly, we denote the jump value of $s$  by
      \begin{equation}\label{def_PP}
        \prepac(\vect{T},\vect{f},Z,\varepsilon) =\inf\{s:\msrpac_{\varepsilon}^{s}(\vect{T},\vect{f},Z)=0\}
                                =\sup\{s: \msrpac_{\varepsilon}^{s}(\vect{T},\vect{f},Z)=+\infty\}.
      \end{equation}

      \begin{defn}\label{def:prepac}
        Given  $\vect{f} \in \vect{C}(\vect{X},\mathbb{R})$ and $Z \subset X_{0}$, we define the \emph{packing topological pressure} (\emph{packing pressure} for short) \emph{of $\vect{T}$ for $\vect{f}$ on  $Z$} by
        $$\prepac(\vect{T},\vect{f},Z)=\lim_{\varepsilon \to 0}\prepac(\vect{T},\vect{f},Z,\varepsilon).$$
        We define the \emph{packing topological entropy} (\emph{packing entropy} for short) \emph{of $\vect{T}$ on $Z$} as
        $$\enttoppac(\vect{T},Z)=\prepac(\vect{T},\vect{0},Z).$$
      \end{defn}

\section{Topological Pressures of Symbolic Dynamical Systems}\label{sect:toppresymsys}

  \subsection{Equivalent formulations of topological pressures}\label{ssect:preequivdefsymsys}
    In the context of symbolic dynamics, we present the pressures by constructing `measures' via cylinders.
    Given $\Omega \subseteq \Sigma_{0}^{\infty}$ and $\vect{f} \in \vect{C}(\vect{\Sigma}(\vect{m}),\R)$, for all $s \in \R$, we define
    \begin{eqnarray}      \label{eq:defQCs}
    && \carpeslow_{\mathscr{C}}^{s}(\vect{\sigma},\vect{f},\Omega) = \varliminf_{n \to \infty}\inf\Big\{\sum_{i=1}^{\infty}\exp{\left(-ns+S_{n}^{\vect{\sigma}}\vect{f}(\omega^{i})\right)}: \bigcup_{i=1}^{\infty}[\omega^{i}\vert{n}] \supseteq \Omega\Big\}, \\   
      &&\carpepup_{\mathscr{C}}^{s}(\vect{\sigma},\vect{f},\Omega) = \varlimsup_{n \to \infty}\sup\Big\{\sum_{i=1}^{\infty}\exp{\left(-ns+S_{n}^{\vect{\sigma}}\vect{f}(\omega^{i})\right)}: \\
    &&\hspace{5.5cm}\{[\omega^{i}\vert{n}]\}_{i=1}^{\infty}\ \text{is disjoint and}\ \omega^{i} \in \Omega\Big\}. \nonumber 
    \end{eqnarray}

We have the following equivalence.
\begin{prop}\label{prop_eqQP}
Given $\Omega \subseteq \Sigma_{0}^{\infty}$ and $\vect{f} \in \vect{C}(\vect{\Sigma}(\vect{m}),\R)$,
    \begin{equation}\label{eq:qrelowpreupcyl}
      \begin{aligned}
        \qrelow(\vect{\sigma},\vect{f},\Omega) &= \sup\{s:\carpeslow_{\mathscr{C}}^{s}(\vect{\sigma},\vect{f},\Omega)=+\infty\} = \inf\{s:\carpeslow_{\mathscr{C}}^{s}(\vect{\sigma},\vect{f},\Omega)=0\}, \\
        \preup(\vect{\sigma},\vect{f},\Omega) &= \sup\{s:\carpepup_{\mathscr{C}}^{s}(\vect{\sigma},\vect{f},\Omega)=+\infty\} = \inf\{s:\carpepup_{\mathscr{C}}^{s}(\vect{\sigma},\vect{f},\Omega)=0\}.
      \end{aligned}
    \end{equation}
  Furthermore, if $\vect{f}$ is equicontinuous or of strongly bounded variation, then $\preL(\vect{\sigma},\vect{f},\Omega)$ and $\preU(\vect{\sigma},\vect{f},\Omega)$ exist, and
  \begin{equation} \label{prop_eqPUPL}
    \begin{aligned}
      \preL(\vect{\sigma},\vect{f},\Omega) &= \sup\{s:\carpeslow_{\mathscr{C}}^{s}(\vect{\sigma},\vect{f},\Omega)=+\infty\} = \inf\{s:\carpeslow_{\mathscr{C}}^{s}(\vect{\sigma},\vect{f},\Omega)=0\}, \\
      \preU(\vect{\sigma},\vect{f},\Omega) &= \sup\{s:\carpepup_{\mathscr{C}}^{s}(\vect{\sigma},\vect{f},\Omega)=+\infty\} = \inf\{s:\carpepup_{\mathscr{C}}^{s}(\vect{\sigma},\vect{f},\Omega)=0\}.
    \end{aligned}
  \end{equation}
\end{prop}
\begin{proof}
Let $\carpeslow_{\varepsilon}^{s}$ and $\carpepup_{\varepsilon}^{s}$ be given by \cite[\S3.1]{CM1}. Since  for every $\varepsilon>0$,   
\begin{equation}\label{eq:carcyldcoincidence}
      \carpeslow_{\varepsilon}^{s}(\vect{\sigma},\vect{f},\Omega) = \carpeslow_{\mathscr{C}}^{s}(\vect{\sigma},\vect{f},\Omega), \quad
      \carpepup_{\varepsilon}^{s}(\vect{\sigma},\vect{f},\Omega) = \carpepup_{\mathscr{C}}^{s}(\vect{\sigma},\vect{f},\Omega),
\end{equation}
Equation \eqref{eq:qrelowpreupcyl} follows from \cite[Prop.3.1]{CM1}.

Furthermore, if $\vect{f}$ is   equicontinuous, the  conclusion \eqref{prop_eqPUPL} follows from Proposition \ref{prop_ExPUPL}.
It remains to prove the conclusion  \eqref{prop_eqPUPL} holds for $\vect{f}$ of strongly bounded variation,
and by  Proportion \ref{prop_ExPUPL}, it suffices to show that
\begin{equation}\label{eq:QgeqP}
\qrelow(\vect{\sigma},\vect{f},\Omega) \geq \prelow(\vect{\sigma},\vect{f},\Omega) \quad \text{and} \quad \qreup(\vect{\sigma},\vect{f},\Omega) \geq \preup(\vect{\sigma},\vect{f},\Omega).
\end{equation}
 
By \eqref{eq:strbndvar}, there exists a constant $b>0$ such that for all $n \geq 1$ and $\varepsilon>0$, we have
$$
\lvert S_{n}^{\vect{\sigma}}\vect{f}(\omega)-S_{n}^{\vect{\sigma}}\vect{f}(\vartheta) \rvert < \lvert S_{n}^{\vect{\sigma}}\vect{f}^{*}(\omega)-S_{n}^{\vect{\sigma}}\vect{f}_{*}(\vartheta) \rvert \leq b
$$
whenever $d_{n}^{\vect{\sigma}}(\omega,\vartheta) < \frac{\varepsilon}{2}$.

Arbitrarily choose a $(n,\varepsilon)$-separated set $E$ and a $(n,\frac{\varepsilon}{2})$-spanning set $F$ for $\Omega$. Since $E \subset \Omega$,  it is clear that  $F$ is also a $(n,\varepsilon)$ spanning for $E$.
We define a mapping $\phi: E \to F$ by choosing a point $\phi(\omega) \in F$ with $d_{n}^{\vect{T}}(\omega,\phi(\omega))<\frac{\varepsilon}{2}$ for each $\omega \in E$.
The mapping $\phi$ is injective since $E$ is $(n,\varepsilon)$ separated. It follows that
\begin{align*}
  \sum_{\vartheta \in F}\e^{S_{n}^{\vect{\sigma}}\vect{f}(\vartheta)} & \geq \sum_{\vartheta \in \phi(E)}\e^{S_{n}^{\vect{\sigma}}\vect{f}(\vartheta)}  = \sum_{\omega \in E}\e^{S_{n}^{\vect{\sigma}}\vect{f}(\phi(\omega))} \geq \e^{-b}\sum_{\omega \in E}\e^{S_{n}^{\vect{\sigma}}\vect{f}(\omega)},
\end{align*}
and by equations \eqref{eq:defQn} and \eqref{eq:defPn}, we have  that for all $n \geq 1$ and $\varepsilon>0$,
$$
Q_{n}(\vect{\sigma},\vect{f},\Omega,\frac{\varepsilon}{2}) \geq \e^{-b}P_{n}(\vect{\sigma},\vect{f},\Omega,\varepsilon).
$$
Combining this with   \eqref{eq:Qe},\eqref{eq:Pe},\eqref{eq:defqre} and \eqref{eq:defpre},  the inequality \eqref{eq:QgeqP} follows, 
 and we obtain the existence of $\preL$ and $\preU$.
\end{proof}

    The above generalizations of classic pressures are restricted to using cylinders of the same rank $n$ (also the number of iterations of the shifts).
    By allowing cylinders of different ranks, we obtain the Bowen and packing pressures in symbolic dynamical systems.

  Given $\Omega \subseteq \Sigma_{0}^{\infty}$ and $\vect{f} \in \vect{C}(\vect{\Sigma}(\vect{m}),\R)$,
  for all $s \in \mathbb{R}$, we define two types of Hausdorff measures, namely,
  \begin{equation}
    \begin{aligned}
      \msrbow_{\mathscr{C}}^{s}(\vect{f},\Omega)=\lim_{N \to \infty}\inf\Big\{\sum_{i=1}^{\infty}\exp{\Big(-n_{i}s+S_{n_{i}}^{\vect{\sigma}}\vect{f}(\omega^{i})\Big)}:
                                                                               \bigcup_{i=1}^{\infty}[\omega^{i}\vert{n_{i}}] \supseteq Z, n_{i} \geq N\Big\}
    \end{aligned}
  \end{equation}
  and
  \begin{equation} \label{def_mcpb}
    \begin{aligned}
&      \msrbppup_{\mathscr{C}}^{s}(\vect{f},\Omega)=\lim_{N \to \infty}\inf\Big\{\sum_{\bv\in \mathscr{U}}\exp{\Big(-n(\bv)s+\sup_{\omega \in [\bv]}S_{n(\bv)}^{\vect{\sigma}}\vect{f}(\omega)\Big)}\Big\}, \\
&   \underline{\msrbpp}_{\mathscr{C}}^{s}(\vect{f},\Omega) =\lim_{N \to \infty}\inf\Big\{\sum_{\bv\in \mathscr{U}}\exp{\Big(-n(\bv)s+\inf_{\omega \in [\bv]}S_{n(\bv)}^{\vect{\sigma}}\vect{f}(\omega)\Big)}\Big\},
    \end{aligned}
  \end{equation}
  where the infimum is taken over all countable covers $\mathscr{U}$ of $\Omega$ consisting of  cylinders  of rank $n(\bv) \geq N$.
We have the following equivalence for  the Bowen pressure of the nonautonomous shift $\vect{\sigma}$ for $\vect{f}$ on $\Omega$. 
\begin{prop}\label{prop_eqpBE}
Given $\Omega \subseteq \Sigma_{0}^{\infty}$ and $\vect{f} \in \vect{C}(\vect{\Sigma}(\vect{m}),\R)$,
$$ 
      \prebow(\vect{\sigma},\vect{f},\Omega)
      =\sup\{s: \msrbow_{\mathscr{C}}^{s}(\vect{f},\Omega)=+\infty\}=\inf\{s: \msrbow_{\mathscr{C}}^{s}(\vect{f},\Omega)=0\}.
$$  
Furthermore, if  $\vect{f}$ is equicontinuous, then
\begin{eqnarray*}
\prebow(\vect{\sigma},\vect{f},\Omega)
&=&\sup\{s: \msrbppup_{\mathscr{C}}^{s}(\vect{f},\Omega)=+\infty\}=\inf\{s: \msrbppup_{\mathscr{C}}^{s}(\vect{f},\Omega)=0\} \\
&=&\sup\{s: \underline{\msrbpp}_{\mathscr{C}}^{s}(\vect{f},\Omega) =+\infty\}=\inf\{s: \underline{\msrbpp}_{\mathscr{C}}^{s}(\vect{f},\Omega) =0\}.
\end{eqnarray*}
\end{prop}
  \begin{proof} 
Let $\msrbppup_{\varepsilon}^{s}$ denote the $\msrbpp_{\varepsilon}^{s}$ defined in \cite[\S3.2]{CM2}. Since for every $\varepsilon>0$,  
    \begin{equation}\label{eq:msrcyldcoincidence}
      \begin{aligned}
        \msrbow_{\varepsilon}^{s}(\vect{\sigma},\vect{f},\Omega) = \msrbow_{\mathscr{C}}^{s}(\vect{f},\Omega), \\
        \msrbppup_{\varepsilon}^{s}(\vect{\sigma},\vect{f},\Omega) = \msrbppup_{\mathscr{C}}^{s}(\vect{f},\Omega),
      \end{aligned}
    \end{equation}
    the conclusion follows by  Definition~\ref{def:prebpp} and \cite[Prop.3.3]{CM2}.
  \end{proof}

  Note that Proposition \ref{prop_eqQP} and Proposition \ref{prop_eqpBE} extend the results in Example~\ref{exmp:symsuepre}.
  Moreover,
  \begin{equation*}
    \begin{aligned}
      \msrbow_{\mathscr{C}}^{s}(\vect{f},\Omega) = \msrbow^{s}(\vect{\sigma},\vect{f},\Omega,\vect{\mathscr{C}}), \\
      \msrbppup_{\mathscr{C}}^{s}(\vect{f},\Omega) = \msrbppup^{s}(\vect{\sigma},\vect{f},\Omega,\vect{\mathscr{C}}).
    \end{aligned}
  \end{equation*}
  where $\vect{\mathscr{C}}=\{[\mathbf{u}]_{k}:\mathbf{u} \in \Sigma_{k}^{k}(\vect{m})]\}_{k=0}^{\infty}$.

 Similarly, we may define packing contents $\msrpac_{N,\mathscr{C}}^{s}$ via cylinders so that for all $\varepsilon>0$,
    $$
    \begin{aligned}
    \msrpac_{\infty,\varepsilon}^{s}(\vect{\sigma},\vect{f},\Omega) = \lim_{N \to \infty}\msrpac_{N,\mathscr{C}}^{s}(\vect{f},\Omega),
    \end{aligned}
    $$
    and the same procedures to define packing measures $\msrpac_{\mathscr{C}}^{s}$ gives that for every $\varepsilon>0$,
    $$
    \msrpac_{\varepsilon}^{s}(\vect{\sigma},\vect{f},\Omega) = \msrpac_{\mathscr{C}}^{s}(\vect{\sigma},\vect{f},\Omega).
    $$

    By Definition~\ref{def:prepac}, we have the following equivalence for packing pressures.
\begin{prop}
  Given $\Omega \subseteq \Sigma_{0}^{\infty}$ and $\vect{f} \in \vect{C}(\vect{\Sigma}(\vect{m}),\R)$,
    $$
      \prepac(\vect{\sigma},\vect{f},\Omega)
      =\sup\{s: \msrpac_{\mathscr{C}}^{s}(\vect{f},\Omega)=+\infty\}=\inf\{s: \msrpac_{\mathscr{C}}^{s}(\vect{f},\Omega)=0\}.
    $$
\end{prop}
    This indicates that the sequence $\vect{\mathscr{C}}=\{[\bu]_{k}: \bu\in \Sigma_k^k\}_{k=0}^{\infty}$ forms in some sense a generator for $\prepac$.

  \subsection{Topological pressures on open sets}\label{ssect:toppreop}
  In this subsection, we study the behavior of the topological pressures on open subsets of $\Sigma_{0}^{\infty}$.

  We require the following lemma, which is a particular case of Theorem~\ref{thm:POeqPS}.
  \begin{lem}\label{lem:PupC}
Given $P\in  \{\prebow,\prepac, \preU\}$, let $\vect{f} \in \vect{C}(\vect{\Sigma}(\vect{m}),\R)$ be  equicontinuous. Then   for all  $\bu \in \Sigma_0^*$,    
    $$
    P(\vect{\sigma},\vect{f},[\bu]) = P(\vect{\sigma},\vect{f},\Sigma_{0}^{\infty}).
    $$
  \end{lem}
  \begin{proof}
    We only provide  the proof for the upper pressure since the  argument for Bowen and packing pressures is  identical.

Since $\vect{f}$ is  equicontinuous, by Proposition \ref{prop_ExPUPL}, we have $\preU(\vect{\sigma},\vect{f},[\bu])=\preup(\vect{\sigma},\vect{f},[\bu])$. 
Since $\Sigma_{0}^{\infty}(\vect{m})=\bigcup_{\bu \in \Sigma_{0}^{l}} [\bu]$ for all $l \geq 0$,     by the finite stability of $\preup$,
    it suffices to show 
    \begin{equation}\label{eq:PupCeqD}
      \preup(\vect{\sigma},\vect{f},[\bu]) = \preup(\vect{\sigma},\vect{f},[\bv]),  \qquad \textit{  for all  $\bu, \bv\in \Sigma_0^{l}$.}
    \end{equation}

Fix $\bu, \bv\in \Sigma_0^{l}$.  For each  $k \in \N$, let $X_{k}=\vect{\sigma}^{k}([\bu]) \subseteq \Sigma_{k}^{\infty}$ and $Y_{k}=\vect{\sigma}^{k}([\bv]) \subseteq \Sigma_{k}^{\infty}$. It is clear that     
    $$
    X_{k}=\{u_{k} \ldots u_{l}\omega: \omega \in \Sigma_{l+1}^{\infty}(\vect{m})\}
 \quad \textit{ and } \quad
    Y_{k}=\{v_{k} \ldots v_{l}\omega: \omega\in \Sigma_{l+1}^{\infty}(\vect{m})\},
    $$
for every $0 \leq k \leq l$ and $X_{k}=Y_{k}=\Sigma_{k}^{\infty}$ for all $k \geq l+1$. Let $T_{k}=\sigma_{k}\vert_{X_{k}}$ and $R_{k}=\sigma_{k}\vert{Y_{k}}$. Thus both $(\vect{X},\vect{T})$  and  $(\vect{Y},\vect{R})$  are  NDSs.  We endow $\vect{X}$  and $\vect{Y}$ with potentials $\vect{f}\vert_{\vect{X}}=\{f_{k}\vert_{X_{k}}\}_{k=0}^{\infty}$ and $\vect{f}\vert_{\vect{Y}}=\{f_{k}\vert_{Y_{k}}\}_{k=0}^{\infty}$, respectively.

To show  \eqref{eq:PupCeqD}, it is equivalent to show
    $$
    \preup(\vect{T},\vect{f}\vert_{\vect{X}},X_{0}) = \preup(\vect{R},\vect{f}\vert_{\vect{Y}},Y_{0}).
    $$
 For every $0 \leq k \leq l$,  let $\pi_{k}: X_k \to Y_k$ be given by
    $$
    \pi_{k}(u_{k} \ldots u_{l}\omega)=(v_{k} \ldots v_{l}\omega)\quad(\omega \in \Sigma_{l+1}^{\infty}(\vect{m})),
    $$
 and for $k \geq l+1$,  let $\pi_{k}=\id_{\Sigma_{k}^{\infty}}$ be the identity mapping on $\Sigma_{k}^{\infty}$.
Then  $\vect{\pi}=\{\pi_{k}\}_{k=0}^{\infty}$ is an equiconjugacy from $(\vect{X},\vect{T})$ to $(\vect{Y},\vect{R})$.
    Since $f_{k}\vert_{X_{k}}=f_{k}\vert_{Y_{k}}$ for all $k \geq l+1$ and $\vect{f}$ is equicontinuous,  by \cite[Thm.7.1 \& Thm.5.7]{CM1}, we have  $\preup(\vect{T},\vect{f}\vert_{\vect{X}},X_{0}) = \preup(\vect{R},\vect{f}\vert_{\vect{Y}},Y_{0})$.
  \end{proof}

Finally, we show that  Theorem~\ref{thm:POeqPS} is a consequence of  Lemma~\ref{lem:PupC}.
  \begin{proof}[Proof of Theorem~\ref{thm:POeqPS}]
We only provide the proof for $\preU$ since others are similar. Since $\Omega \subseteq \Sigma_{0}^{\infty}$, it follows that 
    $$
    \preU(\vect{\sigma},\vect{f},\Sigma_{0}^{\infty}) \geq \preU(\vect{\sigma},\vect{f},\Omega).
    $$
    
    Suppose that $\omega \in \Omega$ is an interior point. Then there exists $N \in \N$ such that  $[\omega\vert{n}] \subseteq \Omega$ for all $n \geq N$.
    By Lemma~\ref{lem:PupC}, we have
    \begin{align*}
      \preU(\vect{\sigma},\vect{f},\Sigma_{0}^{\infty}) = \preU(\vect{\sigma},\vect{f},[\omega\vert{n}]) \leq \preU(\vect{\sigma},\vect{f},\Omega),
    \end{align*}
 and the conclusion holds.
  \end{proof}

  \begin{proof}[Proof of Corollary~\ref{cor:PupeqPP}]
Since $\Omega\subset \Sigma_{0}^{\infty}$ is non-empty and open, by Corollary~\ref{cor:PBPupopen}, for all open $V \subseteq \Sigma_{0}^{\infty}$ with $\Omega \cap V \neq \varnothing$,
    $$
    \preU(\vect{\sigma},\vect{f},\Omega\cap V) = \preU(\vect{\sigma},\vect{f},\Omega)=\preU(\vect{\sigma},\vect{f},\Sigma_{0}^{\infty}).
    $$
Since $\Omega$ is also compact and $\vect{f}$ is equicontinuous, by Proposition \ref{prop_ExPUPL} and \cite[Cor.4.10]{CM1}, we have that  $\preU(\vect{\sigma},\vect{f},\Omega)=\prepac(\vect{\sigma},\vect{f},\Omega)$. Hence the conclusion follows.
  \end{proof}

  \subsection{Pressures of potentials with strongly bounded variation}\label{ssect:estimate}

  Given an equicontinuous $\vect{f} \in \vect{C}(\vect{\Sigma}(\vect{m}),\R)$,  
  let $a_{j,i}\in \big[ \inf_{\vartheta \in [i]_{j}}f_{j}(\vartheta),\sup_{\vartheta \in [i]_{j}}f_{j}(\vartheta) \big]$ for $j \geq 1$ and $1 \leq i \leq m_{j}$  and
  $$
  s_{n}=\frac{1}{n}\sum_{j=0}^{n-1}\log{\Big(\sum_{i=1}^{m_{j}}\e^{a_{j,i}}\Big)},
  $$
  and we write
  \begin{equation}\label{def_s_su}
    \underline{s}=\varliminf_{n \to \infty}s_{n}
    \quad \text{and} \quad
    \overline{s}=\varlimsup_{n \to \infty}s_{n}.
  \end{equation}
  It  immediately follows by Proposition \ref{eq:PlowupOmega} that
$$
  \inf\underline{s} \leq \prelow(\vect{\sigma},\vect{f},\Sigma_{0}^{\infty}) \leq \preup(\vect{\sigma},\vect{f},\Sigma_{0}^{\infty}) \leq \sup\overline{s},
  $$
  where the infimum and supremum are taken over all the possible choices of $a_{j,i} $.

To show Theorem~\ref{thm:Psymsys}, it is sufficient to prove that $\underline{s}$ and $\overline{s}$ are irrelevant to the choice of $a_{j,i}$  for potentials $\vect{f}$ satisfying \eqref{eq:strbndvar}. 
Recall  that $\vect{f}_{*}=\{f_{k,*}\}_{k=1}^{\infty}$, $\vect{f}^{*}=\{f_{k}^{*}\}_{k=1}^{\infty} \in \vect{C}(\vect{\Sigma}(\vect{m}),\R)$, 
where $f_{k,*}$ and $ f_{k}^{*}$ are given by \eqref{def_f*f_*}. 
\begin{lem}\label{lem:1stcord}
Given $P\in  \{\prebow,\prepac,\preL, \preU\}$,       if $\vect{f} \in \vect{C}(\vect{\Sigma}(\vect{m}),\R)$ satisfies \eqref{eq:strbndvar}, then
    $$
    P(\vect{\sigma},\vect{f},\Omega) = P(\vect{\sigma},\vect{f}_{*},\Omega) = P(\vect{\sigma},\vect{f}^{*},\Omega).
    $$
  \end{lem}
  \begin{proof}
    The existence of $\preL$ and $\preU$  follows by  Proposition~\ref{prop_eqQP}.
    We only give  the proof for $\prebow$, and other proofs   are similar.
  
Given  $\vect{f} \in \vect{C}(\vect{\Sigma}(\vect{m}),\R)$ satisfying \eqref{eq:strbndvar}, since
  $$
  S_{k,n}^{\vect{\sigma}}\vect{f}_{*}(\omega) \leq S_{k,n}^{\vect{\sigma}}\vect{f}(\omega) \leq S_{k,n}^{\vect{\sigma}}\vect{f}^{*}(\omega),
  $$
for all $n=1,2,\ldots$ and all $\omega \in \Sigma_{0}^{\infty}$, it follows that 
    \begin{equation}\label{eq:sumiscoincidence}
      \begin{aligned}
        S_{n}^{\vect{\sigma}}\vect{f}_{*}(\omega) \leq S_{n}^{\vect{\sigma}}\vect{f}(\omega) \leq S_{n}^{\vect{\sigma}}\vect{f}_{*}(\omega) + b, \\
        S_{n}^{\vect{\sigma}}\vect{f}^{*}(\omega) - b \leq S_{n}^{\vect{\sigma}}\vect{f}(\omega) \leq S_{n}^{\vect{\sigma}}\vect{f}^{*}(\omega).
      \end{aligned}
    \end{equation}
Combining it with \eqref{def_mcpb}, we have 
\begin{eqnarray*}
&&    \underline{\msrbpp}_{\mathscr{C}}^{s}(\vect{f},\Omega) \leq \msrbow_{\mathscr{C}}^{s}(\vect{f}_{*},\Omega) \leq \underline{\msrbpp}_{\mathscr{C}}^{s-\alpha}(\vect{f},\Omega),   \\
&&    \overline{\msrbpp}_{\mathscr{C}}^{s+\alpha}(\vect{f},\Omega) \leq \msrbow_{\mathscr{C}}^{s}(\vect{f}^{*},\Omega) \leq \overline{\msrbpp}_{\mathscr{C}}^{s}(\vect{f},\Omega).
\end{eqnarray*}
Then the conclusion for $\prebow$ follows by the same argument in \cite[Prop.3.5]{CM1}.
  \end{proof}

The next lemma is the key to  the proof of Theorem~\ref{thm:Psymsysf1stcoord}.  Given a finite cover $\mathscr{U}$ of $\Sigma_{0}^{\infty}(\vect{m})$ consisting of cylinders, that is $\Sigma_{0}^{\infty}(\vect{m})\subset \cup_{\bu\in \mathscr{U}} [\bu]$,  we denote the lowest and the highest ranks of the cylinders in $\mathscr{U}$ by $n_{\min},n_{\max}$, respectively, i.e.,
  $$
  n_{\min}=\min\{n(\bu)=\len{\bu}: \bu \in \mathscr{U}\}, \quad n_{\max}=\max\{n(\bu)=\len{\bu}: \bu \in \mathscr{U}\}.
  $$
  We have the following rank uniformization covering lemma, which is inspired by the proof of \cite[Thm.1]{Hua1994}.
  \begin{lem}\label{lem:rkrdc}
    Given a finite disjoint cover $\mathscr{U}$ of $\Sigma_{0}^{\infty}(\vect{m})$ by cylinders, for every $\vect{f}$ satisfying \eqref{eq:funct1stcoord} and all $s \in \R$,

    \begin{enumerate}[(1)]
      \item     there exists an integer $n_{*}$ with $n_{\min} \leq n_{*} \leq n_{\max}$ such that
            \begin{equation*}
              \sum_{\bu \in \mathscr{U}}\exp{\Big(-n(\bu)s+\sup_{\omega \in [\bu] }S_{n(\bu)}^{\vect{\sigma}}\vect{f}(\omega)\Big)} \geq \sum_{\bu \in \Sigma_0^{n_{*}}}\exp{\Big(-n_{*}s+\sup_{\omega \in [\bu]}S_{n_{*}}^{\vect{\sigma}}\vect{f}(\omega)\Big)};
            \end{equation*}
      \item     there exists an integer $n^{*}$ with $n_{\min} \leq n^{*} \leq n_{\max}$ such that
            \begin{equation*}
              \sum_{\bu \in \mathscr{U}}\exp{\Big(-n(\bu) s+\sup_{\omega \in [\bu]}S_{n(\bu)}^{\vect{\sigma}}\vect{f}(\omega)\Big)} \leq \sum_{\bu \in \Sigma_0^{n^{*}}}\exp{\Big(-n^{*}s+\sup_{\omega \in [\bu]}S_{n^{*}}^{\vect{\sigma}}\vect{f}(\omega)\Big)}.
            \end{equation*}
    \end{enumerate}
  \end{lem}
  \begin{proof}
We only give  the  proof for conclusion (1) since the proofs are similar. We prove it by induction on the integer $n_{\max}-n_{\min}$. 

For $n_{\max}-n_{\min}=0$, let  $n_{*}=n_{\max}=n_{\min}$ and  $\mathscr{U}=\Sigma_0^{n_{*}}$, and the conclusion holds.

Fix $q \geq 1$. Assume that the conclusion holds for all $n_{\max}-n_{\min} \leq q$. Next, we show the conclusion holds for  $n_{\max}-n_{\min} = q+1$.

For each $\bv \in \Sigma_{0}^{n_{\min}}$, we write 
    $$
    \mathscr{U}_{\bv}=\{\bu \in \mathscr{U}: [\bu] \cap [\bv] \neq \varnothing\}.
    $$
    Then $\mathscr{U}_{\bv}$ is a cover of $[\bv]$.  
    Since $n_{\min}$ is the lowest rank of the cylinders in $\mathscr{U}$, by the net properties of cylinders,   the cover $\mathscr{U}_{\bv}$ of $[\bv]$ consists either of a single cylinder $[\bv]$, or of disjoint cylinders of rank strictly greater than $n_{\min}$,    and hence  $\mathscr{U}\cap \Sigma_{0}^{n_{\min}}\neq \emptyset$.   Moreover, for all $\widehat{\bv} \in \mathscr{U}\cap \Sigma_{0}^{n_{\min}}$, we have 
    $$
    \frac{\sum_{\bv \in \mathscr{U}_{\widehat{\bv}}}\exp{\big(-n(\bv)s+\sup_{\omega \in [\bv]}S_{n(\bv)}^{\vect{\sigma}}\vect{f}(\omega)\big)}}{\exp{\big(-n_{\min}s+\sup_{\omega \in [\widehat{\bv}]}S_{n_{\min}}^{\vect{\sigma}}\vect{f}(\omega)\big)}} = \frac{\exp{\big(-n_{\min}s+\sup_{\omega \in [\widehat{\bv}]}S_{n_{\min}}^{\vect{\sigma}}\vect{f}(\omega)\big)}}{\exp{\big(-n_{\min}s+\sup_{\omega \in [\widehat{\bv}]}S_{n_{\min}}^{\vect{\sigma}}\vect{f}(\omega)\big)}}=1.
    $$

    Let
    $$
    \lambda=\min_{\mathbf{u} \in \Sigma_{0}^{n_{\min}}}{\frac{\sum_{\bv \in \mathscr{U}_{\mathbf{u}}}\exp{\big(-n(\bv)s+\sup_{\omega \in [\bv]}S_{n(\bv)}^{\vect{\sigma}}\vect{f}(\omega)\big)}}{\exp{\big(-n_{\min}s+\sup_{\omega \in [\mathbf{u}]}S_{n_{\min}}^{\vect{\sigma}}\vect{f}(\omega)\big)}}}.
    $$
    It is clear that $\lambda \leq 1$.

      \emph{Case 1.} If $\lambda=1$, then
      $$
      \sum_{\bv \in \mathscr{U}_{\mathbf{u}}}\exp{\Big(-n(\bv)s+\sup_{\omega \in [\bv]}S_{n(\bv)}^{\vect{\sigma}}\vect{f}(\omega)\Big)} \geq \exp{\Big(-n_{\min}s+\sup_{\omega \in [\mathbf{u}]}S_{n_{\min}}^{\vect{\sigma}}\vect{f}(\omega)\Big)}
      $$
      for all $\mathbf{u} \in \Sigma_{0}^{n_{\min}}$, and we are able to directly reduce the ranks of cylinders $C_{i} \in \mathscr{U}$
      to derive the estimate
      \begin{align*}
        \sum_{\bu \in \mathscr{U}}\exp{\Big(-n(\bu)s+\sup_{\omega \in [\bu]}S_{n(\bu)}^{\vect{\sigma}}\vect{f}(\omega)\Big)} &= \sum_{\mathbf{u} \in \Sigma_{0}^{n_{\min}}}\sum_{\bv \in \mathscr{U}_{\mathbf{u}}}\exp{\Big(-n(\bv)s+\sup_{\omega \in [\bv]}S_{n(\bv)}^{\vect{\sigma}}\vect{f}(\omega)\Big)} \\
        &\geq \sum_{\mathbf{u} \in \Sigma_{0}^{n_{\min}}}\exp{\Big(-n_{\min}s+\sup_{\omega \in [\mathbf{u}]}S_{n_{\min}}^{\vect{\sigma}}\vect{f}(\omega)\Big)}. 
      \end{align*}
      Taking $n_{*}=n_{\min}$, the conclusion holds.

      \emph{Case 2.} If $\lambda<1$, then  there exists $\bu_0 \in \Sigma_{0}^{n_{\min}}$ with $\bu_0 \notin \mathscr{U}$  such that
      $$
      \lambda=\frac{\sum_{\bv \in \mathscr{U}_{\bu_0}}\exp{\big(-n(\bv)s+\sup_{\omega \in [\bv]}S_{n(\bv)}^{\vect{\sigma}}\vect{f}(\omega)\big)}}{\exp{\big(-n_{\min}s+\sup_{\omega \in [\bu_0]}S_{n_{\min}}^{\vect{\sigma}}\vect{f}(\omega)\big)}}.
      $$

Let
      $$
      W=\{\mathbf{w} \in \Sigma_{n_{\min}+1}^{l}: l \geq n_{\min}+1, \mathbf{v}\mathbf{w} \in \mathscr{U}_{\mathbf{v}}\},
      $$
      and for every $\widehat{\mathbf{u}} \in \Sigma_{0}^{n_{\min}}\cap \mathscr{U}$, we  set
      $$
      \mathscr{U}_{\widehat{\mathbf{u}}}^{\prime}=\{\widehat{\mathbf{u}}\mathbf{w}: \mathbf{w} \in W\}. 
      $$
It is clear that  $\mathscr{U}_{\widehat{\mathbf{u}}}^{\prime}$ covers $[\widehat{\mathbf{u}}]$, i.e. $[\widehat{\mathbf{u}}]\subset \cup_{ \mathbf{w} \in W}[\widehat{\mathbf{u}}\mathbf{w}]$, and the rank of every element of $\mathscr{U}_{\widehat{\mathbf{u}}}^{\prime}$ is at least $n_{\min}+1$. 

For every $\widehat{\mathbf{u}} \in \Sigma_{0}^{n_{\min}}\cap \mathscr{U}$,  we write 
 \begin{align*}
     I:&=   \frac{\sum_{\widehat{\mathbf{u}}\bw \in \mathscr{U}_{\widehat{\mathbf{u}}}^{\prime}}\exp{\big(-n(\widehat{\mathbf{u}}\bw)s+\sup_{\omega \in [\widehat{\mathbf{u}}\bw]}S_{n(\widehat{\mathbf{u}}\bw)}^{\vect{\sigma}}\vect{f}(\omega)\big)}}{\exp{\left(-n_{\min}s+\sup_{\omega \in [\widehat{\mathbf{u}}]}S_{n_{\min}}^{\vect{\sigma}}\vect{f}(\omega)\right)}} \\
        & = \frac{\sum_{\mathbf{w} \in W}\exp{\big(-\len{\widehat{\mathbf{u}}\mathbf{w}}s+\sup_{\omega \in [\widehat{\mathbf{u}}\mathbf{w}]}S_{\len{\widehat{\mathbf{u}}\mathbf{w}}}^{\vect{\sigma}}\vect{f}(\omega)\big)}}{\exp{\left(-n_{\min}s+\sup_{\omega \in [\widehat{\mathbf{u}}]}S_{n_{\min}}^{\vect{\sigma}}\vect{f}(\omega)\right)}} .
   \end{align*}

Since $f_{k}$ is dependent only on the 1st coordinate $\omega_{k}$ of $\omega \in \Sigma_{k}^{\infty}$,  by \eqref{eq:sumknTf}, we have
$$
\sup_{\omega \in [\widehat{\mathbf{u}}]}S_{n_{\min}}^{\vect{\sigma}}\vect{f}(\omega) = S_{n_{\min}}^{\vect{\sigma}}\vect{f}(\vartheta^{\prime})
$$
and
 $$
 \sup_{\omega \in [\widehat{\mathbf{u}}\mathbf{w}]}S_{\len{\widehat{\mathbf{u}}\mathbf{w}}}^{\vect{\sigma}}\vect{f}(\omega) =S_{n_{\min}}^{\vect{\sigma}}\vect{f}(\vartheta') + S_{n_{\min}+1,\len{\mathbf{w}}}^{\vect{\sigma}}\vect{f}(\vartheta')
  $$
 for all $\vartheta' \in [\widehat{\mathbf{u}} \bw]$. Combining it with  and $\len{\widehat{\mathbf{u}}\mathbf{w}}=n_{\min}+\len{\bw}$,  it follows that 
      \begin{align*}
      I &= \frac{\exp{\left(-n_{\min}s+S_{n_{\min}}^{\vect{\sigma}}\vect{f}(\vartheta^{\prime})\right)}\sum_{\mathbf{w} \in W}\exp{\left(-\len{\mathbf{w}}s+S_{n_{\min}+1,\len{\mathbf{w}}}^{\vect{\sigma}}\vect{f}(\vartheta')\right)}}{\exp{\left(-n_{\min}s+\sup_{\omega \in [\widehat{\mathbf{u}}]}S_{n_{\min}}^{\vect{\sigma}}\vect{f}(\omega)\right)}} \\
      &=\sum_{\mathbf{w} \in W}\exp{\left(-\len{\mathbf{w}}s+S_{n_{\min}+1,\len{\mathbf{w}}}^{\vect{\sigma}}\vect{f}(\vartheta')\right)}.
      \end{align*}
Similarly, for all $\vartheta\in [\bu_0 \mathbf{w}]$, we have
$$
S_{n_{\min}}^{\vect{\sigma}}\vect{f}(\vartheta)=\sup_{\omega \in [\bu_0]}S_{n_{\min}}^{\vect{\sigma}}\vect{f}(\omega)
\quad \textit{ and } \quad  S_{n_{\min}+1,\len{\mathbf{w}}}^{\vect{\sigma}}\vect{f}(\vartheta) = S_{n_{\min}+1,\len{\mathbf{w}}}^{\vect{\sigma}}\vect{f}(\vartheta'),
$$
 and it follows that       
   \begin{align*}
     I  &= \frac{\exp{\left(-n_{\min}s+S_{n_{\min}}^{\vect{\sigma}}\vect{f}(\vartheta)\right)}\sum_{\mathbf{w} \in W}\exp{\left(-\len{\mathbf{w}}s+S_{n_{\min}+1,\len{\mathbf{w}}}^{\vect{\sigma}}\vect{f}(\vartheta)\right)}}{\exp{\left(-n_{\min}s+\sup_{\omega \in [\bu_0]}S_{n_{\min}}^{\vect{\sigma}}\vect{f}(\omega)\right)}} \\
        &= \frac{\sum_{\mathbf{w} \in W}\exp{\left(-\len{\bu_0\mathbf{w}}s+\sup_{\omega \in [\bu_0\bw]}S_{\len{\bu_0\mathbf{w}}}^{\vect{\sigma}}\vect{f}(\omega)\right)}}{\exp{\left(-n_{\min}s+\sup_{\omega \in [\bu_0]}S_{n_{\min}}^{\vect{\sigma}}\vect{f}(\omega)\right)}} \\
        &= \frac{\sum_{\bv \in \mathscr{U}_{\bu_0}}\exp{\left(-n(\bv)s+\sup_{\omega \in \bv}S_{n(\bv)}^{\vect{\sigma}}\vect{f}(\omega)\right)}}{\exp{\left(-n_{\min}s+\sup_{\omega \in [\bu_0]}S_{n_{\min}}^{\vect{\sigma}}\vect{f}(\omega)\right)}} =\lambda.
      \end{align*}

Since    for every $\widehat{\mathbf{u}} \in \Sigma_{0}^{n_{\min}}\cap \mathscr{U}$,     
$$
\lambda< \frac{\sum_{\bv \in \mathscr{U}_{\widehat{\mathbf{u}}}}\exp{\left(-n(\bv)s+\sup_{\omega \in [\bv]}S_{n(\bv)}^{\vect{\sigma}}\vect{f}(\omega)\right)}}{\exp{\left(-n_{\min}s+\sup_{\omega \in [\widehat{\mathbf{u}}]}S_{n_{\min}}^{\vect{\sigma}}\vect{f}(\omega)\right)}}, 
$$
 we obtain that 
\begin{equation}\label{clm:b}
              \sum_{\bv \in \mathscr{U}_{\widehat{\mathbf{u}}}^{\prime}}\exp{\big(-n(\bv)s+\sup_{\omega \in [\bv]}S_{n(\bv)}^{\vect{\sigma}}\vect{f}(\omega)\big)} < \sum_{\bv \in \mathscr{U}_{\widehat{\mathbf{u}}}}\exp{\big(-n(\bv)s+\sup_{\omega \in [\bv]}S_{n(\bv)}^{\vect{\sigma}}\vect{f}(\omega)\big)}.
\end{equation}

      Let
      $$
      \mathscr{U}^{\prime}=\left(\mathscr{U} \setminus \Sigma_0^{n_{\min}}\right) \cup \Big(\bigcup_{\widehat{\mathbf{u}} \in \mathscr{U} \cap \Sigma_0^{n_{\min}}}\mathscr{U}_{\widehat{\mathbf{u}}}^{\prime}\Big).
      $$
Since the members of $\mathscr{U}^{\prime}$ are cylinders of rank at least $n_{\min}+1$,  summing \eqref{clm:b} over $ \widehat{\mathbf{u}}\in \mathscr{U} \cap \Sigma_0^{n_{\min}}$ implies that
      $$
      \sum_{\bv \in \mathscr{U}^{\prime}}\exp{\Big(-n(\bv)s+\sup_{\omega \in [\bv]}S_{n(\bv)}^{\vect{\sigma}}\vect{f}(\omega)\Big)} < \sum_{\bv \in \mathscr{U}}\exp{\Big(-n(\bv)s+\sup_{\omega \in [\bv]}S_{n(\bv)}^{\vect{\sigma}}\vect{f}(\omega)\Big)}.
      $$

      By the induction hypothesis, there exists  $n_{*} \in \mathbb{N}$ with $n_{\min} < n_{\min}+1 \leq n_{*} \leq n_{\max}$ such that
      $$
       \sum_{\bv \in \mathscr{U}^{\prime}}\exp{\Big(-n(\bv)s+\sup_{\omega \in [\bv]}S_{n(\bv)}^{\vect{\sigma}}\vect{f}(\omega)\Big)} \geq \sum_{\bv \in \Sigma_0^{n_{*}}}\exp{\Big(-n_{*}s+\sup_{\omega \in [\bv]}S_{n_{*}}^{\vect{\sigma}}\vect{f}(\omega)\Big)}.
      $$
      It follows that
      $$
    \sum_{\bv \in \mathscr{U}}\exp{\Big(-n(\bv)s+\sup_{\omega \in [\bv]}S_{n(\bv)}^{\vect{\sigma}}\vect{f}(\omega)\Big)} \geq \sum_{\bv \in \Sigma_0^{n_{*}}}\exp{\Big(-n_{*}s+\sup_{\omega \in [\bv]}S_{n_{*}}^{\vect{\sigma}}\vect{f}(\omega)\Big)}.
      $$

Combining Case $1$ and Case $2$ together,  we obtain the  conclusion (1).
\end{proof}

  \subsection{Pressures of potentials dependent  on the 1st coordinate}\label{ssect:Psymsysf1stcoord}
 In the subsection, we give the proofs of Theorems~\ref{thm:Psymsysf1stcoord} and \ref{thm:Psymsys}
 and     provide some of their consequences.

  \begin{proof}[Proof of Theorem \ref{thm:Psymsysf1stcoord}]
 Since  $\vect{f}=\{f_{k}\}_{k=1}^{\infty}$ satisfies \eqref{eq:funct1stcoord}, $\vect{f}$ is equicontinuous, and  by \eqref{def_s_su}, Proposition \ref{prop_ExPUPL}  and Proposition \ref{eq:PlowupOmega},   
    $$
    \preL(\vect{\sigma},\vect{f},\Sigma_{0}^{\infty})=    \prelow(\vect{\sigma},\vect{f},\Sigma_{0}^{\infty})=\underline{s},
    \qquad
    \preU(\vect{\sigma},\vect{f},\Sigma_{0}^{\infty})=    \preup(\vect{\sigma},\vect{f},\Sigma_{0}^{\infty})=\overline{s}.
    $$
    The equation \eqref{eq:PupeqPsymsys}  follows from Corollary~\ref{cor:PupeqPP}.
Since   $\prebow(\vect{\sigma},\vect{f},\Sigma_{0}^{\infty}) \leq \prelow(\vect{\sigma},\vect{f},\Sigma_{0}^{\infty})$, it remains to show 
    $$
    \prebow(\vect{\sigma},\vect{f},\Sigma_{0}^{\infty}) \geq \prelow(\vect{\sigma},\vect{f},\Sigma_{0}^{\infty}).
    $$

Given $s \in \mathbb{R}$.  Since $\Sigma_{0}^{\infty}$ is compact, for any given countable cover $\mathscr{U}$ of $\Sigma_{0}^{\infty}$ by cylinders,
    we are always able to find a finite subcover $\mathscr{U}^{\prime}$ with smaller sums
    $$
    \sum_{\bv \in \mathscr{U}^{\prime}}\exp{\Big(-n(\bv)s+\sup_{\omega \in [\bv]}S_{n(\bv)}^{\vect{\sigma}}\vect{f}(\omega)\Big)} \leq \sum_{\bv \in \mathscr{U}}\exp{\Big(-n(\bv)s+\sup_{\omega \in [\bv]}S_{n(\bv)}^{\vect{\sigma}}\vect{f}(\omega)\Big)},
    $$
where $n(\bv)=\len{\bv}$.
Given $N \geq 1$, for any cover $\mathscr{U}^{\prime}$ of $\Sigma_{0}^{\infty}$ consisting of cylinders  of ranks no less than $ N$,
    by the net properties of cylinders, we are able to find a subcover $\mathscr{U}^{\prime\prime} \subseteq \mathscr{U}^{\prime}$ such that
    all members in $\mathscr{U}^{\prime\prime}$ are pairwise disjoint. This implies that 
    $$
    \sum_{\bv \in \mathscr{U}^{\prime\prime}}\exp{\Big(-n(\bv)s+\sup_{\omega \in [\bv]}S_{n(\bv)}^{\vect{\sigma}}\vect{f}(\omega)\Big)} \leq \sum_{\bv \in \mathscr{U}^{\prime}}\exp{\Big(-n(\bv)s+\sup_{\omega \in [\bv]}S_{n(\bv)}^{\vect{\sigma}}\vect{f}(\omega)\Big)}.
    $$

    Hence  to estimate  a lower bound for $\msrbppup_{\mathscr{C}}^{s}(\vect{f},\Sigma_{0}^{\infty})$, without loss of generality,   we assume $\mathscr{U}$  is a finite cover consisting of  disjoint cylinders. By Lemma~\ref{lem:rkrdc}(1), there exists $n_{*} \geq N$ such that
    $$
    \sum_{\bv \in \mathscr{U}}\exp{\Big(-n(\bv)s+\sup_{\omega \in [\bv]}S_{n(\bv)}^{\vect{\sigma}}\vect{f}(\omega)\Big)} \geq \sum_{\bv \in \Sigma_0^{n_{*}}}\exp{\Big(-n_{*}s+\sup_{\omega \in [\bv]}S_{n_{*}}^{\vect{\sigma}}\vect{f}(\omega)\Big)}.
    $$
     Since the above holds for all finite disjoint covers $\mathscr{U}$ of $\Sigma_{0}^{\infty}$ by cylinders, by \eqref{def_mcpb} and \eqref{eq:defQCs},
    it follows that
  $$
  \msrbppup_{\mathscr{C}}^{s}(\vect{f},\Sigma_{0}^{\infty}) \geq \carpeslow_{\mathscr{C}}^{s}(\vect{f},\Sigma_{0}^{\infty}).
  $$
 Hence by Proposition~\ref{prop_eqQP} and~\ref{prop_eqpBE},         we obtain  $\prebow(\vect{\sigma},\vect{f},\Sigma_{0}^{\infty}) \geq \prelow(\vect{\sigma},\vect{f},\Sigma_{0}^{\infty})$.
  \end{proof}

  \begin{proof}[Proof of Theorem~\ref{thm:Psymsys}]
   It is a direct consequence of   Lemma~\ref{lem:1stcord} and Theorem~\ref{thm:Psymsysf1stcoord}.
  \end{proof}

The following results are immediate consequences of Theorem~\ref{thm:Psymsysf1stcoord}.
  \begin{cor}
    Given a real sequence $\{a_k\}_{k=0}^\infty$, let $\vect{f}=\{f_k=a_k\}_{k=0}^\infty$. Then
\begin{eqnarray*}
   && \preL(\vect{\sigma},\vect{f},\Sigma_{0}^{\infty})=\prebow(\vect{\sigma},\vect{f},\Sigma_{0}^{\infty})=\varliminf_{n \to \infty}\frac{1}{n}\sum_{j=0}^{n-1}(\log{m_{j}}+a_{j}),\\  
   && \preU(\vect{\sigma},\vect{f},\Sigma_{0}^{\infty})=\prepac(\vect{\sigma},\vect{f},\Sigma_{0}^{\infty})=\varlimsup_{n \to \infty}\frac{1}{n}\sum_{j=0}^{n-1}(\log{m_{j}}+a_{j}).
\end{eqnarray*}
Moreover, if $\lim_{k\to\infty }a_{k} = a$, then
\begin{eqnarray*}
   && \preL(\vect{\sigma},\vect{f},\Sigma_{0}^{\infty})=\prebow(\vect{\sigma},\vect{f},\Sigma_{0}^{\infty})=\varliminf_{n \to \infty}\frac{1}{n}\sum_{j=0}^{n-1}\log{m_{j}}+a, \\
   && \preU(\vect{\sigma},\vect{f},\Sigma_{0}^{\infty})=\prepac(\vect{\sigma},\vect{f},\Sigma_{0}^{\infty})=\varlimsup_{n \to \infty}\frac{1}{n}\sum_{j=0}^{n-1}\log{m_{j}}+a.
\end{eqnarray*}
Furthermore, we have that 
 \begin{eqnarray*}
   &&\enttoplow(\vect{\sigma},\Sigma_{0}^{\infty})=\enttopbow(\vect{\sigma},\Sigma_{0}^{\infty})=\varliminf_{n \to \infty}\frac{1}{n}\sum_{j=0}^{n-1}\log{m_{j}},\\
   && \enttopup(\vect{\sigma},\Sigma_{0}^{\infty})=\enttoppac(\vect{\sigma},\Sigma_{0}^{\infty})=\varlimsup_{n \to \infty}\frac{1}{n}\sum_{j=0}^{n-1}\log{m_{j}}.
\end{eqnarray*}
  \end{cor}

  \begin{rmk}
    If $m_{k}=m$ and $ f_k=f $ where $f(\omega)=a_{i}$ for all $\omega \in [i]$ and $i=1,\ldots, m$, then  $(\vect{\Sigma}(\vect{m}),\vect{\sigma})$
    degenerates to   $(\Sigma(m),\sigma)$, the above results reduce to 
    \begin{eqnarray*}
   && \prebow(\sigma,f,\Sigma)=\prepac(\sigma,f,\Sigma)=P(\sigma,f)=\log{\Big(\sum_{i=1}^{m}\e^{a_{i}}\Big)}, \\    
   && \enttopbow(\sigma,\Sigma)=\enttoppac(\sigma,\Sigma)=\enttop(\sigma)=\log{m}.
    \end{eqnarray*}
  \end{rmk}

\section{Measure-Theoretic Pressures and Equilibrium States}\label{sect:msrpre}

  \subsection{Measure-theoretic pressure of nonautonomous Bernoulli measures}\label{ssect:msrprenabmsr}

  We first present the almost everywhere exact formulae for the local entropies under certain conditions.   For this purpose, we first cite a law of large numbers, also known as the Kolmogorov's criterion.
  See \cite[\S5.2 Cor.1.]{Chow&Teicher1997}  for the proof.
  \begin{lem}\label{lem:KSLLN}
    Given a probability space $(\Omega,\mathscr{F},\mathbb{P})$, let $\{X_{n}\}_{n=1}^{\infty}$ be a sequence of independent random variables with $\var[X_{n}]<\infty$ for each $n \geq 1$. If $    \sum_{n=1}^{\infty}\frac{1}{n^{2}}\var[X_{n}]<\infty,$
    then
    $$
  \lim_{n\to\infty}  \frac{1}{n}\sum_{j=0}^{n-1}(X_{j} - \E[X_{j}]) = 0 \quad (\mathbb{P} \text{-a.s.}).
    $$
  \end{lem}
The next conclusion follows from the definition of measure-theoretic local entropies. 
\begin{prop} \label{eq:entlocpnt}
Let $\mu$ be the nonautonomous Bernoulli measure generated by \eqref{eq:NAPV}.
 Then for all $\omega \in \Sigma_{0}^{\infty}$,
      \begin{equation}
        \entlow_{\mu}(\vect{\sigma},\omega) = \varliminf_{n \to \infty}-\frac{1}{n}\sum_{j=0}^{n-1}\log{p_{j,\omega_{j}}},
        \quad
        \entup_{\mu}(\vect{\sigma},\omega) = \varlimsup_{n \to \infty}-\frac{1}{n}\sum_{j=0}^{n-1}\log{p_{j,\omega_{j}}}.
      \end{equation}
\end{prop}

 By applying the law of large numbers to  Proposition \ref{eq:entlocpnt},  we obtain a $\mu$-a.e. formula for the measure-theoretic entropies.
  \begin{thm}\label{thm:entlocae}
    Let $\mu$ be the nonautonomous Bernoulli measure generated by \eqref{eq:NAPV}.
    If $\lim_{n \to \infty}\frac{m_{n}}{n^{1-\alpha}}<1$ for some $\alpha>0$, then for $\mu$-a.e. $\omega \in \Sigma_{0}^{\infty}$,
    $$
    \entlow_{\mu}(\vect{\sigma},\omega) = \varliminf_{n \to \infty}\frac{1}{n} \sum_{j=0}^{n-1} H(\mathbf{p}_{j}),
    \qquad
    \entup_{\mu}(\vect{\sigma},\omega) = \varlimsup_{n \to \infty}\frac{1}{n} \sum_{j=0}^{n-1} H(\mathbf{p}_{j}).
    $$

  \end{thm}
  \begin{proof}
    Note that $\mu(B_{n}^{\vect{\sigma}}(\omega,\varepsilon))=\prod_{j=0}^{n + \lfloor-\log{\varepsilon}\rfloor-1}p_{j,\omega_{j}}$ for all $\omega \in \Sigma_{0}^{\infty}$.
    For each $n \geq 1$, we define a random variable $X_{n}$ on $(\Sigma_{0}^{\infty},\mathscr{B},\mu)$ by
    $$
    X_{n}(\omega)=-\log{p_{n,\omega_{n}}}.
    $$
    It is easy to verify that $\{X_{n}\}_{n=1}^{\infty}$ is a sequence of independent random variables with
    \begin{eqnarray*}
    && \E[X_{n}] = -\sum_{i=1}^{m_{n}}p_{n,i}\log{p_{n,i}}=H(\mathbf{p}_{n}),\\
    && \var[X_{n}] = \E[X_{n}^{2}]-(\E[X_{n}])^{2} \leq  \sum_{i=1}^{m_{n}}p_{n,i}(\log{p_{n,i}})^{2} < \infty.
    \end{eqnarray*}
Moreover $    \var[X_{n}] \leq 4 \e^{-2} m_{n} $ since  $-\frac{1}{\e} \leq x\log{x} < 0$ and $0 < x(\log{x})^{2} \leq 4 \e^{-2}$ for  $x\in(0, 1)$. 
By   $\lim_{n \to \infty}\frac{m_{n}}{n^{1-\alpha}}<1$, we obtain
    $$
    \sum_{n=1}^{\infty}\frac{1}{n^{2}}\var[X_{n}]\leq     \sum_{n=1}^{\infty}\frac{1}{n^{2}}4 \e^{-2} m_{n} <\infty.
    $$
Therefore, it follows by Lemma \ref{lem:KSLLN} that for $\mu$-a.e. $\omega \in \Sigma_{0}^{\infty}$, 
    \begin{equation}\label{eq:entlim}
   \lim_{n\to\infty} \frac{1}{n}\sum_{j=0}^{n-1}\Bigl(-\log{p_{j,\omega_{j}}} + H(\mathbf{p}_{n})\Bigr) = 0.
    \end{equation}
Combining this with  Proposition \ref{eq:entlocpnt}, we obtain that   for $\mu$-a.e. $\omega$,
\begin{eqnarray*}
    &&\entlow_{\mu}(\vect{\sigma},\omega) = \varliminf_{n \to \infty}-\frac{1}{n}\sum_{j=0}^{n-1}\log{p_{j,\omega_{j}}} = \varliminf_{n \to \infty}\frac{1}{n} \sum_{j=0}^{n-1} H(\mathbf{p}_{j}), \\
    &&\entup_{\mu}(\vect{\sigma},\omega) = \varlimsup_{n \to \infty}-\frac{1}{n}\sum_{j=0}^{n-1}\log{p_{j,\omega_{j}}} = \varlimsup_{n \to \infty}-\frac{1}{n}\sum_{j=0}^{n-1} H(\mathbf{p}_{j}). \qedhere
\end{eqnarray*}
  \end{proof}

  Similar to Lemma~\ref{lem:1stcord}, we have  the measure-theoretic version equality for potentials satisfying the strongly bounded variation property \eqref{eq:strbndvar}.
  \begin{lem}
    Given $\mu \in M( \Sigma_{0}^{\infty}(\vect{m}))$.  For $P \in  \{\prelow_{\mu}, \preup_{\mu}\}$, if $\vect{f} \in \vect{C}(\vect{\Sigma}(\vect{m}),\R)$ satisfies \eqref{eq:strbndvar}, then
    $$
    P(\vect{\sigma},\vect{f}) = P(\vect{\sigma},\vect{f}_{*}) = P(\vect{\sigma},\vect{f}^{*}).
    $$
  \end{lem}
  \begin{proof}
    It is immediate from \eqref{eq:sumiscoincidence} and \eqref{eq:defpremsrequiv}.
  \end{proof}

Next, we provide the pointwise exact formulae for measure-theoretic local pressures.
  \begin{prop} \label{eq:prelocpnt}
    Let $\mu$ be the nonautonomous Bernoulli measure given by \eqref{eq:NAPV} and $\vect{f} \in \vect{C}(\vect{\Sigma}(\vect{m}),\R)$ satisfy  \eqref{eq:strbndvar}. Given $\omega \in \Sigma_{0}^{\infty}$. Let $a_{k}(\omega)= f_{k} \circ \vect{\sigma}^{k}(\omega)$. 
    \begin{enumerate}[(a)]
      \item   If $p_{*}(\omega):=\inf_{j \in \N}\{p_{j,\omega_{j}}\}>0$, then 
    \begin{equation*}
        \prelow_{\mu}(\vect{\sigma},\vect{f},\omega) = \varliminf_{n \to \infty}\frac{1}{n}\sum_{j=0}^{n-1}\big(a_{j}(\omega)-\log{p_{j,\omega_{j}}}\big), \quad
        \preup_{\mu}(\vect{\sigma},\vect{f},\omega) = \varlimsup_{n \to \infty}\frac{1}{n}\sum_{j=0}^{n-1}\big(a_{j}(\omega)-\log{p_{j,\omega_{j}}}\big).
    \end{equation*}
      \item  If either $\entlow_{\mu}(\vect{\sigma},\omega) = \entup_{\mu}(\vect{\sigma},\omega)$ or $\{a_{k}(\omega)\}_{k=1}^\infty$  converges, then
    \begin{equation*}
        \prelow_{\mu}(\vect{\sigma},\vect{f},\omega) =\entlow_{\mu}(\vect{\sigma},\omega)+\varliminf_{n \to \infty}\frac{1}{n}\sum_{j=0}^{n-1}a_{j}(\omega), \quad
        \preup_{\mu}(\vect{\sigma},\vect{f},\omega) = \entup_{\mu}(\vect{\sigma},\omega) +\varlimsup_{n \to \infty}\frac{1}{n}\sum_{j=0}^{n-1}a_{j}(\omega).
    \end{equation*}

    \end{enumerate}
  \end{prop}
  \begin{proof}
(a) Suppose $p_{*}(\omega)=\inf_{j \in \N}\{p_{j,\omega_{j}}\}>0$. By \eqref{eq:NAPV}, \eqref{eq:defpremsr} and \eqref{def_BSkn}, we have that 
    \begin{align*}
      \prelow_{\mu}(\vect{\sigma},\vect{f},\omega) &= \lim_{\varepsilon \to 0}\varliminf_{n \to \infty}\frac{1}{n}\Big(-\log{\mu(B_{n}^{\vect{\sigma}}(\omega,\varepsilon))}+S_{n}^{\vect{\sigma}}\vect{f}(\omega)\Big) \\
                                          &= \lim_{\varepsilon \to 0}\varliminf_{n \to \infty}\frac{1}{n}\Big(-\log{\prod_{j=0}^{n+\lfloor-\log{\varepsilon}\rfloor-1}p_{j,\omega_{j}}}+\sum_{j=0}^{n-1}a_{j}(\omega)\Big) \\
                                          &= \lim_{\varepsilon \to 0}\varliminf_{n \to \infty}\frac{1}{n}\Big( \sum_{j=0}^{n-1}( a_{j}(\omega) - \log{p_{j,\omega_{j}}} )  - \sum_{j=n}^{n+\lfloor-\log{\varepsilon}\rfloor-1}\log{p_{j,\omega_{j}}}   \Big).
    \end{align*}
    For $0<\varepsilon<\e^{-1}$, since $p_{j,\omega_{j}} \geq p_{*}(\omega)>0$, it follows that
   $$
\varlimsup_{n \to \infty}  -  \frac{1}{n}\sum_{j=n+1}^{n+\lfloor-\log{\varepsilon}\rfloor}\log{p_{j,\omega_{j}}} \leq \varlimsup_{n \to \infty} \frac{\lfloor\log{\varepsilon}\rfloor}{n}\log{p_{*}(\omega)} =0.
    $$
    Hence $    \prelow_{\mu}(\vect{\sigma},\vect{f},\omega) = \varliminf_{n \to \infty}\frac{1}{n}\sum_{j=0}^{n-1}( a_{j}(\omega) - \log{p_{j,\omega_{j}}} )$, and  the proof for $\preup_{\mu}$ is identical.

(b)   The conclusion follows by Proposition \ref{eq:entlocpnt} and  Theorem \ref{thm:entlocae}. 
  \end{proof}

Similar to Theorem \ref{thm:entlocae}, we obtain  $\mu$-a.e. formulae for the lower and upper measure-theoretic local pressures.
  \begin{thm}\label{thm:prelocae}
    Let $\mu$ be the nonautonomous Bernoulli measure given by \eqref{eq:NAPV} and $\vect{f} \in \vect{C}(\vect{\Sigma}(\vect{m}),\R)$ satisfy \eqref{eq:strbndvar}.   
    \begin{enumerate}[(a)]
    \item Suppose that  $p_{*}(\omega)=\inf_{j \in \N}\{p_{j,\omega_{j}}\}>0$ for $\mu$-a.e. $\omega \in \Sigma_{0}^{\infty}$.   If $\lim_{n \to \infty}\frac{m_{n}}{n^{1-\alpha}}<1$  and $\lim_{n \to \infty}\frac{\|f_{n}\|_{\infty}^2}{n^{1-\alpha}}<1$ for some $\alpha>0$,
    then  for almost all $\omega$,
$$
 \prelow_{\mu}(\vect{\sigma},\vect{f},\omega) = \varliminf_{n \to \infty}\frac{1}{n}\sum_{j=0}^{n-1}\big(H(\mathbf{p}_{j}) + \E(f_j)\big), \\
\quad  \preup_{\mu}(\vect{\sigma},\vect{f},\omega) = \varlimsup_{n \to \infty}\frac{1}{n}\sum_{j=0}^{n-1}\big(H(\mathbf{p}_{j}) + \E(f_j)\big).
$$
    \item Suppose for $\mu$-a.e. $\omega \in \Sigma_{0}^{\infty}$, either $\entlow_{\mu}(\vect{\sigma},\omega) = \entup_{\mu}(\vect{\sigma},\omega)$ or $\{f_{k} \circ \vect{\sigma}^{k}(\omega)\}_{k=1}^\infty$ pointwise converges.  If $\lim_{n \to \infty}\frac{m_{n}}{n^{1-\alpha}}<1$ and $\lim_{n \to \infty}\frac{\|f_{n}\|_{\infty}^2}{n^{1-\alpha}}<1$ for some $\alpha>0$,
    then  for almost all $\omega$,
$$
 \prelow_{\mu}(\vect{\sigma},\vect{f},\omega) = \entlow_{\mu}(\vect{\sigma},\omega)+ \varliminf_{n \to \infty}\frac{1}{n}\sum_{j=0}^{n-1}\E(f_j), \\
\quad  \preup_{\mu}(\vect{\sigma},\vect{f},\omega) = \entup_{\mu}(\vect{\sigma},\omega) + \varlimsup_{n \to \infty}\frac{1}{n}\sum_{j=0}^{n-1} \E(f_j). 
$$
    \end{enumerate}

  \end{thm}
  \begin{proof}
(a) Given  $\omega \in \Sigma_{0}^{\infty}$,  for each integer $n \geq 1$, let 
    $$
    Y_{n}(\omega)=f_{n} \circ \vect{\sigma}^{n}(\omega).
    $$
     It is clear that $\{Y_{n}\}_{n=1}^{\infty}$ is a sequence of independent random variables with
    $$
    \E[Y_{n}]=\E(f_{n})
\quad \textit{and } \quad 
    \var[Y_{n}] =\E(f_{n}^2)-\E(f_{n})^{2}<\infty,
    $$
    for each integer $n \geq 1$. 

Note that  $f_{n}(\omega) \leq \|f_{n}\|_{\infty}$ for each integer $n \geq 1$, we have $    \var[Y_{n}] \leq \|f_{n}\|_{\infty}^2. $
Since $\lim_{n \to \infty}\frac{\|f_{n}\|_{\infty}^2}{n^{1-\alpha}}<1$ for some $\alpha>0$, we obtain that 
    $$
    \sum_{n=1}^{\infty}\frac{1}{n^{2}}\var[Y_{n}]<\infty.
    $$
  By Lemma \ref{lem:KSLLN}, it follows that for $\mu$-a.e. $\omega \in \Sigma_{0}^{\infty}$, 
    \begin{equation}\label{eq:potlim}
    \lim_{n\to\infty}  \frac{1}{n}\sum_{j=0}^{n-1}\Bigl(Y_n + \E(f_j)\Bigr) = 0.
    \end{equation}
Since  $p_{*}(\omega)>0$ for $\mu$-a.e. $\omega \in \Sigma_{0}^{\infty}$, combining \eqref{eq:entlim} and \eqref{eq:potlim} with Proposition \ref{eq:prelocpnt}(a), we obtain 
$$
 \prelow_{\mu}(\vect{\sigma},\vect{f},\omega) = \varliminf_{n \to \infty}\frac{1}{n}\sum_{j=0}^{n-1}\big(H(\mathbf{p}_{j}) + \E(f_j)\big), \\
\quad  \preup_{\mu}(\vect{\sigma},\vect{f},\omega) = \varlimsup_{n \to \infty}\frac{1}{n}\sum_{j=0}^{n-1}\big(H(\mathbf{p}_{j}) + \E(f_j)\big)
$$
    for $\mu$-a.e. $\omega \in \Sigma_{0}^{\infty}$. 

(b) The conclusion follows from the same argument as in (a), appealing to Proposition \ref{eq:prelocpnt}(b).
  \end{proof}

By an application of \ref{thm:entlocae} and  \ref{thm:prelocae}, we establish  Theorems \ref{thm:msrent} and \ref{thm:msrpre}.
\begin{proof}[Proof of Theorem \ref{thm:msrent}]
Since $\sup_{n=1,2,\ldots}{m_{n}}<+\infty$, it is clear that  $\lim_{n \to \infty}\frac{m_{n}}{n^{1-\alpha}}<1$ for all $\alpha>0$. Hence, by Theorem \ref{thm:entlocae},  we have  
 $$
    \entlow_{\mu}(\vect{\sigma},\omega) = \varliminf_{n \to \infty}\frac{1}{n} \sum_{j=0}^{n-1} H(\mathbf{p}_{j}),
    \qquad
    \entup_{\mu}(\vect{\sigma},\omega) = \varlimsup_{n \to \infty}\frac{1}{n} \sum_{j=0}^{n-1} H(\mathbf{p}_{j}),
    $$
 for $\mu$-a.e. $\omega \in \Sigma_{0}^{\infty}$,  and the conclusion  follows by integrating  with respect to $\mu$.
\end{proof}

\begin{proof}[Proof of Theorem \ref{thm:msrpre}]
Since  $\sup_{n=1,2,\ldots}{m_{n}}<+\infty$ and $\|\vect{f}\|=\sup_{k \in \mathbb{N}}\{\|f_{k}\|_{\infty}\}<+\infty$, we have that  $\lim_{n \to \infty}\frac{m_{n}}{n^{1-\alpha}}<1$ and $\lim_{n \to \infty}\frac{\|f_{n}\|^2_{\infty}}{n^{1-\alpha}}<1$ for all $\alpha>0$.  The conclusion  is a direct consequence of  Theorem \ref{thm:prelocae} by integrating  with respect to $\mu$.
\end{proof}

  \subsection{Equilibrium states and Gibbs states}\label{ssect:eqlbretgibbs}
We first provide the proofs of Proposition \ref{prop:eqlbrae} and \ref{prop:eqlbrrestr}. Then we present several of their applications autonomous symbolic systems. 
Note that the pressures involved here are nonnegative, see \cite[Prop.5.3]{CM1}.  
 
  \begin{proof}[Proof of Proposition \ref{prop:eqlbrae}]
    We  only prove (1), as (2) follows by an analogous argument.
 Since $\mu \in \eqlbrbow_{\vect{f}}(\Omega)$, 
    it suffices to show that
    $$
    \mu(\{\omega \in \Omega: \prelow_{\mu}(\vect{\sigma},\vect{f},\omega) > \prebow(\vect{\sigma},\vect{f})\})=0
    $$

We prove it by contradiction.    Write $E=\{\omega \in \Omega: \prelow_{\mu}(\vect{\sigma},\vect{f},\omega) > \prebow(\vect{\sigma},\vect{f})\}$, and assume  that $\mu(E)>0$. Let $\nu=\frac{\mu\vert_{E}}{\mu(E)}$.    It is clear that for all $\omega \in \Sigma_{0}^{\infty}$ and $\varepsilon>0$,
    \begin{align*}
      &\varliminf_{n \to \infty}\frac{-\log{\nu([\omega\vert{(n+\lfloor -\log{\varepsilon} \rfloor)}])} + S_{n}^{\vect{\sigma}}\vect{f}(\omega)}{n} \\
      &\geq \varliminf_{n \to \infty}\frac{-\log{\mu([\omega\vert{(n+\lfloor -\log{\varepsilon} \rfloor)}])} + \log{\mu(E)} + S_{n}^{\vect{\sigma}}\vect{f}(\omega)}{n} \\
      &\geq \varliminf_{n \to \infty}\frac{-\log{\mu([\omega\vert{(n + \lfloor -\log{\varepsilon} \rfloor)}])} + S_{n}^{\vect{\sigma}}\vect{f}(\omega)}{n},
    \end{align*}
    and hence by \eqref{eq:defpremsr}, we have that 
    \begin{equation}\label{ineq_pnpm}
    \prelow_{\nu}(\vect{\sigma},\vect{f},\omega) \geq \prelow_{\mu}(\vect{\sigma},\vect{f},\omega).
    \end{equation}
     Integrating on both sides with respect to $\nu$, we obtain that  
    \begin{align*}
      \prelow_{\nu}(\vect{\sigma},\vect{f}) &\geq \frac{1}{\mu(E)}\int_{E}\prelow_{\mu}(\vect{\sigma},\vect{f},\omega)\dif{\mu}(\omega) \\
                                            &> \frac{1}{\mu(E)}\int_{E}\prebow(\vect{\sigma},\vect{f},\Omega)\dif{\mu}(\omega) \\
                                            &\geq \prebow(\vect{\sigma},\vect{f},\Omega).
    \end{align*}
 Since $E \subseteq \Omega$ and $\nu(E)=1$, the above inequality   contradicts the variational principle \eqref{eq:varprinPB},
    and therefore $\mu(E)=0$, which completes  the proof.
  \end{proof}
 
  \begin{proof}[Proof of Proposition \ref{prop:eqlbrrestr}]
    The proof of (1) is given; that of (2) is similar and  omitted.

    By  \eqref{ineq_pnpm}, we obtain that $\prelow_{\nu}(\vect{\sigma},\vect{f},\omega) \geq \prelow_{\mu}(\vect{\sigma},\vect{f},\omega)$ for all $\omega \in \Sigma_{0}^{\infty}$.
    Since $\mu \in \eqlbrbow_{\vect{f}}(\Omega)$, by Proposition~\ref{prop:eqlbrae},
    we have that $\prelow_{\mu}(\vect{\sigma},\vect{f},\omega) = \prebow(\vect{\sigma},\vect{f},\Omega)$ for $\mu$-a.e. $\omega \in \Sigma_{0}^{\infty}$.
   Then it follows that $\prelow_{\nu}(\vect{\sigma},\vect{f}) \geq \prebow(\vect{\sigma},\vect{f},\Omega)$.
    By the variational principle \eqref{eq:varprinPB}\footnote{In fact, it suffices to use the variational inequalities \cite[Lem.6.1 \& Lem.6.3]{CM2}, or equivalently, the Billingsley type theorems \cite[Thm.2.4(2) \& Thm.2.9(2)]{CM2}, avoiding the extra conditions on $\vect{f}$ in \eqref{eq:varprinPP}.},
    this implies that
    $$
    \prebow(\vect{\sigma},\vect{f},\Theta) \geq \prelow_{\nu}(\vect{\sigma},\vect{f}) \geq \prebow(\vect{\sigma},\vect{f},\Omega).
    $$
Since $\prebow(\vect{\sigma},\vect{f},\Theta) \leq \prebow(\vect{\sigma},\vect{f},\Omega)$, the conclusion holds.
  \end{proof}

Given   $\vect{f} \in \vect{C}(\vect{\Sigma}(\vect{m}),\R)$, if there is a constant $C>0$ such that for every $l,n \geq 1$ and $\omega \in \Omega$,
      \begin{equation}\label{eq:Snfalmadd}
        S_{n}^{\vect{\sigma}}\vect{f}(\omega) + S_{l}^{\vect{\sigma}}\vect{f}(\vect{\sigma}^{n}\omega) - C \leq S_{l+n}^{\vect{\sigma}}\vect{f}(\omega) \leq S_{n}^{\vect{\sigma}}\vect{f}(\omega) + S_{l}^{\vect{\sigma}}\vect{f}(\vect{\sigma}^{n}\omega) + C,
      \end{equation}
we say that $\vect{f}$ is \textit{almost subadditive}.

  The following result was first obtained by Barreira \cite{Barreira2006} and Mummert \cite{Mummert2006} independently.
  \begin{prop}\label{prop:Palmadd}
    Let $\Omega \subseteq \Sigma_{0}^{\infty}$ be an autonomous mixing subshift of finite type.
    \begin{enumerate}[(1)]
      \item For all integral $n \geq 1$ and $\mathbf{u} \in \Omega_{0}^{n-1}$, let $\omega^{\mathbf{u}} \in [\mathbf{u}]$.
      Suppose that $\vect{f} \in \vect{C}(\vect{\Sigma}(\vect{m}),\R)$ is equicontinuous
      and  almost subadditive. 
      Then the limit
      \begin{equation}\label{eq:POmega}
        P(\vect{\sigma},\vect{f},\Omega)=\lim_{n \to \infty}\frac{1}{n}\log{\sum_{\mathbf{u} \in \Omega_{0}^{n-1}}\exp{(S_{n}^{\vect{\sigma}}\vect{f}(\omega^{\mathbf{u}}))}}
      \end{equation}
      exists and does not depend on the $\omega^{\mathbf{u}}$ chosen.
      Furthermore,
      $$
      \prebow(\vect{\sigma},\vect{f},\Omega)=\prepac(\vect{\sigma},\vect{f},\Omega)=\preL(\vect{\sigma},\vect{f},\Omega)=\preU(\vect{\sigma},\vect{f},\Omega)=P(\vect{\sigma},\vect{f},\Omega),
      $$
      and there exists a $\sigma$-invariant Borel probability   $\mu$ supported on $\Omega$ such that
      $$
      \prelow_{\mu}(\vect{\sigma},\vect{f})=\preup_{\mu}(\vect{\sigma},\vect{f})= h_{\mu}(\vect{\sigma}\vert_{\Omega}) + \lim_{n \to \infty}\frac{1}{n}\int_{\Sigma_{0}^{\infty}}S_{n}^{\vect{\sigma}}\vect{f}\dif{\mu} = P(\vect{\sigma},\vect{f},\Omega).
      $$
      \item If $\vect{f}$ satisfies \eqref{eq:Snfalmadd} and there is  $b>0$ such that for all $n>0$ and $\mathbf{u} \in \Sigma_{0}^{n-1}$,
      \begin{equation}\label{eq:bndvar}
        \lvert S_{n}^{\vect{\sigma}}\vect{f}(\omega)-S_{n}^{\vect{\sigma}}\vect{f}(\vartheta) \rvert \leq b
      \end{equation}
      whenever $\omega,\vartheta \in [\mathbf{u}]$,
      then there is a $P$-Gibbs state for $\vect{f}$ on $\Omega$.
    \end{enumerate}
  \end{prop}

\begin{rmk}\label{rmk:bndvar}
    The condition \eqref{eq:bndvar} is known as the \emph{bounded variation} property.     Equi-H\"{o}lder continuous potentials $\vect{f} \in \vect{C}(\vect{\Sigma}(\vect{m}),\R)$ are with bounded variation.   Potentials with bounded variation are clearly equicontinuous.
  \end{rmk}

 If  the potentials are almost subadditive,  we have the following results which are direct consequence of Propositions~\ref{prop:eqlbrNABmsr} and \ref{prop:eqlbrrestr}.
  \begin{thm}
    Suppose that $\vect{f}$ is equicontinuous and almost subadditive on some autonomous mixing subshift $\Omega$ of finite type.
    Let $\mu$ be an equilibrium state given in Proposition~\ref{prop:Palmadd}.
    Then $\prebow(\vect{\sigma},\vect{f},\Theta)=\prepac(\vect{\sigma},\vect{f},\Theta)$,
    and $\eqlbrbow_{\vect{f}}(\Theta) \neq \varnothing$ and $\eqlbrpac_{\vect{f}}(\Theta) \neq \varnothing$
    for all non-empty compact $\Theta \subseteq \Omega$ with $\mu(\Theta)>0$.
  \end{thm}
  \begin{cor}\label{cor_VPS}
    Suppose that $\vect{f}$ is equicontinuous and almost subadditive on some autonomous subshift $\Omega$ of finite type.
    Let $\mu$ be an equilibrium state given in Proposition~\ref{prop:Palmadd}.
    Then for all non-empty compact $\Theta \subseteq \Omega$ with $\mu(\Theta)>0$,
\begin{equation}
    \prepac(\vect{\sigma},\vect{f},\Omega) = \sup\{\preup_\mu(\vect{\sigma},\vect{f}): \mu \in M(\Sigma_{0}^{\infty})\ \text{and}\ \mu(\Omega)=1\}.
  \end{equation}

  \end{cor}

\section{Expansive Systems and Their Generators}\label{sect:exps}
    In this section, we discuss the pressures in the so called expansive systems.  
      \subsection{Expansiveness and generators}\label{ssect:expsndsetgnrtr}

      In an expansive NDS $(\vect{X},\vect{T})$ with an expansive constant $\delta>0$, since $d_{X_{j}}(\vect{T}^{j}x,\vect{T}^{j}y) \leq \delta$ for all $j \in \N$ implies $x=y$,  it is clear that  for every $\varepsilon$ with $0<\varepsilon \leq \delta$,
      $$
      \lim_{n \to \infty}B_{n}^{\vect{T}}(x,\varepsilon)=\lim_{n \to \infty}\overline{B}_{n}^{\vect{T}}(x,\varepsilon)=\{x\}
      $$
      for all $x \in X_{0}$. Then we have the following property of Bowen balls.
      \begin{prop}\label{prop:expsball}
        Let $(\vect{X},\vect{T})$ be expansive  with an expansive constant $\delta>0$.
        Then for every $\varepsilon$ with $0<\varepsilon<\delta$,
        $$
        \lim_{n \to \infty}\diam(B_{n}^{\vect{T}}(x,\varepsilon)) = \lim_{n \to \infty}\diam(\overline{B}_{n}^{\vect{T}}(x,\varepsilon)) = 0
        $$
        for all $x \in X_{0}$.
      \end{prop}

      This property also holds for the dynamical refinement of finite open covers.
      \begin{prop}\label{prop: A.4}
        Let $(\vect{X},\vect{T})$ be expansive  with an expansive constant  $\delta>0$.
        Let $\vect{\mathscr{A}}=\{\mathscr{A}_{k}\}_{k=0}^{\infty}$ be a sequence of finite (open) covers $\mathscr{C}^{k}$ of $X_{k}$
        with $\diam(\vect{\mathscr{A}}) \leq \delta$. Then
        $$
        \lim_{n\to\infty} \diam\left(\vee_{n}^{\vect{T}}\vect{\mathscr{A}}\right) =0.
        $$
      \end{prop}
      \begin{proof}
We prove it by contradiction.  Suppose that there exists $\varepsilon_{0}>0$ with a subsequence $\{V_{n_{i}}\}_{i=1}^{\infty}$ ($n_{i}$ is strictly increasing as $i \to \infty$) such that
        $\diam(V_{n_{i}})>\varepsilon_{0}$, where $V_{n_{i}}$ is a member of $\vee_{n_{i}}^{\vect{T}}\vect{\mathscr{A}}$ for each $i \in \mathbb{N}$,
        i.e., $V_{n_{i}}$ is of the form $\bigcap_{j=0}^{n_{i}-1}\vect{T}^{-j}A_{j,i}$ for some members $A_{j,i}$ in $\mathscr{A}_{j}$.
        This implies that there exists $x_{i},y_{i} \in V_{n_{i}}$ such that
        $d_{X_{0}}(x_{i},y_{i})>\varepsilon_{0}$ for each $i \in \mathbb{N}$.
        By the compactness of $X_{0}$, we assume that $x_{i} \to x$ and $y_{i} \to y$ as $i \to \infty$.
        It follows that $d_{X_{0}}(x,y) \geq \varepsilon_{0}$ and $x \neq y$.

        For each $j \in \mathbb{N}$, write $$i_{j}=\min\{i \in \mathbb{N}: n_{i}-1 \geq j\}.$$
        Fix $j$. Consider the infinite sequence $\{A_{j,i}\}_{i=i_{j}}^{\infty} \subseteq \mathscr{A}_{j}$.
        Since $\mathscr{A}_{j}$ is finite, infinitely many of the sets $A_{j,i}$ coincide, and $\{A_{j,i}\}_{i=i_{j}}^{\infty}$ may be decomposed into a finite number of constant subsequences.
        It follows that $\{\vect{T}^{-j}A_{j,i}\}_{i=i_{j}}^{\infty}$, as a set, is finite.
        Recall that for each $i \geq i_{j}$, the two points $x_{i}$ and $y_{i}$ are contained in the same set $\vect{T}^{-j}A_{j,i}$.
        Thus, there exist some $\vect{T}^{-j}A_{j,i}$ containing infinitely many of the points $x_{i}$'s and infinitely many of the points $y_{i}$'s.
        Choose $A_{j,l_{j}} \in \mathscr{A}_{j}$ from $\{A_{j,l_{j}}\}_{i=i_{j}}^{\infty}$ so that $x_{i},y_{i} \in \vect{T}^{-j}A_{j,l_{j}}$
        for infinitely many $i$'s. It implies  that $x,y \in \overline{\vect{T}^{-j}A_{j,l_{j}}}=\vect{T}^{-j}\overline{A_{j,l_{j}}}$, and we have that 
        $$d_{X_{j}}(\vect{T}^{j}x,\vect{T}^{j}y) \leq \diam(\overline{A_{j,l_{j}}}) \leq \diam(\vect{\mathscr{A}}) \leq \delta$$
 for all $j \in \mathbb{N}$. It follows that  $x=y$, which  contradicts $d(x,y) \geq \varepsilon_{0}$.
      \end{proof}

      Similar to the case in TDSs (see \cite[Rmk.2.10]{Keynes&Robertson1969}; see also \cite[Thm.5.21]{Walters1982}), a generator (if it exits) determines the topology on $X_{0}$.
      \begin{prop}\label{prop:expstopbase}
        Given   $(\vect{X},\vect{T})$. Let  $\vect{\mathscr{U}}$ be a generator for $\vect{T}$.
        Then $\bigcup_{n=0}^{\infty}(\vee_{n}^{\vect{T}}\vect{\mathscr{U}})$ is a base for the topology of $X_{0}$.
      \end{prop}
      \begin{proof}
        Recall that every $\vee_{n}^{\vect{T}}\vect{\mathscr{U}}$ is an open cover of $X_{0}$.
        It suffices to show that for every $\varepsilon>0$, there exists $N>0$ such that $\diam(\vee_{N}^{\vect{T}}\vect{\mathscr{U}}) \leq \varepsilon$.

        Suppose otherwise that there exists $\varepsilon_{0}>0$ such that for all $n>0$, there is some member $V_{n}$ of $\vee_{n}^{\vect{T}}\vect{\mathscr{U}}$ with $\diam(V_{n}) > \varepsilon_{0}$.
        Write $V_{n}=\bigcap_{j=0}^{n-1}\vect{T}^{-j}U_{j,n}$ where $U_{j,n} \in \mathscr{U}_{j}$ for each $j$.
        It follows that for every $n>0$, there exist two points $x_{n},y_{n} \in \bigcap_{j=0}^{n-1}\vect{T}^{-j}U_{j,n}$ such that $d_{X_{0}}(x_{n},y_{n}) >\varepsilon_{0}$.
        Since $X_{0}$ is compact, without loss of generality, we  assume that  $x_{n} \to x$ and $y_{n} \to y$ as $n \to \infty$. We have that $d_{X_{0}}(x,y) \geq \varepsilon_{0}$ and that $x \neq y$.

        On the other hand, consider $\{U_{j,n}\}_{n=j+1}^{\infty} \subseteq \mathscr{U}_{j}$ for every fixed $j$.
        Since $\mathscr{U}_{j}$ is finite, $\{U_{j,n}\}_{n=j+1}^{\infty}$ may be decomposed into a finite number of constant subsequences.
        Thus, infinitely many of the $x_{n}$'s and $y_{n}$'s are contained in $\vect{T}^{-j}U_{j,n_{j}}$ where $U_{j,n_{j}}$ is chosen from $\{U_{j,n}\}_{n=j+1}^{\infty} \subseteq \mathscr{U}_{j}$.
        It follows that $x$ and $y$ are both contained in $\overline{\vect{T}^{-j}U_{j,n_{j}}} = \vect{T}^{-j}\overline{U_{j,n_{j}}}$ for every $j$.
        This implies that
        $$
        x,y \in \bigcap_{j=0}^{\infty}\vect{T}^{-j}U_{j,n_{j}},
        $$
        and since $\vect{\mathscr{U}}$ is a generator, it follows that $x=y$, leading to a contradiction.
      \end{proof}

     \begin{proof}[Proof of Theorem \ref{prop:expseqgnrtr}]
        We prove the equivalence of the three statements by showing (2) $\implies$ (3) $\implies$ (1) $\implies$ (2).
It is clear that  (2) $\implies$ (3).

Next, we show (3) $\implies$ (1).        Suppose that $\vect{\mathscr{V}}=\{\mathscr{V}_{k}\}_{k=0}^{\infty}$ is a weak generator with a Lebesgue number $\delta>0$, where $\mathscr{V}_{k}=\{V_{i}^{(k)}\}_{i=1}^{m_{k}}$ for each $k$. 
        Let $x,y \in X_{0}$. If $d_{X_{j}}(\vect{T}^{j}x,\vect{T}^{j}y) \leq \delta$ for all $j \in \N$,
        then for each $j \in \N$, there exists a member $V_{i_{j}}^{(j)}$ of $\mathscr{V}_{j}$ with $\{\vect{T}^{j}x,\vect{T}^{j}y\} \subseteq V_{i_{j}}^{(j)}$,
        and it implies that 
        $$
        \{x,y\} \subseteq \bigcap_{j=0}^{\infty}\vect{T}^{-j}V_{i_{j}}^{(j)}.
        $$
        Since the intersection on the RHS contains at most one point, $x=y$.
        This implies that $\vect{T}$ is expansive with $\delta$ as an expansive constant.

        Finally, we show (1) $\implies$ (2).
        Let $\delta>0$ be an expansive constant for $\vect{T}$.
        Fix $\varepsilon$ with $0< \varepsilon <\frac{\delta}{4}$,  and  by the compactness of $X_{k}$,  for every $k \geq 0$, we choose a finite set $\{x_{1}^{(k)},\ldots,x_{m_{k}}^{(k)}\} \subseteq X_{k}$
        such that
        $$
        X_{k} = \bigcup_{i=1}^{m_{k}}B\Bigl(x_{i}^{(k)},\frac{\delta}{2}-\varepsilon\Bigr).
        $$
Hence,  for all $k \geq 0$,  $\varepsilon$ is a Lebesgue number for the  cover $\mathscr{B}_{k}=\{B(x_{i}^{(k)},\frac{\delta}{2})\}_{i=1}^{m_{k}}$ of $X_{k}$,
        and $\varepsilon$ is a Lebesgue number  for $\vect{\mathscr{B}}=\{\mathscr{B}_{k}\}_{k=0}^{\infty}$.
        
        Suppose that $x,y \in \bigcap_{j=0}^{\infty}\vect{T}^{-j}\overline{B_{i_{j}}^{(j)}}$ where $B_{i_{j}}^{(j)} \in \mathscr{B}_{j}$ for each $j$.
        Since every $B_{i_{j}}^{(j)}$ is a ball with radius $\frac{\delta}{2}$, we have $d_{X_{j}}(\vect{T}^{j}x,\vect{T}^{j}y) \leq \delta$ for all $j$. By the expansiveness of $\vect{T}$, we have  $x=y$.  Therefore, $\bigcap_{j=0}^{\infty}\vect{T}^{-j}\overline{B_{i_{j}}^{(j)}}$ contains at most one point for all sequences $\{B_{i_{j}}^{(j)} \in \mathscr{B}_{j}\}_{j=0}^{\infty}$,
        and $\vect{\mathscr{B}}$ is a generator for $\vect{T}$.
      \end{proof}

      \begin{proof}[Proof of Proposition \ref{prop_eque}]
      It  directly follows from Propositions~\ref{prop:expsball}, \ref{prop: A.4}, \ref{prop:expstopbase} and Theorem \ref{prop:expseqgnrtr}.
\end{proof}

      \begin{prop}
        Suppose that $(\vect{X},\vect{T})$ is   uniformly expansive with a uniform expansive constant $\delta>0$. Then the following hold.
        \begin{enumerate}[(1)]
          \item For every $\varepsilon$ with $0<\varepsilon<\delta$ and for all $k \in \N$, 
          $$
          \lim_{n \to \infty}\diam(B_{k,n}^{\vect{T}}(x,\varepsilon))=\lim_{n \to \infty}\diam(\overline{B}_{k,n}^{\vect{T}}(x,\varepsilon))=0,
          $$ 
          for all $x \in X_{k}$.
          \item For all sequences $\vect{\mathscr{A}}=\{\mathscr{A}_{k}\}_{k=0}^{\infty}$ of finite (open) covers $\mathscr{A}_{k}$ of $X_{k}$ with $\diam(\vect{\mathscr{A}}) \leq \delta$,
                $\lim_{n\to\infty} \diam(\vee_{k,n}^{\vect{T}}\vect{\mathscr{A}}) = 0$  for all $k \in \N$.
          \item Given a uniform (weak) generator $\vect{\mathscr{U}}$ for $\vect{T}$, 
                $\lim_{n\to\infty}  \diam(\vee_{k,n}^{\vect{T}}\vect{\mathscr{A}}) = 0$  for all $k \in \N$.
        \end{enumerate}
      \end{prop}

      It may happen that the convergence to $0$ of the diameters in the above proposition is not uniform in $k$ (see \cite[Exmp.7.13]{Kawan2015}),
      and we require stronger conditions for the generators to recover its generating property for pressures.

      \subsection{Pressures in strongly uniformly expansive systems}\label{ssect:suepre}
      In this subsection, we recover the generating property of generators for topological entropies
      and simplify the calculation and formulation for the pressures and entropies in  strongly uniformly expansive systems.

Let $(\vect{X},\vect{T})$ be strongly uniformly expansive with   sue constant $\delta>0$. Recall  Definition \ref{def:sue}, and it is equivalent to that, for every $\varepsilon>0$, there is an integer $N \geq 1$ such that
      for all $k \in \N$ and $x \in X_{k}$,
      $$
      B_{k,N}^{\vect{T}}(x,\delta) \subseteq B_{k}(x,\varepsilon).
      $$
This implies that the Bowen balls shrink to atoms uniformly in $k$. Hence every  sue NDS $(\vect{X},\vect{T})$  with sue constant $\delta>0$ is uniformly expansive with $\delta$ as a uniform expansive constant (see Definition \ref{def:uniexps}).
      The following conclusion is a direct consequence of  Definition \ref{def:sue}. 
      \begin{lem}\label{lem: A.5}
        Suppose that $(\vect{X},\vect{T})$ is  strongly uniformly expansive with  a  uniform expansive constant $\delta>0$. Then the following hold.
        \begin{enumerate}[(1)]
          \item Given a sequence $\vect{\mathscr{A}}=\{\mathscr{A}_{k}\}_{k=0}^{\infty}$ of finite (open) covers $\mathscr{A}_{k}$ of $X_{k}$ with $\diam(\vect{\mathscr{A}}) \leq \delta$,
                for every $\varepsilon>0$, there is an integer $N \geq 1$ such that 
                $\diam(\vee_{k,N}^{\vect{T}}\vect{\mathscr{A}}) \leq \varepsilon$ for all $k \in \N$.
          \item $\vect{T}$ has a uniform (weak) generator $\vect{\mathscr{U}}$ such that for every $\varepsilon>0$, there exists integral $N>0$ with the property that $\diam(\vee_{k,N}^{\vect{T}}\vect{\mathscr{U}}) \leq \varepsilon$ for all $k \in \N$.
        \end{enumerate}
      \end{lem}

      We require the following property of pressures for the proof of Theorem~\ref{thm:suepre}.
      \begin{prop}\label{prop:qrecovpred2}
        Given a sequence $\vect{\mathscr{U}}$ of finite open covers $\mathscr{U}_{k}$ of $X_{k}$ with Lebesgue number $\delta>0$,
        we have
        $$
        Q(\vect{T},\vect{f},Z,\vect{\mathscr{U}}) \leq P(\vect{T},\vect{f},Z,\frac{\delta}{2})
        $$
        for $Q \in \{\qrelow,\qreup,\prebpp\}$ and the corresponding $P \in \{\prelow,\preup,\prebow\}$.
      \end{prop}
      \begin{proof}
        We only give the proof for $\prebpp(\vect{T},\vect{f},Z,\vect{\mathscr{U}}) \leq \prebow(\vect{T},\vect{f},Z,\frac{\delta}{2})$, and others follows  by the similar arguments.

        Given an integer $N>0$, for every  $(N,\frac{\delta}{2})$-cover $\{B_{n_{i}}^{\vect{T}}(x_{i},\varepsilon)\}_{i=1}^{\infty}$ of $Z$,
        we have
      $$
      Z \subset \bigcup_{i=1}^{\infty}B_{n_{i}}^{\vect{T}}\Big(x,\frac{\delta}{2}\Big) = \bigcup_{i=1}^{\infty}\bigcap_{j=0}^{n-1}\vect{T}^{-j}B_{X_{j}}\Big(\vect{T}^{j}x_{i},\frac{\delta}{2}\Big).
      $$
      For each $i$, since $B_{X_{j}}\big(\vect{T}^{j}x_{i},\frac{\delta}{2}\big)$ is contained in a member of $\mathscr{U}_{j}$ for every $j>0$,
      there exists $\mathbf{U}_{i} \in \vect{\mathscr{U}}_{0}^{n_{i}}$ such that  $B_{n_{i}}^{\vect{T}}\big(x_{i},\frac{\delta}{2}\big) \subset X_{0}[\mathbf{U}_{i}]$.
      Hence $\{\mathbf{U}_{i}\}_{i=1}^{\infty} \subset \vect{\mathscr{U}}_{0}^{n}$ covers $Z$,
      and it follows by \eqref{eq:defmsrbpp} and \eqref{eq:SknTfU} that for all $s \in \R$,
      $$
      \msrbpplow_{N}^{s}(\vect{T},\vect{f},Z,\vect{\mathscr{U}}) \leq \sum_{i=1}^{\infty}\exp{\bigl(-n_{i}s + \underline{S}_{n_{i}}^{\vect{T}}\vect{f}(\mathbf{U}_{i})\bigr)} \leq \sum_{i=1}^{\infty}\exp{\bigl(-n_{i}s + S_{n_{i}}^{\vect{T}}\vect{f}(x_{i})\bigr)},
      $$
      which implies $$\msrbpplow_{N}^{s}(\vect{T},\vect{f},Z,\vect{\mathscr{U}}) \leq \msrbow_{N,\frac{\delta}{2}}^{s}(\vect{T},\vect{f},Z)$$ by the arbitrariness of the $(N,\frac{\delta}{2})$-cover $\{B_{n_{i}}^{\vect{T}}(x_{i},\frac{\delta}{2})\}_{i=1}^{\infty}$.
      The conclusion  immediately follows by \eqref{def_PBep} and \eqref{def_PBPPep}.
      \end{proof}

      \begin{proof}[Proof of Theorem~\ref{thm:suepre}]
        (1) Let $\vect{\mathscr{U}}$ be a uniform (weak) generator for $\vect{T}$.
        Write
        $$
        \mathscr{V}_{k,m}=\bigvee_{j=0}^{m-1}\vect{T}_{k}^{-j}\mathscr{U}_{k+j}
        $$
        for every $k \geq 0$,  and let $\vect{\mathscr{V}}_{m}=\{\mathscr{V}_{k,m}\}_{k=0}^{\infty}$ for each $m \geq 1$.
        By Lemma~\ref{lem: A.5}(2), we have that for all $k \geq 0$, $\diam(\mathscr{V}_{k,m}) \to 0$ as $m \to \infty$,
        and thus
        $$
        \diam(\vect{\mathscr{V}}_{m}) \to 0 \quad \text{as}\ m \to \infty.
        $$     
        It follows by Propositions~\ref{prop:qrecov} that
        \begin{equation}\label{eq:prelimexpans}
          P(\vect{T},\vect{f},Z) = \lim_{m \to \infty}Q(\vect{T},\vect{f},Z,\vect{\mathscr{V}}_{m}) = \lim_{m \to \infty}P(\vect{T},\vect{f},Z,\vect{\mathscr{V}}_{m}),
        \end{equation}
        where $P \in \{\prelow,\preup\}$.
        Similarly, by Proposition~\ref{prop:prebppcov}, we have
        \begin{equation}\label{eq:PBlimexpans}
          \prebow(\vect{T},\vect{f},Z) = \lim_{m \to \infty}\prebpp(\vect{T},\vect{f},Z,\vect{\mathscr{V}}_{m}) = \lim_{m \to \infty}\prebow(\vect{T},\vect{f},Z,\vect{\mathscr{V}}_{m}),
        \end{equation}

        We prove $ \prelow(\vect{T},\vect{f},Z) = \qrelow(\vect{T},\vect{f},Z,\vect{\mathscr{U}})$ and $ \preup(\vect{T},\vect{f},Z) = \qreup(\vect{T},\vect{f},Z,\vect{\mathscr{U}})$ first.

        For every $k \geq 0$, $m \geq 1$, and every given $V_{k} \in \mathscr{V}_{k,m}$, write $V_{k}=\bigcap_{j=0}^{m-1}\vect{T}_{k}^{-j}U_{k+j}^{(k)}$, where $U_{k+j}^{(k)} \in \mathscr{U}_{k+j}$ for each $0 \leq j \leq m-1$.
        It follows that for every $V_{0} \ldots V_{n-1} \in (\vect{\mathscr{V}}_{m})_{0}^{n}$,
        $$
        \bigcap_{j=0}^{n-1}\vect{T}^{-j}V_{j} \subseteq \Big(\bigcap_{j=0}^{n-1}\vect{T}^{-j}U_{j}^{(j)}\Big) \cap \Big(\bigcap_{j=n}^{n+m-1}\vect{T}^{-j}U_{j}^{(n-1)}\Big)
        $$
        For every $n \geq 1$, we define a mapping $\phi_{n,m}:(\vect{\mathscr{V}}_{m})_{0}^{n} \to \vect{\mathscr{U}}_{0}^{n+m}$ by
        $$
        \phi_{n,m}(V_{0} \ldots V_{n-1}) = U_{0}^{(0)} U_{1}^{(1)}\ldots U_{n-1}^{(n-1)} U_{n}^{(n-1)}\ldots U_{n+m-1}^{(n-1)}.
        $$
        Clearly, we have $X_{0}[\mathbf{V}] \subseteq X_{0}[\phi_{n,m}(\mathbf{V})]$ for all $\mathbf{V} \in (\vect{\mathscr{V}}_{m})_{0}^{n}$.
        Hence, if $\Gamma \subseteq (\vect{\mathscr{V}}_{m})_{0}^{n}$ covers $Z$, then $\phi_{n,m}(\Gamma) \subseteq \vect{\mathscr{U}}_{0}^{n+m}$ also covers $Z$.
        It follows by \eqref{eq:SknTfU} that
        $$
        \sum_{\mathbf{U} \in \phi_{n,m}(\Gamma)}\exp{\bigl(\underline{S}_{n+m}^{\vect{T}}\vect{f}(\mathbf{U})\bigr)} = \sum_{\mathbf{V} \in \Gamma}\exp{\bigl(\underline{S}_{n+m}^{\vect{T}}\vect{f}(\phi_{n,m}\mathbf{V})\bigr)} \leq \sum_{\mathbf{V} \in \Gamma}\exp{\bigl(\overline{S}_{n}^{\vect{T}}\vect{f}(\mathbf{V}) + m\|\vect{f}\|\bigr)},
        $$
  and it  implies by \eqref{eq:QnPncov} that
        $$
        Q_{n+m}(\vect{T},\vect{f},Z,\vect{\mathscr{U}}) \leq \e^{m\|\vect{f}\|}P_{n}(\vect{T},\vect{f},Z,\vect{\mathscr{V}}_{m}).
        $$

On the other hand, for every $n,m \geq 1$, we define a mapping $\gamma_{n,m}: \vect{\mathscr{U}}_{0}^{n+m} \to (\vect{\mathscr{V}}_{m})_{0}^{n}$ by
        $$
        \gamma_{n,m}(U_{0} \ldots U_{n+m-1}) = V_{0} \ldots V_{n-1},
        $$
        where $V_{k} = \bigcap_{j=0}^{m-1}\vect{T}_{k}^{-j}U_{k+j}$ for each $0 \leq k \leq n-1$.
        Obviously, we have $X_{0}[\mathbf{U}]=X_{0}[\gamma_{n,m}(\mathbf{U})]$ for all $\mathbf{U} \in \vect{\mathscr{U}}_{0}^{n+m}$.
        Therefore, if $\Phi \subseteq \vect{\mathscr{U}}_{0}^{n+m}$ covers $Z$, then so does $\gamma_{n,m}(\Phi) \subseteq (\vect{\mathscr{V}}_{m})_{0}^{n}$.
        Meanwhile, by \eqref{eq:SknTfU}, we have
        \begin{align*}
        \sum_{\mathbf{V} \in \gamma_{n,m}(\Phi)}\exp{\bigl(\underline{S}_{n}^{\vect{T}}\vect{f}(\mathbf{V}) - m\|\vect{f}\|\bigr)}
        &= \sum_{\mathbf{U} \in \Phi}\exp{\bigl(\underline{S}_{n}^{\vect{T}}\vect{f}(\gamma_{n,m}\mathbf{U}) - m\|\vect{f}\|\bigr)} \\
        &\leq \sum_{\mathbf{U} \in \Phi}\exp{\bigl(\underline{S}_{n+m}^{\vect{T}}\vect{f}(\mathbf{U})\bigr)},
        \end{align*}
and by \eqref{eq:QnPncov}, 
        $$
        \e^{-m\|\vect{f}\|}Q_{n}(\vect{T},\vect{f},Z,\vect{\mathscr{V}}_{m}) \leq Q_{n+m}(\vect{T},\vect{f},\vect{\mathscr{U}}).
        $$

        Combined with \eqref{def_QQPPCV}, these imply that for all $m \geq 1$,
        $$
        Q(\vect{T},\vect{f},Z,\vect{\mathscr{V}}_{m}) \leq Q(\vect{T},\vect{f},Z,\vect{\mathscr{U}}) \leq P(\vect{T},\vect{f},Z,\vect{\mathscr{V}}_{m}),
        $$
        where $Q \in \{\qrelow,\qreup\}$ and $P$ is the corresponding one of $\{\prelow,\preup\}$.
By \eqref{eq:prelimexpans},  we obtain
$$ 
\prelow(\vect{T},\vect{f},Z) = \qrelow(\vect{T},\vect{f},Z,\vect{\mathscr{U}}) \qquad  \textit{and} \qquad \preup(\vect{T},\vect{f},Z) = \qreup(\vect{T},\vect{f},Z,\vect{\mathscr{U}}).
$$

To prove  $ \prebow(\vect{T},\vect{f},Z) = \prebpp(\vect{T},\vect{f},Z,\vect{\mathscr{U}})$, we define $\phi_{*,m}: \bigcup_{n=1}^{\infty}(\vect{\mathscr{V}}_{m})_{0}^{n} \to \bigcup_{n=1}^{\infty}\vect{\mathscr{U}}_{0}^{n+m}$
        and $\gamma_{*,m}: \bigcup_{n=1}^{\infty}\vect{\mathscr{U}}_{0}^{n+m} \to \bigcup_{n=1}^{\infty}(\vect{\mathscr{V}}_{m})_{0}^{n}$
        respectively by
        $\phi_{*,m}(\mathbf{V})=\phi_{\len{\mathbf{V}},m}(\mathbf{V})$ for all $\mathbf{V} \in \bigcup_{n=1}^{\infty}(\vect{\mathscr{V}}_{m})_{0}^{n}$
        and $\gamma_{*,m}(\mathbf{U})=\gamma_{\len{\mathbf{U}},m}(\mathbf{U})$ for all $\mathbf{U} \in \bigcup_{n=1}^{\infty}\vect{\mathscr{U}}_{0}^{n+m}$.
        Similarly, if $\Gamma \subseteq \bigcup_{n=N}^{\infty}(\vect{\mathscr{V}}_{m})_{0}^{n}$ covers $Z$, then $\phi_{*,m}(\Gamma) \subseteq \bigcup_{n=N}^{\infty}\vect{\mathscr{U}}_{0}^{n+m}$ also covers $Z$;
        and if $\Phi \subseteq \bigcup_{n=N}^{\infty}\vect{\mathscr{U}}_{0}^{n+m}$ covers $Z$, then so does $\gamma_{*,m}(\Phi) \subseteq \bigcup_{n=N}^{\infty}(\vect{\mathscr{V}}_{m})_{0}^{n}$.
        Therefore,
        $$
        \e^{-m\|\vect{f}\|}\msrbpplow_{N}^{s}(\vect{T},\vect{f},Z,\vect{\mathscr{V}}_{m}) \leq \msrbpplow_{N+m}^{s}(\vect{T},\vect{f},Z,\vect{\mathscr{U}}) \leq \e^{m\|\vect{f}\|}\msrbppup_{N}^{s}(\vect{T},\vect{f},Z,\vect{\mathscr{V}}_{m})
        $$
        which implies that for all $m \geq 1$,
        $$
        \prebpp(\vect{T},\vect{f},Z,\vect{\mathscr{V}}_{m}) \leq \prebpp(\vect{T},\vect{f},Z,\vect{\mathscr{U}})  \leq  \prebow(\vect{T},\vect{f},Z,\vect{\mathscr{V}}_{m}),
        $$
        and by \eqref{eq:PBlimexpans}, the equality   $ \prebow(\vect{T},\vect{f},Z) = \prebpp(\vect{T},\vect{f},Z,\vect{\mathscr{U}})$ holds.

        (2) First, we prove $ P(\vect{T},\vect{f},Z) = P(\vect{T},\vect{f},Z,\varepsilon)$ for  $P \in \{\prelow,\preup\}$. 

 Fix $\varepsilon$ with $0< \varepsilon <\frac{\delta}{4}$. For every $k \geq 0$,  by the compactness of $X_{k}$, we choose a finite set $\{x_{1}^{(k)},\ldots,x_{m_{k}}^{(k)}\} \subseteq X_{k}$
        such that
        $$
        X_{k} = \bigcup_{i=1}^{m_{k}}B\Bigl(x_{i},\frac{\delta}{2}-\varepsilon\Bigr),
        $$
and    $2\varepsilon$ is a Lebesgue number for  the finite open cover $\mathscr{B}_{k}=\{B(x_{i},\frac{\delta}{2})\}_{i=1}^{m_{k}}$ of $X_{k}$. 
Hence  the sequence $\vect{\mathscr{B}}=\{\mathscr{B}_{k}\}_{k=0}^{\infty}$ has $2\varepsilon$ as a Lebesgue number.
        By Proposition~\ref{prop:qrecovpred2}, we have
        $$
        Q_{n}(\vect{T},\vect{f},Z,\vect{\mathscr{B}}) \leq P_{n}(\vect{T},\vect{f},Z,\varepsilon).
        $$
Since $P_{n}$ is decreasing in $\varepsilon$, it follows by \eqref{def_QQPPCV} and \eqref{eq:defpre} that 
        $$
        Q(\vect{T},\vect{f},Z,\vect{\mathscr{B}}) \leq P(\vect{T},\vect{f},Z,\varepsilon) \leq P(\vect{T},\vect{f},Z),
        $$
        where $P\in\{\prelow,\preup\}$ and the corresponding $Q\in \{\qrelow,\qreup\}$.
Since $\vect{\mathscr{B}}$ is a uniform generator,by the conclusion (1) above, we obtain that $P(\vect{T},\vect{f},Z,\varepsilon) = P(\vect{T},\vect{f},Z),$ for $P\in\{\prelow,\preup\}$.

To prove $ \prepac(\vect{T},\vect{f},Z) = \prepac(\vect{T},\vect{f},Z,\varepsilon)$,   we only need to  combine $ \preup(\vect{T},\vect{f},Z) = \preup(\vect{T},\vect{f},Z,\varepsilon)$ with the argument in the proof of \cite[Thm.4.8]{CM1},  and the equality holds. 

For $ \prebow(\vect{T},\vect{f},Z) = \prebow(\vect{T},\vect{f},Z,\varepsilon)$,  by Propositions~\ref{prop:qrecovpred2} and the monotonicity of $\prebow$ in $\varepsilon$, we have the similar inequalities
        $$
        \prebpp(\vect{T},\vect{f},Z,\vect{\mathscr{B}}) \leq \prebow(\vect{T},\vect{f},Z,\varepsilon) \leq \prebow(\vect{T},\vect{f},Z),
        $$
        and the equality  follows from the conclusion (1) above.
      \end{proof}

      \subsection{Symbolic dynamics of strongly uniformly expansive systems}\label{ssect:suesymdyn}
      In this subsection, we  give the  proof of  Theorem~\ref{thm:suesymext}.
      \begin{proof}[Proof of Theorem~\ref{thm:suesymext}]
Since $(\vect{X},\vect{T})$ is strongly uniformly expansive, we write $\delta>0$  for  a uniform expansive constant of  $\vect{T}$.
        By the proof of Proposition~\ref{prop:expseqgnrtr}, there exists a generator $\vect{\mathscr{B}}$ by finite covers $\mathscr{B}_{k}$ of $X_{k}$ by balls of radius $\frac{\delta}{3}$.
        Write $m_{k}=\#\mathscr{B}_{k}$ and $\mathscr{B}_{k}=\{B_{i}^{(k)}\}_{i=1}^{m_{k}}$ for each $k \in \N$.
        Fix $k$. Let
        $$
        F_{l}^{(k)}=
        \left\{
          \begin{array}{ll}
            \overline{B_{1}^{(k)}}, & l=1, \\
            \overline{B_{l}^{(k)}} \setminus \bigcup_{i=1}^{l}B_{i}^{(k)}, & 1 < l \leq m_{k}. 
          \end{array}
        \right.
        $$
        Note that $$\intr(F_{l}^{(k)})=B_{l}^{(k)} \setminus \overline{\bigcup_{i=1}^{l}B_{i}^{(k)}}$$ for all $l>1$,
        and that $$\partial{F_{j}^{(k)}} \cap \intr(F_{l}^{(k)}) \subseteq B_{l}^{(k)} \setminus \bigcup_{i=1}^{j}B_{i}^{(k)} = \varnothing$$ for all $l<j$.
        Thus, $\mathscr{F}_{k}=\{F_{i}^{(k)}\}_{i=1}^{m_{k}}$ is a cover of $X_{k}$ satisfying the following properties:
        \begin{enumerate}[(a)]
          \item $F_{i}^{(k)} \cap F_{j}^{(k)} = \partial{F_{i}^{(k)}} \cap \partial{F_{j}^{(k)}}$ for all $i \neq j$;
          \item $\bigcup_{i=1}^{m_{k}}\partial{F_{i}^{(k)}} \subseteq \bigcup_{i=1}^{m_{k}}\partial{B_{i}^{(k)}}$ has empty interior.
        \end{enumerate}

        Let $D^{(k)}=\bigcup_{i=1}^{m_{k}}\partial{F_{i}^{(k)}}$ and $D_{k}^{\infty}=\bigcup_{j=k}^{\infty}\vect{T}^{-j}D^{(j)}$.
        It is clear that $D_{k}^{\infty}$ is of first category, and so $X_{k} \setminus D_{k}^{\infty}$ is dense in $X_{k}$.
        By (a), $\widehat{\mathscr{F}}_{j}=\{F_{i}^{(j)} \cap (X_{j} \setminus D_{j}^{\infty})\}_{i=1}^{m_{j}}$ is disjoint for every $j$. Hence for each $x \in X_{k} \setminus D_{k}^{\infty}$, there exists a unique $\omega \in \Sigma_{k}^{\infty}(\vect{m})$ such that
        $\vect{T}^{j}x \in F_{\omega_{j}}^{(j)}$ for all $j \geq k$, and this  defines a mapping $\psi_{k}:X_{k} \setminus D_{k}^{\infty} \to \Sigma_{k}^{\infty}(\vect{m})$.
Moreover, $\psi_{k}$ is injective since   $\vect{T}$ is  expansive.

        Write $\Omega_{k}=\psi_{k}(X_{k} \setminus D_{k}^{\infty})$. The restriction of  $\psi_{k}$ to  $\Omega_{k}$  has a surjective inverse
        $\psi_{k}^{-1}:\Omega_{k} \to X_{k} \setminus D_{k}^{\infty}$. We define  $\pi_{k}: \overline{\Omega_{k}} \to X_{k}$ by
        $$
        \pi_{k}(\omega)=
        \left\{
          \begin{array}{ll}
            \psi_{k}^{-1}(\omega), &\text{if}\ \omega \in \Omega_{k}, \\
            \lim_{\Omega_{k} \ni \vartheta \to \omega}\psi_{k}^{-1}(\vartheta), &\text{otherwise}.
          \end{array}
        \right.
        $$
        
        We first show that each $\pi_{k}$ is well defined ($\psi_{k}^{-1}$ is continuous) and that $\vect{\pi}=\{\pi_{k}\}_{k=0}^{\infty}$ is equicontinuous.
It suffices to show that for each $\varepsilon>0$, there exists an integer $N>0$ such that
        for all $k \in \N$ and $x,y \in X_{k} \setminus D_{k}^{\infty}$, $d_{X_{k}}(x,y)<\varepsilon$ whenever $(\psi_{k}(x))_{j+k}=(\psi_{k}(y))_{j+k}$ for all $0 \leq j \leq N-1$.

Arbitrarily choose  $\varepsilon>0$. For the sequence $\widehat{\vect{\mathscr{F}}}=\{\widehat{\mathscr{F}}_{k}\}_{k=0}^{\infty}$, by Lemma~\ref{lem: A.5}(1), there is an integer $N>0$ such that $\diam(\vee_{k,n}^{\vect{T}}\widehat{\vect{\mathscr{F}}})<\varepsilon$ for all $n \geq N$.
For each $k \in \N$,  arbitrarily choose $x,y \in X_{k} \setminus D_{k}^{\infty}$ satisfying  $(\psi_{k}(x))_{j+k}=(\psi_{k}(y))_{j+k}$ for all $0 \leq j \leq N-1$, and it is clear that  $x$ and $y$  are contained in the same member of $\vee_{k,n}^{\vect{T}}\widehat{\vect{\mathscr{F}}}$. Hence  we have that $d_{X_{k}}(x,y) \leq \diam(\vee_{k,n}^{\vect{T}}\widehat{\vect{\mathscr{F}}}) <\varepsilon$.

        Since $\psi_{k+1} \circ T_{k} = \sigma_{k+1} \circ \psi_{k}$, we have $\pi_{k+1} \circ \sigma_{k} = T_{k+1} \circ \pi_{k}$ for all $k$.
       \end{proof}

\end{document}